\documentclass[reqno, 11pt, letterpaper]{amsart}

\usepackage{amsmath}
\usepackage{amsfonts}
\usepackage{amssymb}
\usepackage{graphicx}
\usepackage{amsthm,graphicx,color,yfonts}
\usepackage{bbm}
\usepackage{hyperref}
\usepackage{mathabx}
\usepackage{array}
\usepackage{umoline}
\usepackage{subcaption}
\usepackage[english]{babel}
\usepackage{ragged2e}

\usepackage[dvipsnames]{xcolor}
\usepackage{tikz}
\usepackage{adjustbox}
\usetikzlibrary{arrows.meta, calc}

\newtheorem{theorem}{Theorem}
\newtheorem{definition}[theorem]{Definition}
\newtheorem{proposition}[theorem]{Proposition}
\newtheorem{lemma}[theorem]{Lemma}
\newtheorem{corollary}[theorem]{Corollary}

\theoremstyle{remark}
\newtheorem{remark}[theorem]{Remark}

\newcommand{\R}{\mathbb{R}}

\newcommand{\Z}{\mathbb{Z}}

\newcommand{\bz}{\mathbf{z}}
\newcommand{\oz}{z}
\newcommand{\tbk}{\tilde{\mathbf{k}}}
\newcommand{\tko}{\tilde{k}_1}
\newcommand{\tkt}{\tilde{k}_2}
\newcommand{\bxi}{\boldsymbol{\xi}}
\newcommand{\xo}{\xi_1}
\newcommand{\xt}{\xi_2}
\newcommand{\tbxi}{\tilde{\boldsymbol{\xi}}}
\newcommand{\txo}{\tilde{\xi}_1}
\newcommand{\txt}{\tilde{\xi}_2}
\newcommand{\ttxo}{\tilde{\tilde{\xi}}_1}
\newcommand{\ttxt}{\tilde{\tilde{\xi}}_2}
\newcommand{\ttbxi}{\tilde{\tilde{\boldsymbol{\xi}}}}
\newcommand{\ph}{\omega}

\usepackage{wrapfig}
\usepackage{tikz}
\usetikzlibrary{arrows,calc,decorations.pathreplacing}
\definecolor{light-gray1}{gray}{0.90}
\definecolor{light-gray2}{gray}{0.80}
\definecolor{light-gray3}{gray}{0.60}

\numberwithin{equation}{section}

\numberwithin{theorem}{section}

\numberwithin{table}{section}

\numberwithin{figure}{section}

\ifx\pdfoutput\undefined
  \DeclareGraphicsExtensions{.pstex, .eps}
\else
  \ifx\pdfoutput\relax
    \DeclareGraphicsExtensions{.pstex, .eps}
  \else
    \ifnum\pdfoutput>0
      \DeclareGraphicsExtensions{.pdf}
    \else
      \DeclareGraphicsExtensions{.pstex, .eps}
    \fi
  \fi
\fi

\title[On the dispersive estimates on a honeycomb lattice]{On the dispersive estimates for the discrete Schr\"odinger equation on a honeycomb lattice}

\author[Y. Hong]{Younghun Hong}
\address{Department of Mathematics, Chung-Ang University, Seoul 06974, Korea}
\email{yhhong@cau.ac.kr}

\author[Y. Tadano]{Yukihide Tadano}
\address{Laboratory of Mathematical Science, Graduate School of Science, University of Hyogo, 
Hyogo 671-2201, Japan}
\email{tadano@sci.u-hyogo.ac.jp}

\author[C. Yang]{Changhun Yang}
\address{Department of Mathematics, Chungbuk National University, Cheongju 28644, Korea }
\email{chyang@chungbuk.ac.kr}

\begin{document}

\begin{abstract}
The discrete Schr\"odinger equation on a two-dimensional honeycomb lattice is a fundamental tight-binding approximation model that describes the propagation of waves on graphene. For the free evolution, we first show that the degenerate frequencies of the dispersion relation are completely characterized by three symmetric periodic curves (Theorem \ref{thm: gradient and hessian}), and that the three curves meet at Dirac points where conical singularities appear (see Figure \ref{fig: first primitive cell}). Based on this observation, we prove the $\ell^1\to \ell^\infty$ dispersion estimates for the linear flow depending on the frequency localization (Theorem \ref{main thm: dispersion estimate}). Collecting all, we obtain the dispersion estimate with $O(|t|^{-2/3})$ decay as well as Strichartz estimates. As an application, we prove small data scattering for a nonlinear model (Theorem \ref{Thm:nonlinear application}). The proof of the key dispersion estimates is based on the associated oscillatory integral estimates with degenerate phases and conical singularities at Dirac points. Our proof is direct and uses only elementary methods.
\end{abstract}

\maketitle

%\tableofcontents

\section{Introduction}

The discrete Schr\"odinger equation is a prototypical quantum lattice model that arises in various fields of theoretical and experimental physics. In condensed matter physics, the discrete Schr\"odinger operator is a model of tight-binding Hamiltonians of electrons in a crystalline solid. It also describes the effective dynamics of Bose-Einstein condensates trapped in a periodic optical lattice; by Bloch-Floquet theory, the discrete equation is derived from a formal tight-binding limit of the continuum equation with a periodic potential \cite{Wallace}, see also \cite[Section 1.2]{Kevrekidis Book}. In optics, discrete models are used to analyze wave propagation in optical waveguide arrays and photorefractive crystals. For more details on this topic, we refer to the book of Kevrekidis \cite{Kevrekidis Book}.

In particular, the two-dimensional honeycomb lattice case has attracted intense attention, because graphene, a single-layered sheet of carbon atoms located on the sites of a honeycomb lattice, is a fascinating material with many extraordinary properties; extremely high electron mobility \cite{Novoselov04}, room-temperature integer quantum Hall effect \cite{Novoselov07}, and chiral tunneling \cite{Katsnelson}.
The above-mentioned papers and subsequent research demonstrate that graphene has enormous potential applications in various areas.

It is known that many of the features of graphene arise from the local band structure near the Dirac points,
where the first two lowest energy bands touch each other at conical singularities. The dynamics of wave-packets localized near the Dirac points is governed by the massless Dirac equation, which explains the unique electronic properties of graphene. In the field of mathematical studies, the dynamics of quantum particles on a graphene sheet is mainly described by the Schr\"odinger equation on $\mathbb{R}^2$ with a honeycomb periodic potential. In the seminal work of Fefferman, Lee-Thorp and Weinstein \cite{FLW2}, it is rigorously shown that the discrete Schr\"odinger operator can be derived from the tight-binding limit of the continuum Schr\"odinger operator with a honeycomb periodic potential. In \cite{FW2}, Fefferman and Weinstein showed that in the vicinity of the Dirac points, wave-packets behave like those satisfying the massless Dirac equation.
See also \cite{AC24, ArbSpa, BCFLR, CW, CA14, CZ15, FFW22, FLW16, FLW1, FW1, FW12, FW20, LWZ} for mathematical (possibly physical) references.
In a similar context, the continuum Schr\"odinger equation is reduced to the discrete Schr\"odinger equation on a honeycomb lattice via the tight-binding approximation \cite{ACZ12}, and then the Dirac equation is derived taking the continuum limit \cite{{ANZ}}.

\subsection{Tight-binding model for graphene}
In this article, we are concerned with the fundamental tight-binding approximation model for graphene \cite{Wallace}. For its mathematical formulation, let 
$$\mathbf{\Lambda}=\mathbf{\Lambda}_\bullet:=\mathbb{Z}\mathbf{v}_1\oplus \mathbb{Z}\mathbf{v}_2$$
be the black-dotted lattice generated by the two linearly independent vectors
$$\mathbf{v}_1=\begin{bmatrix}\cos\frac{\pi}{6}\\ \sin\frac{\pi}{6}\end{bmatrix}=\begin{bmatrix}\frac{\sqrt{3}}{2}\\ \frac{1}{2}\end{bmatrix}\quad\textup{and}\quad \mathbf{v}_2=\begin{bmatrix}\cos(-\frac{\pi}{6})\\ \sin(-\frac{\pi}{6})\end{bmatrix}=\begin{bmatrix}\frac{\sqrt{3}}{2}\\ -\frac{1}{2}\end{bmatrix},$$
and for the unit vector $\mathbf{e}_1=[1 \ 0]^{\textup{T}}$, let 
$$\mathbf{\Lambda}_\circ:=\mathbf{\Lambda}_\bullet+\frac{1}{\sqrt{3}}\mathbf{e}_1,$$
denote the translated white-dotted lattice. Then, the union of the two lattices composes the honeycomb lattice
$$\mathbf{H}=\mathbf{\Lambda}_\bullet\cup \mathbf{\Lambda}_\circ$$
(see Figure \ref{fig: periodic structure of a hexagonal lattice} (A)). By construction, a scalar-valued function on the honeycomb lattice $\mathbf{H}$ can be realized as a $\mathbb{C}^2$-valued function on the \textit{periodic} sub-lattice $\mathbf{\Lambda}$ via the relation 
\begin{equation}\label{identification}
\mathbf{u}(\mathbf{x})=\begin{bmatrix}u_\bullet(\mathbf{x})\\ u_\circ(\mathbf{x})\end{bmatrix}=\begin{bmatrix}u(\mathbf{x})\\ u(\mathbf{x}+\frac{1}{\sqrt{3}}\mathbf{e}_1)\end{bmatrix},\quad \mathbf{x}\in\mathbf{\Lambda}.
\end{equation}
\begin{figure}
  \begin{center}
    \begin{subfigure}[t]{0.48\textwidth}
      \centering
      \includegraphics[width=0.9\textwidth]{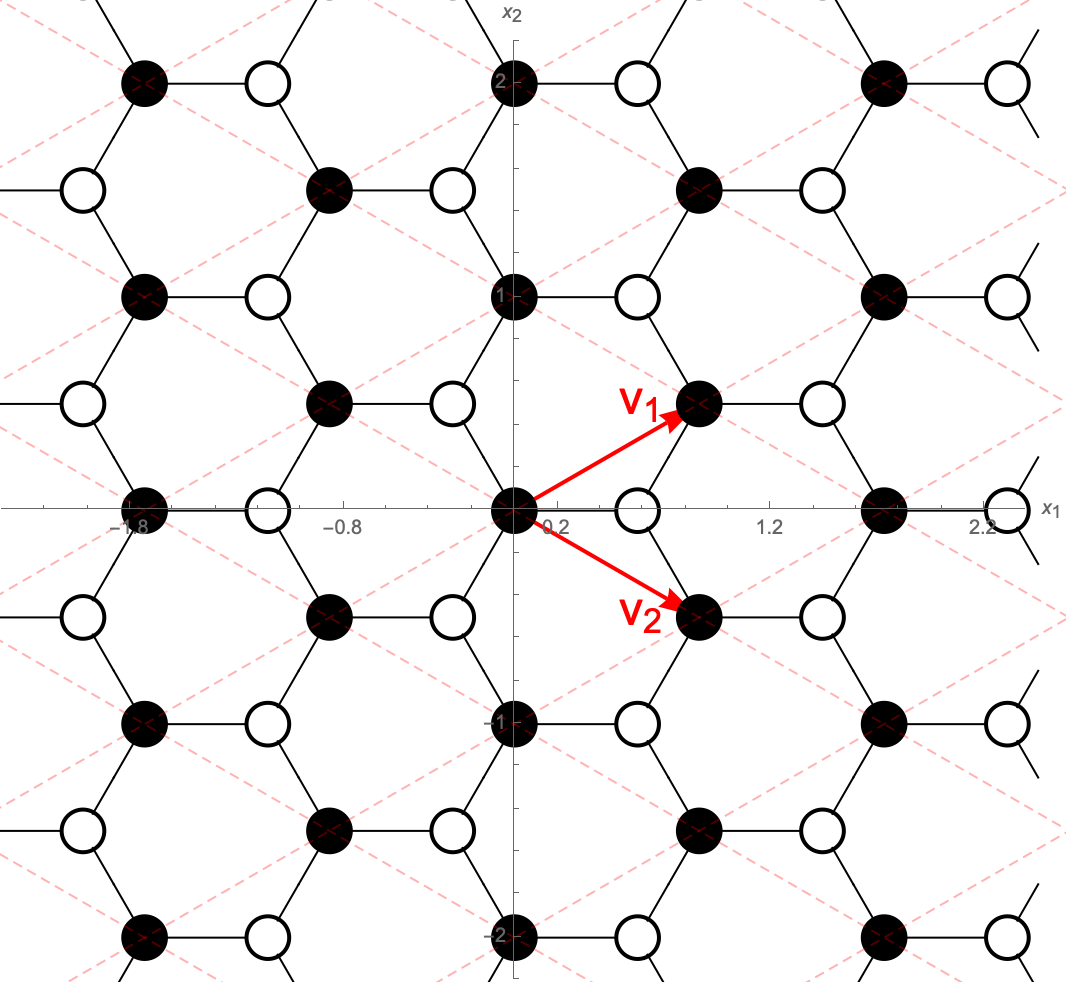}
      \caption{Hexagonal lattice}
      \label{fig:image-a}
    \end{subfigure}
    \quad
    \begin{subfigure}[t]{0.48\textwidth}
      \centering
      \includegraphics[width=0.8\textwidth]{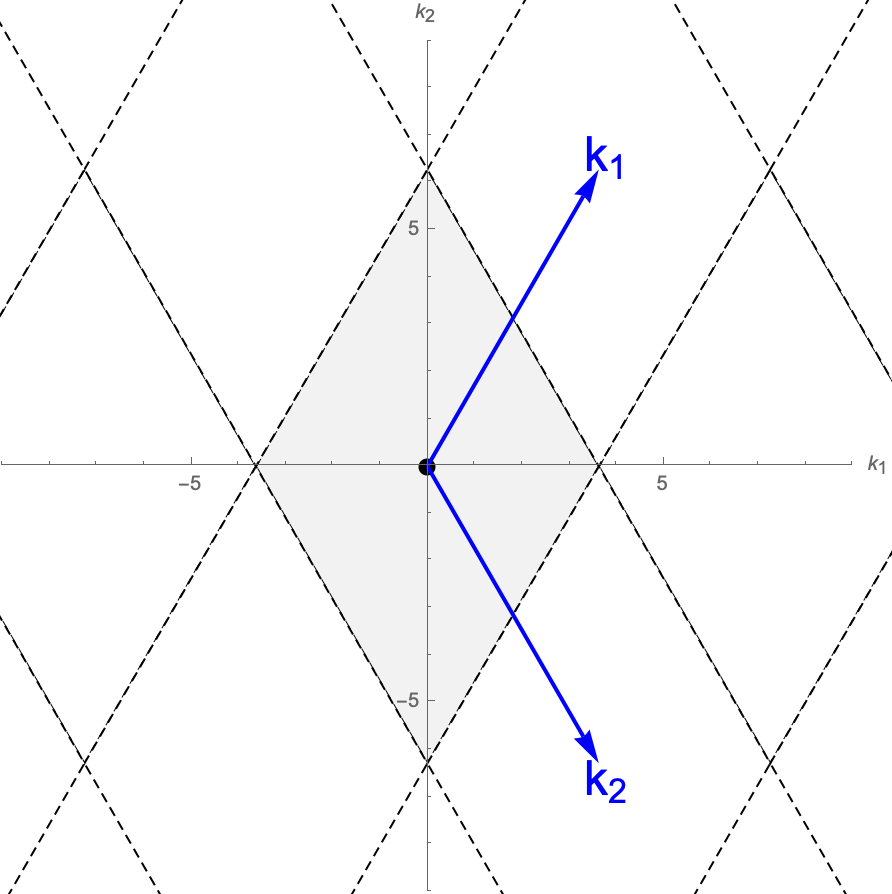}
      \caption{Frequency domain}
      \label{fig:image-b}
    \end{subfigure}
  \end{center}
  \caption{Periodic structure of the honeycomb lattice $\mathbf{H}$}
  \label{fig: periodic structure of a hexagonal lattice}
\end{figure}

For such a $\mathbb{C}^2$-valued function, its Lebesgue norm\footnote{The factor $\frac{\sqrt{3}}{2}$ represents the area of the primitive cell of $\mathbf{\Lambda}$.} is defined as
\begin{equation}\label{eq: discrete Lebesgue norm}
\|\mathbf{u}\|_{\ell_{\mathbf{x}}^r}=\|\mathbf{u}\|_{\ell_{\mathbf{x}}^r(\mathbf{\Lambda}; \mathbb{C}^2)} :=\begin{cases} \Big( \frac{\sqrt{3}}{2}\sum_{\mathbf{x}\in \mathbf{\Lambda}} |\mathbf{u}(\mathbf{x})|^r\Big)^{\frac1r} \quad & \text{ if } 1\le r<\infty, \\ 
  \sup_{\mathbf{x}\in \mathbf{\Lambda}} |\mathbf{u}(\mathbf{x})| \quad &\text{ if } r=\infty,
\end{cases} 
\end{equation}
where $|\mathbf{u}|=\sqrt{|u_{\bullet}|^2+|u_\circ|^2}$. The Fourier transform on the honeycomb lattice $\mathbf{H}$ is defined based on the group structure of the sub-lattice $\mathbf{\Lambda}$ as follows. Let 
$$\mathbf{k}_1=\frac{4\pi}{\sqrt{3}}\begin{bmatrix}\cos\frac{\pi}{3}\\ \sin\frac{\pi}{3}\end{bmatrix}=\begin{bmatrix}\frac{2\pi}{\sqrt{3}}\\2\pi\end{bmatrix}\quad\textup{and}\quad \mathbf{k}_2=\frac{4\pi}{\sqrt{3}}\begin{bmatrix}\cos(-\frac{\pi}{3})\\ \sin(-\frac{\pi}{3})\end{bmatrix}=\begin{bmatrix}\frac{2\pi}{\sqrt{3}}\\-2\pi\end{bmatrix}$$
be the reciprocal lattice vectors such that $\mathbf{v}_i\cdot\mathbf{k}_{j}=2\pi\delta_{ij}$, and define the dual periodic lattice by $$\mathbf{\Lambda}^*:=\mathbb{Z}\mathbf{k}_1\oplus \mathbb{Z}\mathbf{k}_2.$$

\begin{definition}[Fourier transform on the sub-lattice $\mathbf{\Lambda}$]\label{definition: Fourier transform}
$(i)$ For a scalar-valued function $f: \mathbf{\Lambda}\to\mathbb{C}$, the Fourier transform is defined by 
$$\hat{f}(\mathbf{k}):=\frac{\sqrt{3}}{2}\sum_{\mathbf{x}\in \mathbf{\Lambda}} f(\mathbf{x})e^{-i\mathbf{k} \cdot\mathbf{x}}: \mathbb{R}^2/\mathbf{\Lambda}^*\to\mathbb{C},$$
and the inverse Fourier transform of $g: \mathbb{R}^2/\mathbf{\Lambda}^*\to\mathbb{C}$ is defined by 
$$\check{g}(\mathbf{x}):=\frac{1}{(2\pi)^2}\int_{\mathbb{R}^2/\mathbf{\Lambda}^*}g(\mathbf{k})e^{i\mathbf{k} \cdot\mathbf{x}}d\mathbf{k}: \mathbf{\Lambda}\to\mathbb{C}.$$
$(ii)$ For a $\mathbb{C}^2$-valued function $\mathbf{f}=(f_\bullet, f_\circ): \mathbf{\Lambda}\to\mathbb{C}^2$ (resp., $\mathbf{g}=(g_\bullet, g_\circ): \mathbb{R}^2/\mathbf{\Lambda}^*\to\mathbb{C}^2$), the Fourier transform (resp., the inverse Fourier transform) is defined by 
$$\hat{\mathbf{f}}:=(\hat{f}_\bullet, \hat{f}_\circ)\quad\Big(\textup{resp., }\check{\mathbf{g}}=(\check{g}_\bullet, \check{g}_\circ)\Big).$$
\end{definition}

\begin{remark}\label{remark: Fourier transform definition}
$(i)$ The frequency domain $\mathbb{R}^2/\mathbf{\Lambda}^*$ can be identified with the primitive \textit{rhombic cell}, the gray region in Figure \ref{fig: periodic structure of a hexagonal lattice} (B) with the periodic boundary condition\footnote{In some context, $\mathbb{R}^2/\mathbf{\Lambda}^*$ is identified with the primitive hexagonal cell whose vertices are Dirac points, but the rhombic cell is more convenient to use in our analysis.}   
\begin{equation}\label{rhombic cell}
\mathcal{B}_{\textup{rhom}}:=\bigg\{s_1\mathbf{k}_1+s_2\mathbf{k}_2: -\frac{1}{2}< s_1,s_2\leq\frac{1}{2}\bigg\}.
\end{equation}
$(ii)$ In Definition \ref{definition: Fourier transform}, with a slight abuse of notation, we use the same symbols $\hat{\ }$ and $\check{\ }$ for both scalar-valued and $\mathbb{C}^2$-valued functions, but we distinguish them expressing $\mathbb{C}^2$-valued functions in bold, e.g. $\hat{\mathbf{f}}$ and $\check{\mathbf{g}}$.
\\
 $(iii)$ The Fourier transform on $\mathbf{\Lambda}$ is a unitary operator from $\ell_{\mathbf{x}}^2(\mathbf{\Lambda}; \mathbb{C}^2)$ onto $L^2(\mathbb{R}^2/\mathbf{\Lambda}^*)$, similarly to that on the cubic lattice $\mathbb{Z}^d$ (see, e.g., \cite[Remark~3.1.13]{RuTu}). 
\end{remark}

On the honeycomb lattice $\mathbf{H}$, the dynamics of free waves is governed by the Schr\"odinger equation
$$i\partial_t u=-\Delta_{\mathbf{H}} u,$$
where $u:\mathbb{R}\times\mathbf{H}\to\mathbb{C}$ and $\Delta_{\mathbf{H}}$ denotes the standard discrete Laplacian for nearest neighbors\footnote{For $\mathbf{x}\in \mathbf{\Lambda}_\bullet$, $(\Delta_{\mathbf{H}}u)(\mathbf{x})=4\{u(\mathbf{x}+\frac{1}{\sqrt{3}}\mathbf{e}_1)+u(\mathbf{x}+\frac{1}{\sqrt{3}}\mathbf{e}_1-\mathbf{v}_1)+u(\mathbf{x}+\frac{1}{\sqrt{3}}\mathbf{e}_1-\mathbf{v}_2)-3u(\mathbf{x})\}$; for $\mathbf{x}\in \mathbf{\Lambda}_\circ$, $(\Delta_{\mathbf{H}}u)(\mathbf{x})=4\{u(\mathbf{x}-\frac{1}{\sqrt{3}}\mathbf{e}_1)+u(\mathbf{x}-\frac{1}{\sqrt{3}}\mathbf{e}_1+\mathbf{v}_1)+u(\mathbf{x}-\frac{1}{\sqrt{3}}\mathbf{e}_1+\mathbf{v}_2)-3u(\mathbf{x})\}$.}. By \eqref{identification}, this scalar-valued equation can be identified with the $\mathbb{C}^2$-valued linear Schr\"odinger equation on the reduced periodic lattice $\mathbf{\Lambda}$, that is, 
\begin{equation}\label{LS}
i\partial_t\mathbf{u}=-\mathbf{\Delta}\mathbf{u},
\end{equation}
where
$$\mathbf{u}=\mathbf{u}(t,\mathbf{x} ):\mathbb{R}\times\mathbf{\Lambda}\to\mathbb{C}^2$$
and $\mathbf{\Delta}$ is the vector-valued discrete Laplacian given by
\begin{equation}\label{eq: definition of Laplacian on honeycomb}
(\mathbf{\Delta}\mathbf{u})(\mathbf{x}):=4\begin{bmatrix}u_\circ(\mathbf{x})+u_\circ(\mathbf{x} -\mathbf{v}_1)+u_\circ(\mathbf{x} -\mathbf{v}_2)-3u_\bullet(\mathbf{x})\\ u_\bullet(\mathbf{x})+u_\bullet(\mathbf{x}+\mathbf{v}_1)+u_\bullet(\mathbf{x}+\mathbf{v}_2)-3u_\circ(\mathbf{x})\end{bmatrix},\quad \mathbf{u}=\begin{bmatrix}u_\bullet\\u_\circ\end{bmatrix}: \mathbf{\Lambda}\to\mathbb{C}^2.
\end{equation}
Then, the solution to the equation \eqref{LS} with initial data $\mathbf{u}_0$ can be expressed as
\begin{equation}\label{eq: flow factorization}
\boxed{\quad e^{it\mathbf{\Delta}}\mathbf{u}_0=e^{-12it}\mathbf{O}(-i\nabla_{\mathbf{x}})\begin{bmatrix} e^{-4it\varphi(-i\nabla_{\mathbf{x}})}&0\\
0&e^{4it\varphi(-i\nabla_{\mathbf{x}})}\end{bmatrix}\mathbf{O}(-i\nabla_{\mathbf{x}})^*\mathbf{u}_0,\quad}
\end{equation}
where $\varphi(-i\nabla_{\mathbf{x}})$ is the scalar-valued Fourier multiplier with symbol
$$\varphi(\mathbf{k})=|\ph(\mathbf{k})|=\sqrt{3+2\cos(\mathbf{k}\cdot\mathbf{v}_1) +2\cos(\mathbf{k}\cdot\mathbf{v}_2) +2\cos(\mathbf{k}\cdot(\mathbf{v}_1-\mathbf{v}_2))}$$
with 
$$\ph(\mathbf{k})=1+e^{i\mathbf{k}\cdot\mathbf{v}_1}+e^{i\mathbf{k}\cdot\mathbf{v}_2}$$
and $\mathbf{O}(-i\nabla_{\mathbf{x}})$ is the vector-valued Fourier multiplier with symbol 
$$\mathbf{O}(\mathbf{k})=\begin{bmatrix}\ \frac{1}{\sqrt{2}}&\frac{\overline{\ph(\mathbf{k})}}{\sqrt{2}|\ph(\mathbf{k})|}\\
-\frac{\ph(\mathbf{k})}{\sqrt{2}|\ph(\mathbf{k})|}&\frac{1}{\sqrt{2}}\end{bmatrix}$$
(see Appendix \ref{sec: kernel} for the derivation of the formula \eqref{eq: flow factorization}).

\begin{remark}\label{remark: O operator}
By Parseval's theorem, $\mathbf{O}(-i\nabla_{\mathbf{x}})$ is unitary on $\ell^2(\mathbf{\Lambda}; \mathbb{C}^2)$. Moreover, it is bounded on $\ell^p(\mathbf{\Lambda}; \mathbb{C}^2)$ with $1<p<\infty$ (see Lemma \ref{lem:boundedness of O}).
\end{remark}

By the representation \eqref{eq: flow factorization}, the dynamics of the Schr\"odinger flow $e^{it\mathbf{\Delta}}\mathbf{u}_0$ is completely determined by the geometric structure of the frequency surfaces 
\begin{equation}\label{eq: frequency surfaces}
\mathcal{S}_\pm:=\Big\{\big(\mathbf{k},\pm \varphi(\mathbf{k})\big):\mathbf{k}\in \mathbb{R}^2/\mathbf{\Lambda}^*\Big\}.
\end{equation}
Note that the points where the two surfaces meet, i.e., where $\varphi(\mathbf{k})=0$, are called \textit{Dirac points}.
In fact, the Dirac points are the two points $[0 \ \pm\frac{4\pi}{3}]^{\textup{T}}+\mathbf{\Lambda}^*$ in the primitive rhombic zone $\mathcal{B}_{\textup{rhom}}$ (see \eqref{rhombic cell} and Figure \ref{fig: first primitive cell}). It is well known that the frequency surfaces $\mathcal{S}_\pm$ have a conical singularity at a Dirac point $\mathbf{K}_\star$. By Fefferman, Lee-Thorp and Weinstein \cite{FLW2}, it is shown that in the tight-binding regime, the first two lowest energy bands for the continuum model converge to $\pm\varphi(\mathbf{k})$ uniformly in $\mathbf{k}$.

\section{Main results}

The purpose of this article is to investigate the detailed dispersive properties of wave propagation in a honeycomb lattice, depending on frequency localization. In a cubic lattice, $L^1\to L^\infty$ dispersion bounds have been established for discrete Schr\"odinger, Klein-Gordon and wave equations; \cite{DaFiYo, EKT, MZ20, MZ, PS, SaKaWe} for a one-dimensional lattice, and \cite{BCH2312, BCH, BG, HY2019,  CI21, Sch, SK05} for multi-dimensional lattices. Such time decay estimates are a fundamental tool for studying discrete dispersive equations in various aspects. 
 For example, they have been applied to nonlinear problems such as the continuum limit of discrete models 
\cite{CH2501, CA2023, HY2019-2, HKY2021, HKNY2021, HKY2023, IZ-CR05-01, IZ-CR05-02, IZ-SIAM09, IZ-JMPA12, KOVW2023,  Wang2501}. 
We also note that in a similar context, uniform resolvent estimates of discrete Schr\"odinger operators on a cubic lattice are proved \cite{CS, TT, Taira20, Taira21}.  The spectral and scattering properties of discrete Schr\"odinger operators on honeycomb (and general) lattices are also considered in \cite{Ando13,AnIsMo,PaRi,Ta},   where they avoid treating the energy level of the Dirac points.  To the best of the authors' knowledge, 
 no dispersive estimates on the honeycomb lattice have yet been established, despite its physical importance. 

\subsection{Characterization of degenerate frequencies}
By the representation \eqref{eq: flow factorization}, the time-decay of the Schr\"odinger flow $e^{it\mathbf{\Delta}}\mathbf{u}_0$ is determined by
the phase function
\begin{equation}\label{eq: phase function}
\boxed{\quad
\varphi(\mathbf{k})=\sqrt{3+2\cos(\mathbf{k}\cdot\mathbf{v}_1) +2\cos(\mathbf{k}\cdot\mathbf{v}_2) +2\cos(\mathbf{k}\cdot(\mathbf{v}_1-\mathbf{v}_2))}\quad}
\end{equation}
where $\mathbf{v}_1=(\frac{\sqrt{3}}{2},\frac{1}{2})$ and $\mathbf{v}_2=(\frac{\sqrt{3}}{2},-\frac{1}{2})$. In order to find the precise dispersion rate, a crucial step is to compute the local series expansions of the phase.

Our first main result provides the explicit formulae for the Hessian of the phase $\varphi(\mathbf{k})$ and its determinant, given by the three periodic functions 
\begin{equation}\label{eq: alpha function notations}
\left\{\begin{aligned}
\alpha_{1}(\mathbf{k})&:=1+\cos(\mathbf{k}\cdot\mathbf{v}_2)+\cos(\mathbf{k}\cdot(\mathbf{v}_1-\mathbf{v}_2)),\\
\alpha_{2}(\mathbf{k})&:=1+\cos(\mathbf{k}\cdot\mathbf{v}_1)+\cos(\mathbf{k}\cdot(\mathbf{v}_1-\mathbf{v}_2)),\\
\alpha_{12}(\mathbf{k})&:=1+\cos(\mathbf{k}\cdot\mathbf{v}_1)+\cos(\mathbf{k}\cdot\mathbf{v}_2),
\end{aligned}\right.
\end{equation}
from which degenerate frequencies are completely characterized by the three periodic curves (see Figure \ref{fig: first primitive cell} and Remark \ref{remark: characterization of degenerate frequencies} $(ii)$). 

\begin{theorem}[Characterization of degenerate frequencies]\label{thm: gradient and hessian}
For any frequency $\mathbf{k}\in \mathbb{R}^2/\mathbf{\Lambda}^*$ such that $\varphi(\mathbf{k})\neq0$, i.e, $\mathbf{k}$ is not a Dirac point, we have
$$(\nabla^2\varphi)(\mathbf{k})=\frac{1}{\varphi(\mathbf{k})^3}\mathbf{V}\begin{bmatrix} 
-\alpha_{2}(\mathbf{k})(\alpha_1(\mathbf{k})+\alpha_{12}(\mathbf{k})) & \alpha_{1}(\mathbf{k})\alpha_{2}(\mathbf{k}) \\
\alpha_{1}(\mathbf{k})\alpha_{2}(\mathbf{k}) & -\alpha_{1}(\mathbf{k})(\alpha_2(\mathbf{k})+\alpha_{12}(\mathbf{k}))
\end{bmatrix}\mathbf{V}^\textup{T},$$
where $\mathbf{V}=[\mathbf{v}_1\ \mathbf{v}_2]$, and
\begin{equation}\label{eq: determinant of the Hessian}
\textup{det}(\nabla^2\varphi)(\mathbf{k})=\frac{3\alpha_{1}(\mathbf{k})\alpha_{2}(\mathbf{k})\alpha_{12}(\mathbf{k})}{4\varphi(\mathbf{k})^4}.
\end{equation}
As a consequence, $\textup{det}(\nabla^2\varphi)(\mathbf{k})=0$ if and only if
\begin{equation}\label{eq: degeneracy equation}
\alpha_{1}(\mathbf{k})=0,\quad \alpha_{2}(\mathbf{k})=0\quad \textup{or}\quad \alpha_{12}(\mathbf{k})=0.
\end{equation}
\end{theorem}

\begin{figure}
    \centering
    \includegraphics[width=0.6\linewidth]{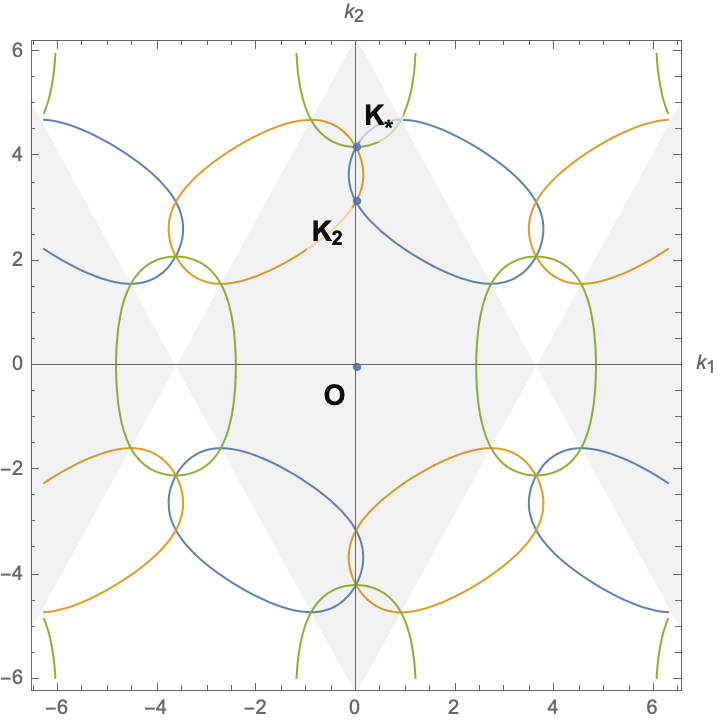}
    \caption{Degenerate frequency curves in Theorem~\ref{thm: gradient and hessian}. }
    \label{fig: first primitive cell}
\end{figure}

\begin{remark}\label{remark: characterization of degenerate frequencies}
$(i)$ Theorem \ref{thm: gradient and hessian} provides more detailed geometric information of the well-known frequency surfaces $\mathcal{S}_\pm$ for the tight-binding model of graphene. To the best of authors' knowledge, this is the first result for the classification of degenerate frequencies of the phase $\varphi(\mathbf{k})$.\\
$(ii)$ Notably, the determinant of the Hessian \eqref{eq: determinant of the Hessian} is factorizable, and the phase function $\varphi$ is degenerate if and only if \eqref{eq: degeneracy equation} is satisfied. By drawing the graphs of $\alpha_{1}(\mathbf{k})=0$, $\alpha_{2}(\mathbf{k})=0$ and $\alpha_{12}(\mathbf{k})=0$ (see Figure \ref{fig: first primitive cell}), one can see that the degenerate frequencies are located on the three simple periodic curves  with symmetry; the three degenerate frequency curves are symmetric under 60 degree rotation. Moreover, all three curves intersect only at Dirac points. These facts are completely non-trivial, and rely heavily on the symmetric algebraic structure of the graphene lattice model. Indeed, it is not easy to expect factorization of $\textup{det}(\nabla^2\varphi)(\mathbf{k})$ a priori, because the direct expansion of $\textup{det}(\nabla^2\varphi)(\mathbf{k})$ has many rational functions of trigonometric functions, and it is too difficult to reorganize them using the trigonometric identities by hand. The ``magical" formula \eqref{eq: determinant of the Hessian} was obtained unexpectedly by \textit{Mathematica-aided} computations.  
 Note also that for non-symmetric hexagonal lattice models, such as the boron nitride model \cite{Ran}, the phase function does not have a similar factorization property.\\
$(iii)$ The notation \eqref{eq: alpha function notations} has symmetry in that $\alpha_j(\mathbf{k})=\frac{\varphi(\mathbf{k})^2-1}{2}-\cos(\mathbf{k}\cdot\mathbf{v}_j)$ for $j=1,2$, and $\alpha_{12}(\mathbf{k})=\frac{\varphi(\mathbf{k})^2-1}{2}-\cos(\mathbf{k}\cdot(\mathbf{v}_1-\mathbf{v}_2))$.
\end{remark}

By Theorem \ref{thm: gradient and hessian}, the periodic frequency domain is decomposed as 
$$\mathbb{R}^2/\mathbf{\Lambda}^*=\bigsqcup_{j=0}^3 \mathcal{K}_j,$$
where $\mathcal{K}_j$ denotes the set of intersections of $j$-many frequency curves, precisely,  
\begin{equation}\label{eq: decomposition of the frequency domain}
\begin{aligned}
\mathcal{K}_0&:=\Big\{\mathbf{k}\in(\mathbb{R}^2/\mathbf{\Lambda}^*)\setminus \mathcal{K}_3: \textup{det}(\nabla^2\varphi)(\mathbf{k})\neq0\Big\},\\
\mathcal{K}_1&:=\Big\{\mathbf{k}\in\mathbb{R}^2/\mathbf{\Lambda}^*: \textup{det}(\nabla^2\varphi)(\mathbf{k})=0\Big\}\setminus (\mathcal{K}_2\cup \mathcal{K}_3),\\
\mathcal{K}_2&:=\bigg\{\bigg(\frac{\sqrt{3}\pi}{2},\pm\frac{\pi}{2}\bigg), \bigg(\frac{\sqrt{3}\pi}{6},\pm\frac{3\pi}{2}\bigg), (0,\pm\pi)\bigg\}+\mathbf{\Lambda}^*,\\
\mathcal{K}_3&:=\bigg(0,\pm\frac{4\pi}{3}\bigg)+\mathbf{\Lambda}^*\quad\textup{(Dirac points)}.
\end{aligned}
\end{equation}

\subsection{Dispersion and Strichartz estimates}

Based on the characterization of degenerate frequencies (Theorem \ref{thm: gradient and hessian}), we establish our second main result, that is, the dispersion estimates for the Schr\"odinger flow on the honeycomb lattice $\mathbf{H}$, whose decay rate depends on frequency localization.

\begin{theorem}[Frequency localized dispersion estimates]\label{main thm: dispersion estimate}
For each $\mathbf{K}\in \mathcal{K}_j$ with $j=0,1,2,3$, there exist $C_j=C_j(\mathbf{K})>0$ and a smooth cut-off $\chi_{\approx \mathbf{K}}$ such that $\chi_{\approx \mathbf{K}}$ is supported in a sufficiently small neighborhood of $\mathbf{K}$, $\chi_{\approx \mathbf{K}}=1$ near $\mathbf{K}$, and 
\begin{equation}\label{eq: dispersion estimate}
\|\mathbf{O}(-i\nabla_{\mathbf{x}})^*e^{it\mathbf{\Delta}}\mathbf{O}(-i\nabla_{\mathbf{x}})P_{\approx\mathbf{K}}\mathbf{u}_0\|_{\ell_{\mathbf{x}}^\infty(\mathbf{\Lambda}; \mathbb{C}^2)}\leq\frac{C_j}{(1+|t|)^{a_j}}\|\mathbf{u}_0\|_{\ell_{\mathbf{x}}^1(\mathbf{\Lambda}; \mathbb{C}^2)},
\end{equation}
where $\mathbf{O}(-i\nabla_{\mathbf{x}})$ is given by \eqref{eq: flow factorization}, $P_{\approx\mathbf{K}}$ is the Fourier multiplier with symbol $\chi_{\approx \mathbf{K}}$, that is, $\widehat{P_{\approx\mathbf{K}}u}(\mathbf{k})=\chi_{\approx \mathbf{K}}(\mathbf{k})\hat{u}(\mathbf{k})$, and 
\begin{equation}\label{eq: dispersion estimate rate}
a_j=\left\{\begin{aligned}
&1 &&\textup{if }j=0,\\
&\tfrac{5}{6}&&\textup{if }j=1,\\
&\tfrac{2}{3}&&\textup{if }j=2,\\
&\tfrac{5}{6}&&\textup{if }j=3.
\end{aligned}\right.
\end{equation}
\end{theorem}

\begin{remark}[Decay rate]
$(i)$ In the study of dispersive estimates for discrete models, oscillatory integrals with degenerate phases arise naturally. Such degeneracy typically leads to slower time decay, as observed in the discrete  Schrödinger \cite{SK05}, fractional Schrödinger \cite{CA2023,Wang2501}, and Klein–Gordon equations \cite{BG,CI21}, compared to their counterparts in $\mathbb{R}^d$. 
By contrast, in the (massless) wave case, both the continuous and discrete models exhibit degeneracy. Interestingly, however, the discrete wave equation enjoys faster time decay \cite{MPPT,Sch,BCH2312}. This improvement is due to additional oscillations generated by higher-order terms in the phase function.  \\ 
$(ii)$ The slowest $O(|t|^{-2/3})$-decay bound is obtained at the intersections of two degenerate frequency curves, not at Dirac points.\\
$(iii)$ Near Dirac points, even though the frequency surfaces $\mathcal{S}_\pm$ (see \eqref{eq: frequency surfaces}) are asymptotically conic (see Lemma \ref{lemma: phase function asymptotic near the Dirac point}), the faster $O(|t|^{-5/6})$-dispersion rate is obtained compared to the $O(|t|^{-1/2})$-decay rate for the standard linear wave equation on $\mathbb{R}^2$.
Indeed, the phase $\varphi(\mathbf{k})$ behaves like that for the wave equation, precisely, $\varphi(\mathbf{K}+\mathbf{k})\approx \frac{\sqrt{3}}{2}|\mathbf{k}|$ but only in the leading order. By applying the same argument used for deriving the decay of the wave equation on $\R^2$, that is, relying only on the leading-order oscillation, we can readily show that the integral near Dirac points also exhibits $O(|t|^{-1/2})$-decay (see Proposition~\ref{prop: preliminary degenerate oscillatory integral near the Dirac point}). However, the faster $O(|t|^{-5/6})$-decay can only be captured from the additional oscillation by higher-order terms (see Lemma \ref{lemma: phase function asymptotic near the Dirac point} and \ref{lemma: modified phase function asymptotic near the Dirac point}).
\end{remark}

\begin{remark}[Non-degenerate frequency case $j=0$]\label{remark: non-degenerate frequency case}
For each $\mathbf{K}\in\mathcal{K}_0$, one can choose a smooth cut-off $\chi_{\approx \mathbf{K}}$ with sufficiently small support where the phase is non-degenerate. Thus, the standard oscillatory integral estimate (see Proposition~\ref{Lam:Stationary phase method}) immediately yields Theorem \ref{main thm: dispersion estimate} with $j=0$, while a more delicate analysis is needed when $\mathbf{K}$ is located on a degenerate frequency curve. 
\end{remark}

\begin{remark}[Degenerate frequency case $j=1,2,3$]\label{remark: degenerate frequency case}
\noindent $(i)$ $(j=1,2)$ When $\mathbf{K}\in\mathcal{K}_1\cup\mathcal{K}_2$, two different proofs of the bound \eqref{eq: dispersion estimate} are provided for the convenience of the readers. In Sections \ref{sec: oscillatory integral localized at a Dirac point at an intersection of two degenerate frequency curves} and \ref{sec: oscillatory integral localized at a non-intersection point}, a direct, self-contained proof will be given,  involving elementary integration by parts and changes of variables. Alternatively, in Appendix \ref{sec: alternative proof of K1 and K2}, the same bound \eqref{eq: dispersion estimate} will be shown by the known oscillatory integral theory, such as celebrated Varchenko's theorem (Theorem \ref{Varchenko's theorem}) relating Newton polygons and asymptotics of the oscillatory integrals, and Karpushkin's theorem (Theorem \ref{Thm:Karpushkin}) for stability of oscillatory integrals (see also \cite{Dui, IkMu, PhSt}). The latter proof is shorter, but it might be less approachable to some readers, because the proofs of the known theorems involve other known theories, e.g., resolution of singularities in algebraic geometry \cite{Hironaka}, or unfoldings of singularities in symplectic geometry \cite{ArGuVar}.\\
$(ii)$ $(j=3)$ Another benefit of the former elementary proof is that the physically most interesting case $\mathbf{K}\in\mathcal{K}_3$ can be treated in a similar way. Note that near a Dirac point, the phase function is not only degenerate but also nondifferentiable as the frequency surface is asymptotically conic.  However, the theorems used in Appendix \ref{sec: alternative proof of K1 and K2} cannot be applied directly to this case because they require some regularity of phase functions  \footnote{It might be possible to prove the desired bound modifying the known algorithm (see \cite{BCH} for instance).}.\\
$(iii)$ The proofs in Sections \ref{sec: dispersion near a Dirac point}--\ref{sec: oscillatory integral localized at a non-intersection point} are based on the robust technique to show dispersion estimates for the wave-type equations \cite{Oh, Stein}.
\end{remark}

\begin{remark}[Sharpness of decay rate]
The decay rates in \eqref{eq: dispersion estimate rate} are optimal except for the case $j=3$.
As explained in Remark~\ref{remark: non-degenerate frequency case} and Remark~\ref{remark: degenerate frequency case}, the bound for $j=0$ is obtained via the non-degenerate phase method (see Proposition~\ref{Lam:Stationary phase method}), and the bounds for $j=1,2$ are established through the Newton polyhedron method (Varchenko’s theorem (Theorem \ref{Varchenko's theorem})), both of which yield not only the upper bounds but also precise asymptotic expansions of the relevant oscillatory integrals. For $j=3$, corresponding to Dirac points, the phase function is non-analytic, so these theories cannot be applied; although we expect the bound to be sharp, the asymptotics remain out of reach.
\end{remark}

By compactness of the frequency domain $\mathbb{R}^2/\mathbf{\Lambda}^*$, collecting all and interpolating with the unitarity of the flow $e^{it\mathbf{\Delta}}$, we obtain the dispersion estimate without frequency localization.  

\begin{corollary}[Dispersion estimate]\label{cor: dispersion estimate}
For $r\geq 2$, we have
$$\|\mathbf{O}(-i\nabla_{\mathbf{x}})^*e^{it\mathbf{\Delta}}\mathbf{u}_0\|_{ \ell_{\mathbf{x}}^r(\mathbf{\Lambda}; \mathbb{C}^2)}\lesssim \frac{1}{(1+|t|)^{\frac{4}{3}(\frac{1}{2}-\frac{1}{r})}}\|\mathbf{O}(-i\nabla_{\mathbf{x}})^*\mathbf{u}_0\|_{\ell_{\mathbf{x}}^{r'}(\mathbf{\Lambda}; \mathbb{C}^2)},$$
where $r'$ is the Hölder conjugate of $r$, i.e., $\frac1r+\frac{1}{r'}=1$.
\end{corollary}

\begin{remark}
The global decay rate in Corollary \ref{cor: dispersion estimate}  should be considered optimal, regardless of the unresolved sharpness issue in the Dirac case.
\end{remark}

As a consequence, employing the standard interpolation argument \cite{KeelTao}, we deduce Strichartz estimates. We call $(q,r)$ an \textit{admissible pair} if
$$\frac{3}{q}+\frac{2}{r}=1\quad\textup{and}\quad 2\leq q,r\leq\infty.$$
\begin{corollary}[Strichartz estimates]\label{cor: Strichartz}
For admissible pairs $(q,r)$ and $(\tilde{q},\tilde{r})$, we have
\begin{equation}\label{eq: Strichart estimates 1}
\|\mathbf{O}(-i\nabla_{\mathbf{x}})^*e^{it\mathbf{\Delta}}\mathbf{u}_0\|_{L_t^q(\mathbb{R}; \ell_{\mathbf{x}}^r(\mathbf{\Lambda}; \mathbb{C}^2))}\lesssim \|\mathbf{u}_0\|_{\ell_{\mathbf{x}}^{2}(\mathbf{\Lambda}; \mathbb{C}^2)}
\end{equation}
and 
\begin{equation}\label{eq: Strichart estimates 2}
\bigg\|\mathbf{O}(-i\nabla_{\mathbf{x}})^*\int_0^t e^{i(t-t')\mathbf{\Delta}}\mathbf{F}(t') dt'\bigg\|_{L_t^q(\mathbb{R}; \ell_{\mathbf{x}}^r(\mathbf{\Lambda}; \mathbb{C}^2))}\lesssim \|\mathbf{O}(-i\nabla_{\mathbf{x}})^*\mathbf{F}\|_{L_t^{\tilde{q}'}(\mathbb{R}; \ell_{\mathbf{x}}^{\tilde{r}'}(\mathbf{\Lambda}; \mathbb{C}^2))}.
\end{equation}
\end{corollary}

\begin{remark}
In Corollary \ref{cor: dispersion estimate} and \ref{cor: Strichartz}, we may drop $\mathbf{O}(-i\nabla_{\mathbf{x}})^*$ provided that $r, \tilde{r}\neq \infty$, because $\mathbf{O}(-i\nabla_{\mathbf{x}})^*$ is bounded on $\ell^r(\mathbf{\Lambda}; \mathbb{C}^2)$ for $1<r<\infty$ (see Lemma \ref{lem:boundedness of O}).
\end{remark}

\subsection{Nonlinear application}
As a simple application of our main linear estimates, we establish the small data scattering for the nonlinear Schr\"odinger equation on the reduced periodic lattice $\mathbf{\Lambda}$ with a power-type nonlinearity\footnote{By \eqref{identification}, the $\mathbb{C}^2$-valued equation \eqref{NLS} is equivalent to the scalar-valued equation $i\partial_t u=-\Delta_{\mathbf{H}}u\pm |u|^{p-1}u$ on the hexagonal lattice $\mathbf{H}$.}
\begin{equation}\label{NLS}
i\partial_t\mathbf{u}=-\mathbf{\Delta}\mathbf{u} +  
  \mathcal{N}(\mathbf{u}),
\end{equation}
where 
$$\mathbf{u}=\mathbf{u}(t,\mathbf{x} )=\begin{bmatrix}u_\bullet(t,\mathbf{x} )\\u_\circ(t,\mathbf{x} )\end{bmatrix}:\mathbb{R}\times\mathbf{\Lambda}\to\mathbb{C}^2$$
and the nonlinear term is given by 
\begin{align*}
  \mathcal{N}(\mathbf{u}) = \pm\begin{bmatrix}
    |u_{\bullet}|^{p-1} u_{\bullet}  \\ 
    |u_\circ|^{p-1} u_{\circ} 
  \end{bmatrix}
  =\pm\begin{bmatrix}
    |u_{\bullet}|^{p-1} & 0  \\ 
   0 &  |u_\circ|^{p-1}  
  \end{bmatrix}\mathbf{u}.
\end{align*}

\begin{theorem}[Small data scattering]\label{Thm:nonlinear application}
Suppose that $r>2$, $p>3+\max(\frac{3}{r-2}-\frac32,\frac12-\frac2r)$ and $\|\mathbf{u}_0\|_{L_{\mathbf{x}}^{r'}(\mathbf{\Lambda}; \mathbb{C}^2)} \le \epsilon_0$ for sufficiently small $\epsilon_0>0$. Then, the Cauchy problem with initial data $\mathbf{u}_0$ has a unique global solution satisfying 
\begin{align}\label{nonlinear time decay}
\sup_{t\in\mathbb{R}}\Big\{(1+|t|)^{\frac43(\frac12-\frac1r)}\|\mathbf{u}(t)\|_{ \ell_{\mathbf{x}}^{r}(\mathbf{\Lambda}; \mathbb{C}^2)}\Big\}\lesssim\|\mathbf{u}_0\|_{ \ell_{\mathbf{x}}^{r'}(\mathbf{\Lambda}; \mathbb{C}^2)}.
\end{align}  
Furthermore, there exists a forward-in-time (resp., backward-in-time) scattering state $\mathbf{u}_+\in \ell_{\mathbf{x}}^{2}(\mathbf{\Lambda}; \mathbb{C}^2)$ (resp., $\mathbf{u}_-\in \ell_{\mathbf{x}}^{2}(\mathbf{\Lambda}; \mathbb{C}^2)$) such that for $t\geq 1$ (resp., $t\leq -1$),
\begin{equation}\label{eq: nonlinear scattering}
  \|e^{-it\mathbf{\Delta}}\mathbf{u}(t)-\mathbf{u}_\pm\|_{\ell_{\mathbf{x}}^{2}(\mathbf{\Lambda}; \mathbb{C}^2)}\lesssim \frac{\|\mathbf{u}_0\|_{ \ell_{\mathbf{x}}^{r'}(\mathbf{\Lambda}; \mathbb{C}^2)}^p}{(1+|t|)^{\sigma_{r,p}-1}},
\end{equation}
where
\begin{equation}\label{eq: nonlinear decay rate}
\sigma_{r,p}:= \left\{\begin{aligned}
&\frac43\bigg(\frac12-\frac1r\bigg)p &&\text{if }p\ge r-1,\\
&\frac{4}{3}\bigg(\frac{1}{2}-\frac{1}{r'p}\bigg)p&&\text{if }p<r-1.
\end{aligned}\right.
\end{equation}
\end{theorem}

\begin{remark}
$(i)$ Theorem \ref{Thm:nonlinear application} includes super-cubic nonlinearities, i.e., $p>3$. Note that for the assumption on $p$, the minimum value of $3+\max(\frac{3}{r-2}-\frac32,\frac12-\frac2r)$ is attained when $r=4$, and it is $3$. Indeed, by the time decay rate (Corollary \ref{cor: dispersion estimate}), it is natural to conjecture that small data scattering holds if $p>\frac{5}{2}=1+\frac{3}{2}$,  since in this regime the nonlinear term is integrable in time. The gap is due to the lack of the vector-field identity\footnote{That is, $x+2it\nabla=e^{it\Delta}xe^{-it\Delta}$ for free Schrödinger equation in $\mathbb{R}^d$.} in the discrete setting.\\
$(ii)$ If one employs a weaker $O(|t|^{-1/2})$ wave-like $\ell^1\to \ell^\infty$ dispersion bound from the crude estimate (Proposition \ref{prop: preliminary degenerate oscillatory integral near the Dirac point}), 
a similar small data scattering can be obtained, but the range of nonlinearities is restricted to $p>\max(2+\frac{4}{r-2},4-\frac2r)$, with $p>\frac{3+\sqrt{17}}{2}$.
\end{remark}

\subsection{Notations}\label{subsec: notations}
Throughout this article, we denote $A\lesssim B$ (resp., $A\gtrsim B$, or $A\sim B$) if there exists $C\geq 1$ such that $A\leq CB$ (resp., $A\geq B/C$, or $\frac{1}{C}A\leq B\leq CA$). If the implicit constant $C\geq 1$ depends on some other parameter $a$, then we denote by $A\lesssim_a B$, $A\gtrsim_a B$, or $A\sim_a B$. However, if such dependence is not essential in analysis, the subscript $a$ is omitted. It is important to note that in any case, the implicit constants do not depend on $t\in\mathbb{R}$ and the specific choice of $\mathbf{x}$.

Let $\chi_0:\mathbb{R}\to[0,\infty)$ be a smooth cut-off such that 
\begin{equation}\label{eq: chi0}
\chi_0(s)=\left\{\begin{aligned}
&1&&\textup{if }|s|\leq\tfrac{1}{2},\\
&0&&\textup{if }|s|\geq 1,
\end{aligned}\right.
\end{equation}
and define 
\begin{equation}\label{eq: chi1}
\chi_1:=1-\chi_0.
\end{equation}

For a non-negative integer $m$, let $\mathcal{O}_m(\mathbf{k})$ denote an analytic function near the origin which satisfies $|\mathcal{O}_m(\mathbf{k})|\lesssim |\mathbf{k}|^m$ and is of the form 
\begin{equation}\label{eq: rm}
\mathcal{O}_m(\mathbf{k})=\sum_{m_1+m_2\geq m} c_{m_1,m_2}k_1^{m_1} k_2^{m_2}.
\end{equation}
In most situations, the specific choice of the coefficients in \eqref{eq: rm} is not essential, but only their bounds are important. If this is the case, abusing notations, we express any such analytic functions by $\mathcal{O}_m(\mathbf{k})$ rather than introducing more notations like $\tilde{\mathcal{O}}_m(\mathbf{k})$, $\tilde{\tilde{\mathcal{O}}}_m(\mathbf{k})$, ...

\subsection{Organization of the paper}
The rest of the paper is organized as follows. In Section \ref{sec: Proof of phase formulae}, we give a proof of the first main result (Theorem \ref{thm: gradient and hessian}). The next four sections are devoted to the proof of the dispersion estimate (Theorem \ref{main thm: dispersion estimate}). In Section \ref{sec: reduction to the oscillatory integral estimates}, we reduce the proof of the dispersion estimate to that of the degenerate phase oscillatory integral estimate. Then, in Sections \ref{sec: dispersion near a Dirac point}, \ref{sec: oscillatory integral localized at a Dirac point at an intersection of two degenerate frequency curves} and \ref{sec: oscillatory integral localized at a non-intersection point}, we prove the oscillatory integral bound corresponding to the case $\mathbf{K}\in\mathcal{K}_j$ with $j=3,2,1$, respectively. In Section \ref{sec: nonlinear application}, we prove small data scattering (Theorem \ref{Thm:nonlinear application}). In Appendix \ref{sec: kernel} and \ref{sec: boundedness of O operator}, we provide the proof of the factorization formula of the linear flow \eqref{eq: flow factorization} and that of the boundedness of the operator $\mathbf{O}(-i\nabla_{\mathbf{x}})$ (Lemma \ref{lem:boundedness of O}).  In Appendix \ref{Appen: Review of OIE}, we briefly summarize the well-known oscillatory integral estimates. In Appendix \ref{sec: alternative proof of K1 and K2}, we provide alternative proofs of the main results in Section \ref{sec: oscillatory integral localized at a Dirac point at an intersection of two degenerate frequency curves} and \ref{sec: oscillatory integral localized at a non-intersection point}. As mentioned in Remark \ref{remark: degenerate frequency case} $(i)$, these proofs might be easier to read for those who have expertise in degenerate oscillatory integrals. 

\subsection{Acknowledgment}
Y. Hong was supported by National Research Foundation of Korea (NRF) grant funded by the Korean government (MSIT) (No. RS-2023-00208824, No. RS-2023-00219980). Y. Tadano was supported by JSPS KAKENHI Grant Number JP23K12991. C. Yang was supported by the National Research Foundation of Korea (NRF) grant funded by the Korean government (MSIT) (No. 2021R1C1C1005700). The authors would like to thank Professor Sung-Jin Oh for explaining the integration-by-parts trick crucially used throughout this article.

\section{Characterization of degenerate frequencies: Proof of Theorem \ref{thm: gradient and hessian}}\label{sec: Proof of phase formulae}

In this section, we prove Theorem \ref{thm: gradient and hessian}. Introducing the new variable $\tilde{\mathbf{k}}=(\mathbf{k}\cdot\mathbf{v}_1, \mathbf{k}\cdot\mathbf{v}_2)=\mathbf{V}^{\textup{T}}\mathbf{k}$, we write the phase function in a simpler form as 
$$\varphi(\mathbf{k})=\tilde{\varphi}(\tilde{\mathbf{k}}):=\sqrt{3+2\cos \tilde{k}_1 +2\cos \tilde{k}_2 +2\cos(\tilde{k}_1-\tilde{k}_2)}.$$
Then, it suffices to show that 
\begin{equation}\label{eq: Hessian of phase in k tilde}
\nabla_{\tilde{\mathbf{k}}}^2\tilde{\varphi}=\frac{1}{\tilde{\varphi}^3}\begin{bmatrix} 
-\tilde{\alpha}_{2}(\tilde{\alpha}_1+\tilde{\alpha}_{12}) & \tilde{\alpha}_{1}\tilde{\alpha}_{2} \\
\tilde{\alpha}_{1}\tilde{\alpha}_{2} & -\tilde{\alpha}_{1}(\tilde{\alpha}_2+\tilde{\alpha}_{12}),
\end{bmatrix}
\end{equation}
where
$$\left\{\begin{aligned}
\tilde{\alpha}_{1}&=1+\cos \tilde{k}_2+\cos(\tilde{k}_1-\tilde{k}_2),\\
\tilde{\alpha}_{2}&=1+\cos \tilde{k}_1+\cos(\tilde{k}_1-\tilde{k}_2),\\
\tilde{\alpha}_{12}&=1+\cos \tilde{k}_1+\cos \tilde{k}_2,
\end{aligned}\right.$$
because together with the identity $\tilde{\varphi}^2=\tilde{\alpha}_1+\tilde{\alpha}_2+\tilde{\alpha}_{12}$, it implies that 
$$\begin{aligned}
\textup{det}(\nabla_{\tilde{\mathbf{k}}}^2\tilde{\varphi})=\frac{\tilde{\alpha}_1\tilde{\alpha}_2}{\tilde{\varphi}^6}\tilde{\alpha}_{12}(\tilde{\alpha}_1+\tilde{\alpha}_2+\tilde{\alpha}_{12})=\frac{\tilde{\alpha}_1\tilde{\alpha}_2\tilde{\alpha}_{12}}{\tilde{\varphi}^4}.
\end{aligned}$$
For the proof of \eqref{eq: Hessian of phase in k tilde}, we compute  
$$\nabla_{\tilde{\mathbf{k}}}\tilde{\varphi}=\bigg(-\frac{\sin \tilde{k}_1+\sin(\tilde{k}_1-\tilde{k}_2)}{\tilde{\varphi}} , -\frac{\sin \tilde{k}_2-\sin(\tilde{k}_1-\tilde{k}_2)}{\tilde{\varphi}}\bigg).$$
For the Hessian matrix, we calculate the second derivatives. Indeed, differentiating $\partial_{\tilde{k}_1}\tilde{\varphi}$, we obtain 
\begin{equation}\label{(1,1) entry}
\partial_{\tilde{k}_1}^2 \tilde{\varphi}=-\frac{(\cos \tilde{k}_1+\cos(\tilde{k}_1-\tilde{k}_2))\tilde{\varphi}^2+(\sin \tilde{k}_1+\sin(\tilde{k}_1-\tilde{k}_2))^2}{\tilde{\varphi}^3}.
\end{equation}
In \eqref{(1,1) entry}, we simplify the numerator replacing all sine functions by cosines, but we also rearrange the terms in order of $c_{12}=\cos(\tilde{k}_1-\tilde{k}_2)$. Indeed, the first term in the numerator can be written as 
\begin{align}\begin{aligned}\label{cos part in (1,1) entry}
\big(\cos \tilde{k}_1+\cos(\tilde{k}_1-\tilde{k}_2)\big)\tilde{\varphi}^2&=(\cos \tilde{k}_1+c_{12})(3+2\cos \tilde{k}_1+2\cos \tilde{k}_2+2c_{12})\\
&=2(c_{12})^2+(3+4\cos \tilde{k}_1+2\cos \tilde{k}_2)c_{12}\\
&\quad+2\cos^2 \tilde{k}_1+(3+2\cos \tilde{k}_2)\cos \tilde{k}_1.
\end{aligned}\end{align}
For the second term, applying the angle-sum formula in the form 
\begin{equation}\label{tricky trig identity}
\sin\theta_1\sin\theta_2=\cos(\theta_1-\theta_2)-\cos\theta_1\cos\theta_2,
\end{equation}
we obtain 
\begin{align}\begin{aligned}\label{sin part in (1,1) entry}
\big(\sin \tilde{k}_1+\sin(\tilde{k}_1-\tilde{k}_2)\big)^2&=\sin^2\tilde{k}_1+\sin^2(\tilde{k}_1-\tilde{k}_2)+2\sin \tilde{k}_1\sin(\tilde{k}_1-\tilde{k}_2)\\
&=(1-\cos^2\tilde{k}_1)+(1-(c_{12})^2)+2\big(\cos \tilde{k}_2-(\cos \tilde{k}_1)c_{12}\big)\\
&=-(c_{12})^2-2(\cos \tilde{k}_1)c_{12}+(-\cos^2\tilde{k}_1+2\cos \tilde{k}_2+2).
\end{aligned}\end{align}
 Hence, plugging \eqref{cos part in (1,1) entry} and \eqref{sin part in (1,1) entry} into \eqref{(1,1) entry}, we obtain that 
$$\begin{aligned}
\partial_{\tilde{k}_1}^2\tilde{\varphi}&=-\frac{(c_{12})^2+(3+2\cos \tilde{k}_1+2\cos \tilde{k}_2)c_{12}+(\cos \tilde{k}_1+1)(\cos \tilde{k}_1+2\cos \tilde{k}_2+2)}{\tilde{\varphi}^3}\\
&=-\frac{(c_{12}+\cos \tilde{k}_1+1)(c_{12}+\cos \tilde{k}_1+2\cos \tilde{k}_2+2)}{\tilde{\varphi}^3}=-\frac{\tilde{\alpha}_2(\tilde{\alpha}_1+\tilde{\alpha}_{12})}{\tilde{\varphi}^3}.
\end{aligned}$$
By symmetry, switching the roles of $\tilde{k}_1$ and $\tilde{k}_2$, we prove that $\partial_{\tilde{k}_2}^2\tilde{\varphi}=-\frac{\tilde{\alpha}_1(\tilde{\alpha}_2+\tilde{\alpha}_{12})}{\tilde{\varphi}^3}$.

It remains to compute $\partial_{\tilde{k}_1}\partial_{\tilde{k}_2}\tilde{\varphi}$. Indeed, we have
$$\partial_{\tilde{k}_1}\partial_{\tilde{k}_2}\tilde{\varphi}=\frac{\cos(\tilde{k}_1-\tilde{k}_2)\tilde{\varphi}^2-(\sin \tilde{k}_1+\sin(\tilde{k}_1-\tilde{k}_2))(\sin \tilde{k}_2-\sin(\tilde{k}_1-\tilde{k}_2))}{\tilde{\varphi}^3}.$$
We again replace sines by cosines in the numerator using \eqref{tricky trig identity}, 
$$\begin{aligned}
&\cos(\tilde{k}_1-\tilde{k}_2)\tilde{\varphi}^2-(\sin \tilde{k}_1+\sin (\tilde{k}_1-\tilde{k}_2))(\sin \tilde{k}_2-\sin (\tilde{k}_1-\tilde{k}_2))\\
&=\cos(\tilde{k}_1-\tilde{k}_2)\tilde{\varphi}^2-\sin \tilde{k}_1\sin \tilde{k}_2+\sin \tilde{k}_1\sin(\tilde{k}_1-\tilde{k}_2)+\sin(\tilde{k}_2-\tilde{k}_1)\sin \tilde{k}_2+\sin^2(\tilde{k}_1-\tilde{k}_2)\\
&=c_{12}(3+2\cos \tilde{k}_1+2\cos \tilde{k}_2+2c_{12})-(c_{12}-\cos \tilde{k}_1\cos \tilde{k}_2)+(\cos \tilde{k}_2-(\cos \tilde{k}_1)c_{12})\\
&\quad+(\cos \tilde{k}_1-c_{12}\cos \tilde{k}_2)+1-(c_{12})^2,
\end{aligned}$$
where $c_{12}=\cos(\tilde{k}_1-\tilde{k}_2)$. Then, we arrange terms in order of $c_{12}$ as follows,
$$\begin{aligned}
\partial_{\tilde{k}_1}\partial_{\tilde{k}_2}\tilde{\varphi}&=\frac{(c_{12})^2+(2+\cos \tilde{k}_1+\cos \tilde{k}_2)c_{12}+(1+\cos \tilde{k}_1)(1+\cos \tilde{k}_2)}{\tilde{\varphi}^3}\\
&=\frac{(c_{12}+1+\cos \tilde{k}_1)(c_{12}+1+\cos \tilde{k}_2)}{\tilde{\varphi}^3}=\frac{\tilde{\alpha}_1\tilde{\alpha}_2}{\tilde{\varphi}^3}.
\end{aligned}$$
Therefore, collecting all, we prove \eqref{eq: Hessian of phase in k tilde}.

\section{Reduction to the oscillatory integral estimate}\label{sec: reduction to the oscillatory integral estimates}

By the factorization structure \eqref{eq: flow factorization}, the core part of the frequency localized Schr\"odinger flow $e^{it\mathbf{\Delta}}P_{\approx\mathbf{K}}\mathbf{u_0}$ is given by the scalar-valued flow $e^{-it\varphi(-i\mathbf{\nabla}_\mathbf{x})}P_{\approx \mathbf{K}}u_0$, that is, the Fourier multiplier such that 
$$(e^{-it\varphi(-i\mathbf{\nabla}_\mathbf{x})}P_{\approx \mathbf{K}}u_0)^\wedge(\mathbf{k})=e^{-it\varphi(\mathbf{k})}\chi_{\approx \mathbf{K}}(\mathbf{k})\hat{u}_0(\mathbf{k}),$$
where $\varphi(\mathbf{k})$ is the phase function given in \eqref{eq: phase function} and $u_0: \mathbf{\Lambda}\to\mathbb{C}$ is scalar. Thus, by the Fourier transform (see Definition \ref{definition: Fourier transform}), the scalar flow has the integral representation
$$(e^{-it\varphi(-i\mathbf{\nabla}_\mathbf{x})}P_{\approx \mathbf{K}}u_0)(\mathbf{x})=\frac{\sqrt{3}}{8\pi^2}\sum_{\mathbf{y}\in\mathbf{\Lambda}}\mathbb{I}_{\mathbf{k}\approx\mathbf{K}}\bigg(t;\frac{\mathbf{x}-\mathbf{y}}{t}\bigg)u_0(\mathbf{y}),$$
where
\begin{equation}\label{eq: localized oscillatory integral}
\boxed{\quad\mathbb{I}_{\mathbf{k}\approx\mathbf{K}}(t;\bz):=\int_{-\infty}^\infty\int_{-\infty}^\infty e^{-it(\varphi(\mathbf{k})-\bz\cdot\mathbf{k})}\chi_{\approx \mathbf{K}}(\mathbf{k})d\mathbf{k}\quad}
\end{equation}
for $\bz\in\mathbb{R}^2$. Note that in \eqref{eq: localized oscillatory integral}, the domain $\mathbb{R}^2/\mathbf{\Lambda}^*$ of integration can be replaced by $\mathbb{R}^2$ extending the cut-off $\chi_{\approx \mathbf{K}}$ trivially. Thus, it suffices to consider the oscillatory integral $\mathbb{I}_{\mathbf{k}\approx\mathbf{K}}(t;\bz)$.

In addition, by symmetries, Theorem \ref{main thm: dispersion estimate} can be further reduced as follows. First, we may assume that $\mathbf{K}$ is contained in the primitive rhombic cell $\mathcal{B}_{\textup{rhom}}$ (see \eqref{rhombic cell} and the gray region in Figure \ref{fig: periodic structure of a hexagonal lattice} (B)). Recalling the definition of the phase $\varphi(\mathbf{k})$ (see \eqref{eq: phase function}), we observe that for the integral $\mathbb{I}_{\mathbf{k}\approx\mathbf{K}}(t;\mathbf{v})$, the change of the variables $\mathbf{k}\mapsto\tilde{\mathbf{k}}$ by
$$\begin{bmatrix}\mathbf{k}\cdot\mathbf{v}_1\\\mathbf{k}\cdot\mathbf{v}_2\end{bmatrix}=\begin{bmatrix}\tilde{\mathbf{k}}\cdot\mathbf{v}_2\\\tilde{\mathbf{k}}\cdot\mathbf{v}_1\end{bmatrix}\quad\textup{or}\quad \begin{bmatrix}\mathbf{k}\cdot\mathbf{v}_1\\\mathbf{k}\cdot\mathbf{v}_2\end{bmatrix}=-\begin{bmatrix}\tilde{\mathbf{k}}\cdot\mathbf{v}_1\\\tilde{\mathbf{k}}\cdot\mathbf{v}_2\end{bmatrix}$$
does not change the structure of the integral except that $\mathbf{K}$ and $\mathbf{v}$ are relocated. Therefore, $\mathbf{K}$ can be moved to the first quadrant $\mathbb{R}_{\geq0}^2$. Similarly, by the change of the variables by
$$\begin{bmatrix}\mathbf{k}\cdot\mathbf{v}_1\\\mathbf{k}\cdot\mathbf{v}_2\end{bmatrix}=\begin{bmatrix}\tilde{\mathbf{k}}\cdot\mathbf{v}_2\\\tilde{\mathbf{k}}\cdot\mathbf{v}_1\end{bmatrix}\quad\textup{or}\quad \begin{bmatrix}\mathbf{k}\cdot(\mathbf{v}_1-\mathbf{v}_2)\\\mathbf{k}\cdot\mathbf{v}_2\end{bmatrix}=\begin{bmatrix}\tilde{\mathbf{k}}\cdot\mathbf{v}_1\\-\tilde{\mathbf{k}}\cdot\mathbf{v}_2\end{bmatrix},$$
we may switch the roles of $\alpha_{1}(\mathbf{k})$, $\alpha_{2}(\mathbf{k})$ and $\alpha_{12}(\mathbf{k})$. Therefore, for degenerate frequencies (see \eqref{eq: decomposition of the frequency domain}), it is enough to consider the reduced cases; 
\begin{equation}\label{eq: reduced decomposition of the frequency domain}
\boxed{\quad\begin{aligned}
\mathcal{K}_1':&=\big\{\mathbf{K}\in\mathcal{B}_{\textup{rhom}}\cap\mathbb{R}_{\geq0}^2: \alpha_{1}(\mathbf{K})=0\big\}\setminus (\mathcal{K}_2\cup\mathcal{K}_3)\\
\mathcal{K}_2':&=\big\{\mathbf{K}_2=(0,\pi)\big\}\quad\textup{(intersection of two curves)}\\
\mathcal{K}_3':&=\big\{\mathbf{K}_\star=(0,\tfrac{4\pi}{3})\big\}\quad\textup{(Dirac point)}.
\end{aligned}\quad}
\end{equation}

In conclusion, the proof of the main theorem (Theorem \ref{main thm: dispersion estimate}) is reduced to that of the following oscillatory integral estimate.

\begin{theorem}[Oscillatory integral estimate]\label{thm: oscillatory integral estimates}
For each $\mathbf{K}\in \mathcal{K}_j'$ with $j=1,2,3$, there exists a compactly supported smooth function $\chi_{\approx \mathbf{K}}$ such that $\chi_{\approx \mathbf{K}}=1$ in a sufficiently small neighborhood of $\mathbf{K}$ and $$\sup_{\mathbf{z}\in\mathbb{R}^2}|\mathbb{I}_{\mathbf{k}\approx\mathbf{K}}(t;\bz)|\lesssim\frac{1}{(1+|t|)^{a_j}},$$
where $a_j$ is given by \eqref{eq: dispersion estimate rate}.
\end{theorem}

The non-degenerate case $j=0$ can be excluded (see Remark \ref{remark: non-degenerate frequency case}). The next three sections are devoted to the proof of Theorem \ref{thm: oscillatory integral estimates} for $j=1,2,3$. Before proceeding, we introduce several auxiliary tools.

\section{Oscillatory integral localized at a Dirac point} \label{sec: dispersion near a Dirac point}

In this section, we prove Theorem \ref{thm: oscillatory integral estimates} for $\mathbf{K}=\mathbf{K}_\star=(0,\frac{4\pi}{3})\in\mathcal{K}_3'$ (see Figure \ref{fig: first primitive cell}). 
We further show that the integral $\mathbb{I}_{\mathbf{k}\approx\mathbf{K}}(t;\bz)$ decays faster away from certain critical directions of the parameter $\bz$.

\begin{theorem}[Oscillatory integral estimate; $\mathbf{K}=\mathbf{K}_\star$]\label{thm: near a Dirac point}
Suppose that $\mathbf{K}_\star=(0,\frac{4\pi}{3})$, and let $\chi_{\approx \mathbf{K}_\star}(\mathbf{k}):=\chi_0(\frac{|\mathbf{k}-\mathbf{K}_\star|}{\delta})$ in the integral \eqref{eq: localized oscillatory integral}, where $\delta>0$ is a sufficiently small number and $\chi_0$ is defined in \eqref{eq: chi0}. Then, there exists $C_3>0$ such that
\begin{equation}\label{eq: near a Dirac point 1}
\sup_{\bz\in\mathbb{R}^2}|\mathbb{I}_{\mathbf{k}\approx\mathbf{K}_\star}(t;\bz)|\leq\frac{C_3}{(1+|t|)^{5/6}}.
\end{equation}

Moreover, for any small $\epsilon>0$, let $\Theta_\epsilon=\{ (\cos\theta,\sin\theta)\in \mathbb{S}^1: |\theta-\frac{n\pi}{3}| < \epsilon \text{ for some } n\in \mathbb{Z}\}$ (See figure \eqref{fig: Bad direction} below). Then, 
there exists $C_3'(\epsilon)>0$, with $C_3'(\epsilon)\to \infty$ as $\epsilon\to0$, such that  
\begin{equation}\label{eq: near a Dirac point 2}
\sup_{\frac{\bz}{|\bz|}\, \notin \, \Theta_\epsilon} |\mathbb{I}_{\mathbf{k}\approx\mathbf{K}_\star}(t;\bz)|\leq\frac{C_3'(\epsilon)}{1+|t|}.
\end{equation} 
\end{theorem}

\begin{figure}[h]
\begin{center}
\adjustbox{center}{
\begin{tikzpicture}[scale=2.8]
\draw (0,0) circle(1);

\draw[->] (-1.2,0) -- (1.3,0);
\draw[->] (0,-1.15) -- (0,1.15);

\foreach \angle in {0, 60, 120, 180, 240, 300}
    \draw[dashed, blue,thick] (0,0) -- ({cos(\angle)}, {sin(\angle)});

\draw[blue] (0.25,0) arc (0:59:0.25);
\draw[blue] ({0.25*cos(60)},{0.25*sin(60)}) arc (59:119:0.25);
\node[blue] at (0.35,0.15) {$\frac{\pi}{3}$};
\node[blue] at (0.1,0.35) {$\frac{\pi}{3}$};

\foreach \angle in {0, 60, 120, 180, 240,300}
    \draw[line width=3pt, red] ({cos(\angle-5)},{sin(\angle-5)}) arc(\angle-5:\angle+5:1);
    
\draw[<->, red, thick] (1.1,0) -- (1.1,0.1);

\node[right] at (1.02,-0.5) {$\mathbb{S}^1$};
\node[red] at (1.2,0.05) {$\epsilon$};
\node[red] at (0.7,0.95) {\large$\boldsymbol{\Theta_\epsilon}$};
\end{tikzpicture}
}
\end{center}  
\caption{$\Theta_\epsilon$: Directions for slower decay of $\mathbb{I}_{\mathbf{k}\approx\mathbf{K}_\star}(t;\bz)$ in Theorem~\ref{thm: near a Dirac point}.}
    \label{fig: Bad direction}
\end{figure}

\begin{remark}
$(i)$ By the reduction in Section \ref{sec: reduction to the oscillatory integral estimates}, the oscillatory integral estimate \eqref{eq: near a Dirac point 1} implies the desired dispersion estimate \eqref{eq: dispersion estimate} for all $\mathbf{K}\in\mathcal{K}_3$.\\
$(ii)$ By \eqref{eq: near a Dirac point 2}, if $\mathbf{u}_0$ has nonzero value only at one point $\mathbf{x}_0\in\mathbf{\Lambda}$, then the flow decays faster 
$$|\mathbf{O}(-i\nabla_{\mathbf{x}})^*e^{it\mathbf{\Delta}}\mathbf{O}(-i\nabla_{\mathbf{x}})P_{\approx\mathbf{K}_\star}\mathbf{u}_0(\mathbf{x})|\lesssim\frac{1}{1+|t|}|\mathbf{u}_0(\mathbf{x}_0)|$$
when $\mathbf{x}$ is away from the three straight lines $\mathbf{x}=\mathbf{x}_0+v(\cos\frac{\pi n}{3},\sin\frac{\pi n}{3})$ with $v>0$.
\end{remark}

\subsection{Reduction to the degenerate oscillatory integral estimate}

For the integral \eqref{eq: localized oscillatory integral} with
$\chi_{\approx \mathbf{K}}=\chi_0(\frac{|\,\cdot\,-\mathbf{K}_\star|}{\delta})$, where $\chi_0$ is defined in \eqref{eq: chi0}, by a simple change of variables by translation, we have 
$$\mathbb{I}_{\mathbf{k}\approx\mathbf{K}_\star}(t;\bz)=e^{it\mathbf{K}_\star\cdot\bz}\,\mathbb{I}_{
\mathbf{K}_\star}(t;|\bz|,\theta),$$
where $\bz=z(\cos\theta,\sin\theta)$ with $z>0$ and $\theta\in[0,2\pi)$, and 
$$\mathbb{I}_{\mathbf{K}_\star}(t;\oz,\theta):=\int_{-\infty}^\infty\int_{-\infty}^\infty e^{-it(\varphi(\mathbf{K}_\star+\mathbf{k}')-\oz(\cos\theta,\sin\theta)\cdot\mathbf{k}')}\chi_0\bigg(\frac{|\mathbf{k}'|}{\delta}\bigg)d\mathbf{k}'.$$
Thus, Theorem \ref{thm: near a Dirac point} can be reformulated as follows.

\begin{proposition}[Reformulation of Theorem \ref{thm: near a Dirac point}]\label{prop: near a Dirac point}
$$\begin{aligned}
\sup_{\oz >0\textup{ and }\theta\,\in\,[0,2\pi)} |\mathbb{I}_{\mathbf{K}_\star}(t;\oz ,\theta)|&\leq C_3(1+|t|)^{-\frac{5}{6}},\\
 \sup_{\oz >0\textup{ and } 
|\theta-\frac{n\pi}{3}|\ge\epsilon \text{ for all } n\in\Z}|\mathbb{I}_{\mathbf{K}_\star}\big(t; \oz , \theta)|&\leq C_3'(1+|t|)^{-1}.
\end{aligned}$$
\end{proposition}

For the integral $\mathbb{I}_{\mathbf{K}_\star}(t; \oz , \theta)$, we change the variable by the rotation $\mathbf{k}'=\mathbf{R}_\theta\tbk=(\tko\cos\theta-\tkt\sin\theta, \tko\sin\theta+\tkt\cos\theta)$ to convert $\oz(\cos\theta,\sin\theta)\cdot \mathbf{k}'$ into $\oz \tko$, removing the linear term in the $\tkt$-direction.

\begin{remark}
In our analysis, the linear component $\bz\cdot \mathbf{k}=\oz_1k_1+\oz_2k_2$ in the phase is the most problematic, because even when $\bz$ is small, the linear term still dominates higher-order terms for small $\mathbf{k}$. By rotation, one of the linear terms, $\oz_2k_2$, is removed. Then, the $k_2$-direction becomes the good direction for integration by parts, while the $k_1$-direction is the bad one.
\end{remark}

Subsequently, we decompose the integral as
\begin{equation}\label{eq: Dirac oscillatory integral decomposition}
\begin{aligned}
\mathbb{I}_{\mathbf{K}_\star}(t; \oz, \theta)&=
\int_{-\infty}^\infty\int_{-\infty}^\infty e^{-it(\varphi(\mathbf{K}_\star+\mathbf{R}_\theta\tbk)-\oz \tko)}\chi_0\bigg(\frac{|\tbk|}{\delta}\bigg)\chi_1\bigg(\frac{\tkt}{\delta \tko}\bigg)d\tko d\tkt\\
&\quad+
\int_{-\infty}^\infty\int_{-\infty}^0 e^{-it(\varphi(\mathbf{K}_\star+\mathbf{R}_\theta\tbk)-\oz\tko)}\chi_0\bigg(\frac{|\tbk|}{\delta}\bigg)\chi_0\bigg(\frac{\tkt}{\delta \tko}\bigg)d\tko d\tkt\\
&\quad+
\int_{-\infty}^\infty\int_0^\infty e^{-it(\varphi(\mathbf{K}_\star+\mathbf{R}_\theta\tbk)-\oz \tko)}\chi_0\bigg(\frac{|\tbk|}{\delta}\bigg)\chi_0\bigg(\frac{\tkt}{\delta \tko}\bigg)d\tko d\tkt\\
&=:\mathbb{A}_{\mathbf{K}_\star}(t; \oz, \theta)+\mathbb{B}_{\mathbf{K}_\star}(t;\oz, \theta)+\mathbb{C}_{\mathbf{K}_\star}(t; \oz, \theta),
\end{aligned}
\end{equation}
where $\chi_0$ and $\chi_1$ are the smooth cut-offs given by \eqref{eq: chi0} and \eqref{eq: chi1}. The figure on the left in Figure \ref{fig: NearDiracCase} describes the decomposition of the frequency domain. The domain of integration for $\mathbb{C}_{\mathbf{K}_\star}(t; \oz, \theta)$ is the shaded region near the positive $\tilde{k}_1$-axis, and that for $\mathbb{B}_{\mathbf{K}_\star}(t; \oz, \theta)$ is the shaded region near the negative $\tilde{k}_1$-axis. The remainder is for $\mathbb{A}_{\mathbf{K}_\star}(t; \oz, \theta)$.
The purpose of this decomposition is to separate the directions where the phase is non-stationary, that is, where the first-order derivative of the phase function does not vanish.

\begin{figure}
    \centering
    \includegraphics[width=1\linewidth]{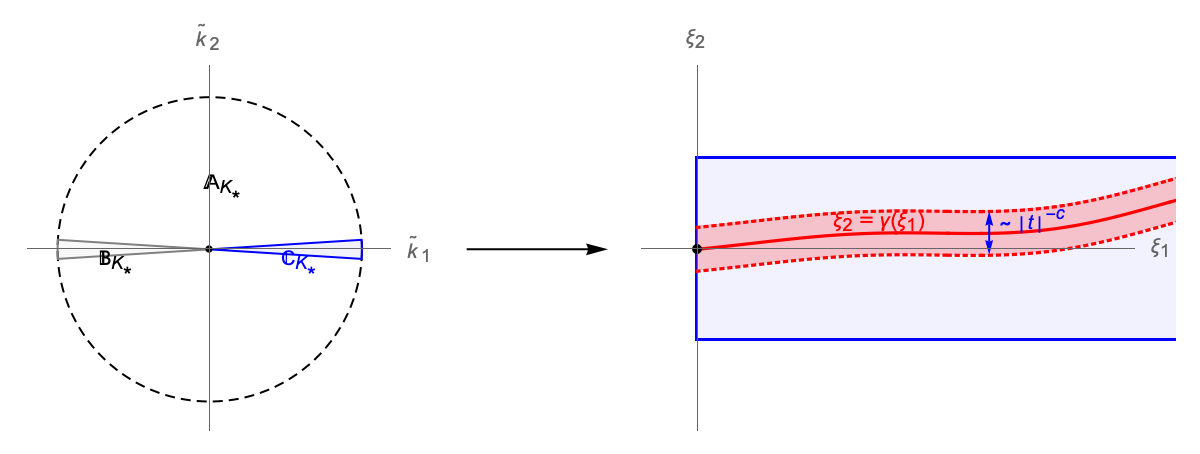}
    \caption{Frequency domains  in \eqref{eq: Dirac oscillatory integral decomposition}, change of variables in \eqref{change of variable}, analytic curve $\gamma$ in Lemma~\ref{lemma: construction of curve gamma}, and truncation in \eqref{eq: C Dirac decomposition} }
    \medskip
    \label{fig: NearDiracCase}
\end{figure}

For the first two components in the decomposition, the desired bounds can be shown by the non-stationary phase estimate, in other words, by integration by parts (Lemma \ref{lemma: non-stationary estimates near the Dirac point});  see also  Remark~\ref{rem:decomposition of integrals}. Subsequently, the proof of Proposition \ref{prop: near a Dirac point} is reduced to showing the bound for the integral $\mathbb{C}_{\mathbf{K}_\star}(t; \oz, \theta)$ (see Proposition \ref{prop: refined degenerate oscillatory integral near the Dirac point}).

\begin{lemma}[Bounds for $\mathbb{A}_{\mathbf{K}_\star}(t; \oz, \theta)$ and $\mathbb{B}_{\mathbf{K}_\star}(t; \oz, \theta)$]\label{lemma: non-stationary estimates near the Dirac point}
$$|\mathbb{A}_{\mathbf{K}_\star}(t; \oz, \theta)|+|\mathbb{B}_{\mathbf{K}_\star}(t; \oz, \theta)|\lesssim (1+|t|)^{-1}.$$
\end{lemma}

For the proof, we employ the asymptotic expansion of $\varphi(\mathbf{K}_\star+\mathbf{R}_\theta\tbk)^2$.

\begin{lemma}[Phase function asymptotic]\label{lemma: phase function asymptotic near the Dirac point}
For $\tbk\approx\mathbf{0}$, we have
\begin{equation}\label{eq: phase function square asymptotic near the Dirac point}
\varphi(\mathbf{K}_\star+\mathbf{R}_\theta\tbk)^2=\frac{3}{4}\bigg(|\tbk|^2+a_\theta \tko^3-3b_\theta \tko^2\tkt-3a_\theta \tko \tkt^2+b_\theta \tkt^3-\frac{1}{16}|\tbk|^4+\mathcal{O}_5(\tbk)\bigg),
\end{equation}
where $a_\theta=\frac{\sin\theta(3\cos^2\theta-\sin^2\theta)}{2\sqrt{3}}$, $b_\theta=\frac{\cos\theta(3\sin^2\theta-\cos^2\theta)}{2\sqrt{3}}$, and $\mathcal{O}_5(\tbk)$ denotes an analytic function such that $|\mathcal{O}_5(\tbk)|\lesssim|\tbk|^5$ near the origin (see \eqref{eq: rm}). Moreover, the coefficients $a_\theta$ and $b_\theta$ satisfy 
\begin{equation}\label{eq: a b relation}
a_\theta^2+b_\theta^2=\frac{1}{12}.
\end{equation}
\end{lemma}

\begin{proof}
Recalling $\mathbf{K}_\star=(0,\tfrac{4\pi}{3})$, $\mathbf{v}_1=(\frac{\sqrt{3}}{2},\frac{1}{2})$ and $\mathbf{v}_2=(\frac{\sqrt{3}}{2},-\frac{1}{2})$, we substitute $\mathbf{K}_\star+\mathbf{k'}$ into \eqref{eq: phase function} and, by the angle sum formula, expand it for $\mathbf{k}'\approx \mathbf{0}$
$$\begin{aligned}
&\varphi(\mathbf{K}_\star+\mathbf{k}')^2\\
&=3+2\cos((\mathbf{K}_\star+\mathbf{k}')\cdot\mathbf{v}_1) +2\cos((\mathbf{K}_\star+\mathbf{k}')\cdot\mathbf{v}_2) +2\cos((\mathbf{K}_\star+\mathbf{k}')\cdot(\mathbf{v}_1-\mathbf{v}_2))\\
&=3+2\cos\bigg(\frac{\sqrt{3}k_1'+k_2'}{2}+\frac{2\pi}{3}\bigg)+2\cos\bigg(\frac{\sqrt{3}k_1'-k_2'}{2}-\frac{2\pi}{3}\bigg)+2\cos\bigg(k_2'+\frac{4\pi}{3}\bigg)\\
&=3-\cos\bigg(\frac{\sqrt{3}k_1'+k_2'}{2}\bigg)-\cos\bigg(\frac{\sqrt{3}k_1'-k_2'}{2}\bigg)-\cos k_2'\\
&\quad-\sqrt{3}\bigg\{\sin\bigg(\frac{\sqrt{3}k_1'+k_2'}{2}\bigg)-\sin\bigg(\frac{\sqrt{3}k_1'-k_2'}{2}\bigg)-\sin k_2'\bigg\}.
\end{aligned}$$
Then, the Taylor series for the sine and the cosine functions yield
$$\begin{aligned}
\varphi(\mathbf{K}_\star+\mathbf{k}')^2&=\frac{(\frac{\sqrt{3}k_1'+k_2'}{2})^2}{2}+\frac{(\frac{\sqrt{3}k_1'-k_2'}{2})^2}{2}+\frac{(k_2')^2}{2}-\frac{(\frac{\sqrt{3}k_1'+k_2'}{2})^4}{24}-\frac{(\frac{\sqrt{3}k_1'-k_2'}{2})^4}{24}-\frac{(k_2')^4}{24}+\cdots\\
&\quad-\sqrt{3}\bigg\{-\frac{(\frac{\sqrt{3}k_1'+k_2'}{2})^3}{6}+\frac{(\frac{\sqrt{3}k_1'-k_2'}{2})^3}{6}+\frac{(k_2')^3}{6}+\cdots\bigg\}\\
&=\cdots=\frac{3|\mathbf{k}'|^2}{4}+\frac{3\sqrt{3}}{8}(k_1')^2k_2'-\frac{\sqrt{3}}{8}(k_2')^3-\frac{3}{64}|\mathbf{k}'|^4+\mathcal{O}_5(\mathbf{k}').
\end{aligned}$$
Thus, it follows that 
$$\begin{aligned}
\varphi(\mathbf{K}_\star+\mathbf{R}_\theta\tbk)^2&=\frac{3|\tbk|^2}{4}+\frac{3\sqrt{3}}{8}(\tko\cos\theta-\tkt\sin\theta)^2(\tko\sin\theta+\tkt\cos\theta)\\
&\quad-\frac{\sqrt{3}}{8}(\tko\sin\theta+\tkt\cos\theta)^3-\frac{3}{64}|\tbk|^4+\mathcal{O}_5(\tbk),
\end{aligned}$$
where $\mathcal{O}_5(\mathbf{R}_\theta\tbk)$ is still denoted by $\mathcal{O}_5(\tbk)$ with an abuse of notation. Subsequently, expanding the products and rearranging terms in order, we prove the desired asymptotic formula \eqref{eq: phase function square asymptotic near the Dirac point}. Moreover, by direct calculations, one can show that $a_\theta^2+b_\theta^2=\frac{\sin^2\theta(3\cos^2\theta-\sin^2\theta)^2}{12}+\frac{\cos^2\theta(3\sin^2\theta-\cos^2\theta)^2}{12}=\frac{(\cos^2\theta+\sin^2\theta)^3}{12}=\frac{1}{12}$.
\end{proof}

\begin{remark}\label{rem:decomposition of integrals}
For $\mathbb{A}_{\mathbf{K}_\star}(t; \oz, \theta)$ and $\mathbb{B}_{\mathbf{K}_\star}(t; \oz, \theta)$, Lemma \ref{lemma: phase function asymptotic near the Dirac point} is good enough, because we can estimate them by integration by parts. Indeed, in the integral  $\mathbb{A}_{\mathbf{K}_\star}(t; \oz, \theta)$, the leading-order term of the $\tkt$-directional derivative of the phase $\varphi(\mathbf{K}_\star+\mathbf{R}_\theta\tbk)-\oz \tko$, i.e., $\frac{\sqrt{3}\tkt}{2|\tbk|}$, dominates the other higher-order terms. 
For $\mathbb{B}_{\mathbf{K}_\star}(t; \oz, \theta)$, the leading-order term of the $\tko$-directional derivative is given by $\frac{\sqrt{3}\tko}{2|\tbk|}-\oz$, and $-\oz$ has a favorable sign near the negative $\tko$-axis in the sense that it ensures the derivative remains negative and bounded away from zero.
\end{remark}

\begin{proof}[Proof of Lemma \ref{lemma: non-stationary estimates near the Dirac point}]
It suffices to show that $|\mathbb{A}_{\mathbf{K}_\star}(t; \oz, \theta)|+|\mathbb{B}_{\mathbf{K}_\star}(t; \oz, \theta)|\lesssim_\delta |t|^{-1}$, because it is obvious that $|\mathbb{A}_{\mathbf{K}_\star}(t; \oz, \theta)|+|\mathbb{B}_{\mathbf{K}_\star}(t; \oz, \theta)|\lesssim_\delta1$. For $\mathbb{A}_{\mathbf{K}_\star}(t; \oz, \theta)$, we observe from Lemma \ref{lemma: phase function asymptotic near the Dirac point} that  in the integral $\mathbb{A}_{\mathbf{K}_\star}(t; \oz, \theta)$, 
\begin{equation}\label{eq: dirac phase first derivative 1}
\big|\partial_{\tkt}\big\{\varphi(\mathbf{K}_\star+\mathbf{R}_\theta\tbk)\big\}\big|=\frac{|\partial_{\tkt}\{\varphi(\mathbf{K}_\star+\mathbf{R}_\theta\tbk)^2\}|}{2\varphi(\mathbf{K}_\star+\mathbf{R}_\theta\tbk)}=\frac{|\frac{3}{2}\tkt+\mathcal{O}_2(\tbk)|}{\sqrt{3|\tbk|^2+\mathcal{O}_3(\tbk)}}\sim_\delta1,
\end{equation}
because we have $|\tkt|\sim_\delta|\tbk|$ due to the cut-off $\chi_1(\frac{\tkt}{\delta \tko})$ (see \eqref{eq: chi1} for the definition of $\chi_1$). 
By integration by parts, it follows that
$$\mathbb{A}_{\mathbf{K}_\star}(t; \oz, \theta)=\frac{1}{it}\int_{-\infty}^\infty\int_{-\infty}^\infty e^{-it(\varphi(\mathbf{K}_\star+\mathbf{R}_\theta\tbk)-\oz \tko)}\partial_{\tkt}\Bigg\{\frac{\chi_0(\frac{|\tbk|}{\delta})\chi_1(\frac{\tkt}{\delta \tko})}{\partial_{\tkt}(\varphi(\mathbf{K}_\star+\mathbf{R}_\theta\tbk))}\Bigg\}d\tkt d\tko.$$
Note that in the integral, by the identity $f''=\frac{(f^2)''-2(f')^2}{2f}$, Lemma \ref{lemma: phase function asymptotic near the Dirac point} and \eqref{eq: dirac phase first derivative 1}, 
\begin{equation}\label{eq: dirac phase second derivative 1}
\begin{aligned}
\big|\partial_{\tkt}^2\big\{\varphi(\mathbf{K}_\star+\mathbf{R}_\theta\tbk)\big\}\big|&=\frac{|\partial_{\tkt}^2\{\varphi(\mathbf{K}_\star+\mathbf{R}_\theta\tbk)^2\}-2\{\partial_{\tkt}(\varphi(\mathbf{K}_\star+\mathbf{R}_\theta\tbk))\}^2|}{2\varphi(\mathbf{K}_\star+\mathbf{R}_\theta\tbk)}\\
&=\frac{|\frac{3}{2}+\mathcal{O}_1(\tbk)-2\{\partial_{\tkt}(\varphi(\mathbf{K}_\star+\mathbf{R}_\theta\tbk))\}^2|}{\sqrt{3|\tbk|^2+\mathcal{O}_3(\tbk)}}\lesssim_\delta\frac{1}{|\tbk|},
\end{aligned}
\end{equation}
and thus, by \eqref{eq: dirac phase first derivative 1} and \eqref{eq: dirac phase second derivative 1}, 
$$\begin{aligned}
\bigg|\partial_{\tkt}\bigg\{\frac{\chi_0(\frac{|\tbk|}{\delta})\chi_1(\frac{\tkt}{\delta \tko})}{\partial_{\tkt}(\varphi(\mathbf{K}_\star+\mathbf{R}_\theta\tbk))}\bigg\}\bigg|&\leq\frac{\frac{1}{\delta}|\chi_0'(\frac{|\tbk|}{\delta})|}{|\partial_{\tkt}(\varphi(\mathbf{K}_\star+\mathbf{R}_\theta\tbk))|}+\frac{|\chi_0'(\frac{\tkt}{\delta \tko})|\frac{1}{\delta|\tko|}}{|\partial_{\tkt}(\varphi(\mathbf{K}_\star+\mathbf{R}_\theta\tbk))|}\\
&\quad+\frac{|\partial_{\tkt}^2(\varphi(\mathbf{K}_\star+\mathbf{R}_\theta\tbk))|}{|\partial_{\tkt}(\varphi(\mathbf{K}_\star+\mathbf{R}_\theta\tbk))|^2}\\
&\lesssim_\delta 1+\frac{1}{|\tbk|}\lesssim\frac{1}{|\tkt|},
\end{aligned}$$
since $\chi_1'=-\chi_0'$ and $\frac{1}{\delta|\tko|}\sim_\delta\frac{1}{|\tbk|}$ in the support of $\chi_0'(\frac{\tkt}{\delta \tko})$.  
Note also that $\chi_0(\frac{|\tbk|}{\delta})\chi_1(\frac{\tkt}{\delta \tko})$ is supported in $\{(\tko, \tkt): |\tkt|\leq2\delta\textup{ and }|\tkt|\geq \frac{\delta}{2} |\tko|\}$ (see \eqref{eq: chi0} and \eqref{eq: chi1} for the definition of $\chi_0$ and $\chi_1$). Therefore, we prove that
$$|\mathbb{A}_{\mathbf{K}_\star}(t; z, \theta)|\lesssim_\delta\frac{1}{|t|}\int_{-2\delta}^{2\delta}\int_{-\frac{2}{\delta}|\tkt|}^{\frac{2}{\delta}|\tkt|} \frac{1}{|\tkt|}d\tko d\tkt \sim\frac{1}{|t|}.$$

Similarly, for $\mathbb{B}_{\mathbf{K}_\star}(t; \oz, \theta)$, by integration by parts with
$$\big|\partial_{\tko}\big\{\varphi(\mathbf{K}_\star+\mathbf{R}_\theta\tbk)-\oz \tko \big\}\big|=\cdots=\bigg|\frac{\frac{3}{2}\tko+\mathcal{O}_2(\tbk)}{\sqrt{3|\tbk|^2+\mathcal{O}_3(\tbk)}}-\oz\bigg|\gtrsim 1$$
(since $(-\tko)\sim |\tbk|$, $\tko<0$ and $\oz>0$ in the integral), we write 
$$\mathbb{B}_{\mathbf{K}_\star}(t; \oz, \theta)=\frac{1}{it}\int_{-\infty}^0\int_{-\infty}^\infty e^{-it(\varphi(\mathbf{K}_\star+\mathbf{R}_\theta\tbk)-\oz\tko)}\partial_{\tko}\bigg\{\frac{\chi_0(\frac{|\tbk|}{\delta})\chi_0(\frac{\tkt}{\delta \tko})}{\partial_{\tko}(\varphi(\mathbf{K}_\star+\mathbf{R}_\theta\tbk)-\oz\tko)}\bigg\}d\tkt d\tko.$$
Repeating \eqref{eq: dirac phase second derivative 1}, one can show that $|\partial_{\tko}^2\{\varphi(\mathbf{K}_\star+\mathbf{R}_\theta\tbk)\}|\lesssim_\delta\frac{1}{|\tbk|}$. Then, estimating as before but switching the roles of $\tko$ and $\tkt$, one can show that $|\mathbb{B}_{\mathbf{K}_\star}(t; \oz, \theta)|\lesssim_\delta|t|^{-1}$.
\end{proof}

\subsection{Preliminary degenerate oscillatory integral estimate}
For Proposition \ref{prop: near a Dirac point}, by Lemma \ref{lemma: non-stationary estimates near the Dirac point} and \eqref{eq: Dirac oscillatory integral decomposition}, it suffices to consider the integral $\mathbb{C}_{\mathbf{K}_\star}(t; \oz, \theta)$. 

In this subsection, for the reader's convenience, 
we present a primitive weaker decay estimate for $\mathbb{C}_{\mathbf{K}_\star}(t; \oz, \theta)$ motivated by the argument for the dispersion estimate for wave-type equations \cite{Oh, Stein}, and our argument relies solely on the oscillations arising from the leading-order term of the phase function.
In the next subsection, we derive a faster decay estimate, referred to as the refined degenerate oscillatory integral estimate, by incorporating the additional oscillation effects arising from higher-order terms in the phase function.

\begin{proposition}[Preliminary degenerate oscillatory integral estimate]\label{prop: preliminary degenerate oscillatory integral near the Dirac point}
\begin{equation}\label{eq:preliminary degenerate oscillatory integral}
  |\mathbb{C}_{\mathbf{K}_\star}(t; \oz, \theta)|\lesssim_\delta (1+|t|)^{-\frac{1}{2}}.  
\end{equation}
\end{proposition}

\begin{remark}
The proof of the proposition is much simpler than the refined bound (Proposition \ref{prop: refined degenerate oscillatory integral near the Dirac point}). This simpler 2D-wave-like $O(|t|^{-1/2})$ decay bound would be good enough to investigate the connection between the Schr\"odinger equation on a honeycomb lattice and the Dirac equation via the continuum limit.
\end{remark}

First, we note that the phase function in $\mathbb{C}_{\mathbf{K}_\star}(t; \oz, \theta)$  is not analytic at the origin. Thus, changing the variables by
\begin{equation}\label{change of variable}
(\xo, \xt) =\bigg(\tko,\frac{\tkt}{\tko} \bigg),
\end{equation}
we modify the integral as 
\begin{equation}\label{eq: modified Dirac I3}
\mathbb{C}_{\mathbf{K}_\star}(t; \oz, \theta)=
\int_{-\infty}^\infty\int_0^\infty e^{-it\psi(\bxi; \oz,\theta)} \tilde{\chi}_\delta(\bxi)\xo d\xo d\xt,
\end{equation}
where
\begin{equation}\label{eq: Dirac modified phase}
\psi(\bxi; z,\theta):=\varphi\Big(\mathbf{K}_\star+\mathbf{R}_\theta[\xo, \xo\xt]^{\textup{T}}\Big)-\oz\xo
\end{equation}
and
$$\tilde{\chi}_\delta(\bxi)=\chi_0\bigg(\frac{\xo\sqrt{1+\xt^2}}{\delta}\bigg)\chi_0\bigg(\frac{\xt}{\delta}\bigg).$$
The restriction of $\psi$ to $\xi_1\ge0$ admits an analytic
extension to a neighborhood of the origin. In what follows,
$\psi$ denotes this analytic extension.

\begin{remark}\label{remark: chi delta support}
$\tilde{\chi}_\delta$ is smooth, and it is supported in $[-10\delta, 10\delta]^2$, since $\sqrt{1+\xt^2}\approx 1$ in the support of $\chi_0(\frac{\xt}{\delta})$.\\
\end{remark}

By this modification, the integral in \eqref{eq: modified Dirac I3} has an analytic phase function. Its expansion is given as follows.

\begin{lemma}[Modified phase function near the origin]\label{lemma: modified phase function asymptotic near the Dirac point}
$\psi(\bxi; \oz,\theta)$ is analytic near the origin, and
$$\begin{aligned}
\psi(\bxi; \oz,\theta)&=\frac{\sqrt{3}\xo}{2}\bigg\{1-\frac{2}{\sqrt{3}}\oz+\frac{1}{2}\xt^2+\bigg(\frac{a_\theta}{2}-\frac{3b_\theta}{2} \xt-\frac{7a_\theta}{4}  \xt^2+\frac{5b_\theta}{4} \xt^3\bigg)\xo \\
&\qquad\qquad+\bigg(-\frac{1}{32}-\frac{a_\theta^2}{8}+\frac{3a_\theta b_\theta}{4}\xt+\frac{3(44a_\theta^2-3)}{64} \xt^2-\frac{29}{8}a_\theta b_\theta \xt^3\bigg)\xo^2\\
&\qquad\qquad+\xo^3\mathcal{O}_0(\bxi)+\xt^4\tilde{\mathcal{O}}_0(\bxi)\bigg\},
\end{aligned}$$
where $a_\theta$ and $b_\theta$ are the coefficients given in Lemma \ref{lemma: phase function asymptotic near the Dirac point}, and $\mathcal{O}_0(\bxi)$ and $\tilde{\mathcal{O}}_0(\bxi)$ denote analytic functions near the origin such that $|\mathcal{O}_0(\bxi)|, |\tilde{\mathcal{O}}_0(\bxi)|\lesssim1$ (see \eqref{eq: rm}).
\end{lemma}

\begin{proof}
By Lemma \ref{lemma: phase function asymptotic near the Dirac point}, we write 
$$\begin{aligned}
&\varphi\Big(\mathbf{K}_\star+\mathbf{R}_\theta[\xo, \xo\xt]^{\textup{T}}\Big)^2\\
&=\frac{3\xo^2}{4}\bigg\{1+\xt^2+a_\theta \xo-3b_\theta \xo \xt -3a_\theta \xo \xt^2+b_\theta \xo \xt^3-\frac{1}{16}\xo^2(1+\xt^2)^2+\frac{\mathcal{O}_5(\xo, \xo \xt)}{\xo^2}\bigg\}\\
&=\frac{3\xo^2}{4}\big(1+\xt^2+A_1\xo+A_2\xo^2+\xo^3 \mathcal{O}_0(\bxi)\big),
\end{aligned}$$
where
$$A_1=a_\theta-3b_\theta \xt-3a_\theta  \xt^2+b_\theta \xt^3\quad\textup{and}\quad A_2=-\frac{1}{16}(1+\xt^2)^2.$$
In the above identity, we used that $\frac{\mathcal{O}_5(\xo, \xo\xt)}{\xo^5}=\sum_{\alpha+\beta\geq 5} c_{\alpha\beta}\xo^{\alpha+\beta-5}\xt^\beta$ is an analytic function near the origin, so it can be written as $\mathcal{O}_0(\bxi)$ (recall our notation in \eqref{eq: rm}). 
Hence, by the Taylor series 
\begin{equation*}
\sqrt{1+x} = \sum_{n=0}^\infty \frac{(-1)^{n+1}(2n)!}{(2n-1)4^nn!n!}x^{n}, \quad \text{for} \quad |x|<1,
\end{equation*}
it follows that  
$$\begin{aligned}
&\varphi(\mathbf{K}_\star+\mathbf{R}_\theta[\xo, \xo\xt]^{\textup{T}})\\
&=\frac{\sqrt{3}\xo}{2}\bigg\{1+\frac{1}{2}\big(\xt^2+A_1\xo+A_2\xo^2+\xo^3 \mathcal{O}_0(\bxi)\big)-\frac{1}{8}\big(\xt^2+A_1\xo+A_2\xo^2+\xo^3 \mathcal{O}_0(\bxi)\Big)^2\\
&\qquad\qquad\quad+\frac{1}{16}\big(\xt^2+A_1\xo+A_2\xo^2+\xo^3 \mathcal{O}_0(\bxi)\big)^3 \\ 
&\qquad\qquad\quad+\sum_{n=4}^\infty \frac{(-1)^{n+1}(2n)!}{(2n-1)4^nn!n!}\Big(\xt^2+A_1\xo+A_2\xo^2+\xo^3 \mathcal{O}_0(\bxi)\big)^{n} \bigg\}.
\end{aligned}$$

Recalling the notation \eqref{eq: rm}, in $\{\cdots\}$, we can move all cubic and higher-order terms with respect to $\xo$ in $\xo^3\mathcal{O}_0(\bxi)$ and all quartic and higher-order terms with respect to $\xt$ in $\xt^4 \mathcal{O}_0(\bxi)$ as follows. First, we observe that by the binomial theorem,
$$\begin{aligned}
\big(\xt^2+A_1\xo+A_2\xo^2+\xo^3 \mathcal{O}_0(\bxi)\big)^4&=\sum_{\alpha=0}^4 \frac{4!}{\alpha!(4-\alpha)!} \xt^{2\alpha}\big(A_1\xo+A_2\xo^2+\xo^3 \mathcal{O}_0(\bxi)\big)^{4-\alpha}\\
&=\xo^3\mathcal{O}_0(\bxi)+\xt^4\tilde{\mathcal{O}}_0(\bxi)
\end{aligned}$$
and $|\xt^2+A_1\xo+A_2\xo^2+\xo^3 \mathcal{O}_0(\bxi)|\ll1$ near the origin. Hence, we may include 
$$\sum_{n=4}^\infty \frac{(-1)^{n+1}(2n)!}{(2n-1)4^nn!n!}\Big(\xt^2+A_1\xo+A_2\xo^2+\xo^3 \mathcal{O}_0(\bxi)\big)^{n}$$
in $\xo^3\mathcal{O}_0(\bxi)+\xt^4\tilde{\mathcal{O}}_0(\bxi)$, and subsequently, 
\begin{align}\begin{aligned}\label{Modified phase asymptotic}
\varphi(\mathbf{K}_\star+\mathbf{R}_\theta[\xo, \xo \xt]^{\textup{T}})&=\frac{\sqrt{3}\xo}{2}\bigg\{1+\frac{1}{2}\big(\xt^2+A_1\xo+A_2\xo^2\big)-\frac{1}{8}\big(\xt^2+A_1\xo+A_2\xo^2\big)^2\\
&\qquad\qquad\quad+\frac{1}{16}\big(\xt^2+A_1\xo+A_2\xo^2\big)^3+\xo^3\mathcal{O}_0(\bxi)+\xt^4\tilde{\mathcal{O}}_0(\bxi)\bigg\}\\
&=\frac{\sqrt{3}\xo}{2}\bigg\{1+\frac{1}{2}\xt^2+\bigg(\frac{1}{2}-\frac{1}{4}\xt^2\bigg)A_1\xo+\bigg(-\frac{1}{8}+\frac{3}{16}\xt^2\bigg)A_1^2\xo^2\\
&\qquad\qquad\quad+\bigg(\frac{1}{2}-\frac{1}{4}\xt^2\bigg)A_2\xo^2+\xo^3\mathcal{O}_0(\bxi)+\xt^4\tilde{\mathcal{O}}_0(\bxi)\bigg\}.
\end{aligned}\end{align}
In the above expression, expanding the products, more higher-order terms are generated, and they also can be included in $\xo^3\mathcal{O}_0(\bxi)+\xt^4\tilde{\mathcal{O}}_0(\bxi)$. Precisely, we have
$$\begin{aligned}
\bigg(\frac{1}{2}-\frac{1}{4}\xt^2\bigg)A_1&=\frac{1}{2}(a_\theta-3b_\theta \xt - 3a_\theta  \xt^2+b_\theta \xt^3)-\frac{1}{4}\xt^2(a_\theta -3b_\theta \xt)+\xt^4\tilde{\mathcal{O}}_0(\bxi)\\
&=\frac{a_\theta}{2}-\frac{3b_\theta}{2} \xt-\frac{7a_\theta}{4}  \xt^2+\frac{5b_\theta}{4} \xt^3+\xt^4\tilde{\mathcal{O}}_0(\bxi)
\end{aligned}$$
and
$$\begin{aligned}
&\bigg(-\frac{1}{8}+\frac{3}{16}\xt^2\bigg)A_1^2+\bigg(\frac{1}{2}-\frac{1}{4}\xt^2\bigg)A_2\\
&=-\frac{1}{8}\big\{a_\theta^2-6a_\theta b_\theta \xt+3(3b_\theta^2-2a_\theta^2) \xt^2+20a_\theta b_\theta \xt^3\big\}+\frac{3}{16}(a_\theta^2\xt^2-6 a_\theta b_\theta \xt^3)\\
&\quad-\frac{1}{32}(1+2\xt^2)+\frac{1}{64}\xt^2+\xt^4\tilde{\mathcal{O}}_0(\bxi)\\
&=-\frac{1}{32}-\frac{a_\theta^2}{8}+\frac{3a_\theta b_\theta}{4}\xt+\frac{3(20a_\theta^2-24b_\theta^2-1)}{64} \xt^2-\frac{29}{8}a_\theta b_\theta \xt^3+\xt^4\tilde{\mathcal{O}}_0(\bxi),
\end{aligned}$$
but by \eqref{eq: a b relation}, we have $20a_\theta^2-24b_\theta^2-1=44a_\theta^2-3$. Therefore, plugging these into \eqref{Modified phase asymptotic}, we prove the lemma. 
\end{proof}

As a direct consequence, we obtain a preliminary decay estimate for $\mathbb{C}_{\mathbf{K}_\star}(t; \oz, \theta)$.

\begin{proof}[Proof of Proposition \ref{prop: preliminary degenerate oscillatory integral near the Dirac point}]
By the trivial bound $|\mathbb{C}_{\mathbf{K}_\star}(t; \oz, \theta)|\lesssim_\delta1$, it suffices to show that $|\mathbb{C}_{\mathbf{K}_\star}(t; \oz, \theta)|\lesssim_\delta|t|^{-1/2}$. Indeed, by Lemma \ref{lemma: modified phase function asymptotic near the Dirac point}, we have
$$|\partial_{\xt}^2\psi(\bxi; \oz, \theta)|=\frac{\sqrt{3}|\xo|}{2}|1+\mathcal{O}_1(\bxi)|\sim|\xo|$$
in the support of $\tilde{\chi}_\delta$. Thus, by Fubini's theorem and the van der Corput lemma (Proposition~\ref{Lam:Van der Corput}) for the $\xt$-integral, it follows from \eqref{eq: modified Dirac I3} that 
$$\begin{aligned}
|\mathbb{C}_{\mathbf{K}_\star}(t; \oz, \theta)|&=\bigg|\int_0^\infty\bigg\{\int_{-\infty}^\infty e^{-it\psi(\bxi; \oz, \theta)} \tilde{\chi}_\delta(\bxi) d\xt \bigg\}\xo d\xo \bigg|\\
&\lesssim\int_0^{2\delta} \frac{1}{(|t|\xo)^{1/2}}\xo d\xo\sim_\delta|t|^{-1/2},
\end{aligned}$$
because the support of $\tilde{\chi}_\delta$ is contained in $[-2\delta, 2\delta]^2$ (see Remark \ref{remark: chi delta support}).
\end{proof}

\subsection{Refined degenerate oscillatory integral estimate}\label{sec: refined degenerate oscillatory integral estimate}

In the previous subsection, the $O(|t|^{-1/2})$-decay is obtained only from the oscillation of the modified phase function in the $\xi_2$-direction. In this subsection, by capturing the additional oscillation in the $\xi_1$-direction, we prove the following improved bound.

\begin{proposition}[Refined degenerate oscillatory integral near the Dirac point $\mathbf{K}_\star$]\label{prop: refined degenerate oscillatory integral near the Dirac point}
$$\begin{aligned}
\sup_{\oz >0\textup{ and }\theta\in[0,2\pi)} |\mathbb{C}_{\mathbf{K}_\star}(t; \oz, \theta)|&\lesssim(1+|t|)^{-\frac{5}{6}},\\
\sup_{\oz >0\textup{ and } 
|\theta-\frac{n\pi}{3}|\ge\epsilon \text{ for all } n\in\Z}|\mathbb{C}_{\mathbf{K}_\star}(t; \oz, \theta)|&\lesssim_{\epsilon}(1+|t|)^{-1}.
\end{aligned}$$
\end{proposition}

\begin{remark}\label{remark: a theta lower bound}
$(i)$ Since $a_\theta=\frac{\sin\theta(3\cos^2\theta-\sin^2\theta)}{2\sqrt{3}}=0$ if and only if $\theta=\frac{\pi n}{3}$, the condition for the faster $O(|t|^{-1})$-decay in Proposition \ref{prop: refined degenerate oscillatory integral near the Dirac point} is equivalent to $|a_\theta|\gtrsim 1$.\\
$(ii)$ Proposition \ref{prop: near a Dirac point} and the main result (Theorem \ref{thm: near a Dirac point}) follow from Lemma \ref{lemma: non-stationary estimates near the Dirac point} and Proposition \ref{prop: refined degenerate oscillatory integral near the Dirac point} in the decomposition \eqref{eq: Dirac oscillatory integral decomposition}.
\end{remark}

 \subsection{Proof of Proposition~\ref{prop: refined degenerate oscillatory integral near the Dirac point}} 
We recall from \eqref{change of variable} that by the change of variables, the sectorial domain of integration for $\mathbb{C}_{\mathbf{K}_\star}(t; \oz, \theta)$ is transformed into the rectangular one (see Figure \ref{fig: NearDiracCase}). 
To refine the bound in the previous section, we find an analytic curve $\xt=\gamma(\xo)$, where the modified phase function $\psi(\bxi; \oz, \theta)$ is stationary in the $\xt$-direction at the frequency $\bxi=(\xo, \gamma(\xo))$. Then, near the curve (the red region), we only consider the oscillation in the $\xo$-direction, but the additional decay is obtained from the small measure. On the other hand, away from the curve (the blue region), we take the additional decay by integration by parts in the $\xt$-direction. The $O(|t|^{-c})$-width of the red region will be chosen later to optimize the sum of the two bounds.

For this, using Lemma \ref{lemma: modified phase function asymptotic near the Dirac point}, we differentiate the phase in the $\xt$-direction and write 
\begin{equation}\label{eq: k2 derivative of varphi}
\partial_{\xt}\psi(\bxi; \oz, \theta)=\frac{\sqrt{3}\xo}{2}G(\bxi),
\end{equation}
where 
\begin{equation}\label{eq: G function}
\begin{aligned}
G(\bxi):&=\xt+\bigg(-\frac{3b_\theta}{2}-\frac{7a_\theta}{2}  \xt+\frac{15b_\theta}{4} \xt^2\bigg)\xo\\
&\quad+\bigg(\frac{3a_\theta b_\theta}{4}+\frac{3(44a_\theta^2-3)}{32} \xt-\frac{87}{8}a_\theta b_\theta \xt^2\bigg)\xo^2+\xo^3\mathcal{O}_0(\bxi)+\xt^3\tilde{\mathcal{O}}_0(\bxi).
\end{aligned}
\end{equation}
By the implicit function theorem, we determine the frequencies where $G(\bxi)$ vanishes, which is equivalent to $\partial_{\xi_2}\psi(\bxi; \oz, \theta)=0$ when $\xo>0$.

\begin{lemma}\label{lemma: construction of curve gamma}
There exist small $\delta_0>0$ and an analytic curve $\gamma: (-\delta_0, \delta_0)\to\mathbb{R}$ such that if $|\xo|<\delta_0$, then $G(\xo, \gamma(\xo))=0$ and 
$$\gamma(\xo)=\frac{3b_\theta}{2}\xo+\frac{9a_\theta b_\theta}{2}\xo^2+\mathcal{O}_3(\xo),$$
where $\mathcal{O}_3(\xo)$ is an analytic function of $\xo$ near the origin such that $|\mathcal{O}_3(\xo)|\lesssim |\xo|^3$ (see \eqref{eq: rm}).
\end{lemma}

\begin{proof}
Note that $G(\mathbf{0})=0$ and $\partial_{\xt}G(\bxi)\geq\frac{1}{2}$ near the origin, because
\begin{equation}\label{eq: k2 derivative of G}
\begin{aligned}
\partial_{\xt}G(\bxi)&=1+\bigg(-\frac{7a_\theta}{2}  +\frac{15b_\theta}{2} \xt\bigg)\xo+\bigg(\frac{3(44a_\theta^2-3)}{32} -\frac{87}{4}a_\theta b_\theta \xt\bigg)\xo^2\\
&\quad+\xo^3\mathcal{O}_0(\bxi)+\xt^2\tilde{\mathcal{O}}_0(\bxi).
\end{aligned}
\end{equation}
Therefore, since $G$ is analytic, the analytic implicit function
theorem gives a small $\delta_0>0$ and an analytic function
$\gamma:(-\delta_0,\delta_0)\to\mathbb{R}$ such that
$\gamma(0)=0$, $G(\xo,\gamma(\xo))=0$, and
$\gamma'(\xo)=-\frac{(\partial_{\xo}G)(\xo,\gamma(\xo))}
{(\partial_{\xt}G)(\xo,\gamma(\xo))}$.
Note that $|\gamma'(\xo)|\lesssim1$, because
$|\partial_{\xo}G(\bxi)|\lesssim1$ and
$\partial_{\xt}G(\bxi)\geq\frac12$ near the origin.
After decreasing $\delta_0$ if necessary, repeated implicit
differentiation gives $|\gamma^{(j)}(\xo)|\lesssim_j1$
for every integer $j\geq2$.

As a consequence, plugging the series expansion for $\gamma(\xo)=\sum_{m=1}^\infty d_m \xo^m$ into the equation 
$$\begin{aligned}
0&=G(\xo, \gamma(\xo))\\
&=\gamma(\xo)+\bigg(-\frac{3b_\theta}{2}-\frac{7a_\theta}{2}\gamma(\xo)+\frac{15b_\theta}{4} \gamma(\xo)^2\bigg)\xo\\
&\quad+\bigg(\frac{3a_\theta b_\theta}{4}+\frac{3(44a_\theta^2-3)}{32} \gamma(\xo)-\frac{87}{8}a_\theta b_\theta \gamma(\xo)^2\bigg)\xo^2+\xo^3\mathcal{O}_0(\bxi)+\gamma(\xo)^3\tilde{\mathcal{O}}_0(\bxi)
\end{aligned}$$
(see \eqref{eq: G function}), one can determine the coefficients $d_1$ and $d_2$. Indeed, collecting all cubic and higher-order terms in $\mathcal{O}_3(\xo)$, the above equation can be written as 
$$\begin{aligned}
0&=d_1\xo+d_2\xo^2+\bigg(-\frac{3b_\theta}{2}-\frac{7a_\theta}{2}  d_1\xo\bigg)\xo+\frac{3a_\theta b_\theta}{4}\xo^2+\mathcal{O}_3(\xo)\\
&=\bigg(d_1-\frac{3b_\theta}{2}\bigg)\xo+\bigg(d_2-\frac{7a_\theta}{2}  d_1+\frac{3a_\theta b_\theta}{4}\bigg)\xo^2+\mathcal{O}_3(\xo).
\end{aligned}$$
Therefore, the first two coefficients in the expansion $\gamma(\xo)=\sum_{m=1}^\infty d_m \xo^m$ are given by  $d_1=\frac{3b_\theta}{2}$ and $d_2=\frac{9a_\theta b_\theta}{2}$, and the rest can be collected as $\mathcal{O}_3(\xi_1)$.
\end{proof}

Next, truncating around the curve $\xt=\gamma(\xo)$ (see Figure \ref{fig: NearDiracCase}), we decompose 
\begin{equation}\label{eq: C Dirac decomposition}
\begin{aligned}
\mathbb{C}_{\mathbf{K}_\star}(t; \oz, \theta)&=\sum_{j=0,1}
\int_0^\infty \int_{-\infty}^\infty e^{-it\psi(\bxi; \oz, \theta)} \tilde{\chi}_\delta(\bxi)\chi_j\bigg(\frac{\xt-\gamma(\xo)}{|t|^{-c}}\bigg)\xo d\xt d\xo\\
&=:\sum_{j=0,1} \mathbb{C}^j_{\mathbf{K}_\star}(t; \oz, \theta)
\end{aligned}
\end{equation}
where $\chi_0$ and $\chi_1$ are given in \eqref{eq: chi0} and \eqref{eq: chi1} and $c>0$ will be chosen later. For $\mathbb{C}^0_{\mathbf{K}_\star}(t; \oz, \theta)$, a decay bound is obtained only from the oscillation for the ${\xo}$-variable.

\begin{lemma}[Bound for $\mathbb{C}^0_{\mathbf{K}_\star}(t; \oz, \theta)$]\label{lemma: bound for C_0 Dirac}
$$\begin{aligned}
\sup_{\oz >0\textup{ and }\theta\in[0,2\pi)} |\mathbb{C}^0_{\mathbf{K}_\star}(t; \oz, \theta)|&\lesssim(1+|t|)^{-\frac{1}{3}-c},\\
\sup_{\oz >0\textup{ and } 
|\theta-\frac{n\pi}{3}|\ge\epsilon \text{ for all } n\in\Z}|\mathbb{C}^0_{\mathbf{K}_\star}(t; \oz, \theta)|&\lesssim_{\epsilon}(1+|t|)^{-\frac{1}{2}-c}.
\end{aligned}$$
\end{lemma}

\begin{proof}
We may assume that $|t|\geq 1$, since $\mathbb{C}^0_{\mathbf{K}_\star}(t; \oz, \theta)$ is bounded. For the proof, by change of variables and Fubini's theorem, we write 
\begin{equation}\label{eq: C Dirac 0}
\mathbb{C}^0_{\mathbf{K}_\star}(t; \oz, \theta)=\int_{-\infty}^\infty\bigg\{\int_0^\infty e^{-it\tilde{\psi}(\tbxi)} \tilde{\tilde{\chi}}_\delta(\tbxi)\txo d\txo \bigg\} \chi_0\bigg(\frac{\txt}{|t|^{-c}}\bigg)d\txt,
\end{equation}
where 
\begin{equation}\label{eq: double tilde varphi}
\tilde{\psi}(\tbxi):=\psi\big(\txo, \txt+\gamma(\txo); \oz,\theta\big)
\end{equation}
and
\begin{equation}\label{eq: double tilde chi}
\tilde{\tilde{\chi}}_\delta(\tbxi):=\tilde{\chi}_\delta\big(\txo, \txt+\gamma(\txo)\big).
\end{equation}
Note here that by Remark \ref{remark: chi delta support} and Lemma \ref{lemma: construction of curve gamma}, $\tilde{\tilde{\chi}}_\delta$ is smooth and is supported in $[-20\delta, 20\delta]^2$. For the phase function, plugging $(\txo, \txt+\gamma(\txo))$ in Lemma \ref{lemma: modified phase function asymptotic near the Dirac point}, we obtain that 
$$\begin{aligned}
&\tilde{\psi}(\tbxi)=\frac{\sqrt{3}\txo}{2}\Bigg\{1-\frac{2}{\sqrt{3}}\oz+\frac{1}{2}\big(\txt+\gamma(\txo)\big)^2\\
&+\bigg(\frac{a_\theta}{2}-\frac{3b_\theta}{2} \big(\txt+\gamma(\txo)\big)-\frac{7a_\theta}{4}  (\txt+\gamma(\txo))^2+\frac{5b_\theta}{4} \big(\txt+\gamma(\txo)\big)^3\bigg)\txo\\
&+\bigg(-\frac{1}{32}-\frac{a_\theta^2}{8}+\frac{3a_\theta b_\theta}{4}\big(\txt+\gamma(\txo)\big)+\frac{3(44a_\theta^2-3)}{64} \big(\txt+\gamma(\txo)\big)^2-\frac{29}{8}a_\theta b_\theta \big(\txt+\gamma(\txo)\big)^3\bigg)\txo^{\,2}\\
&+\txo^{\,3}\mathcal{O}_0^{(1)}\big(\txo,\txt+\gamma(\txo)\big)
+\big(\txt+\gamma(\txo)\big)^4
\mathcal{O}_0^{(2)}\big(\txo,\txt+\gamma(\txo)\big)\Bigg\}.
\end{aligned}$$
Since $G(\txo,\gamma(\txo))=0$ by Lemma
\ref{lemma: construction of curve gamma},
\eqref{eq: k2 derivative of varphi} shows that the Taylor expansion
of the expression in braces contains no term linear in $\txt$.
Substituting
$\gamma(\txo)=\frac{3b_\theta}{2}\txo
+\frac{9a_\theta b_\theta}{2}\txo^2+\mathcal{O}_3(\txo)$
into the term independent of $\txt$, we obtain
\begin{equation}\label{eq: translated modified phase function}
\begin{aligned}
\tilde{\psi}(\tbxi)
&=\frac{\sqrt{3}\txo}{2}\bigg\{
1-\frac{2}{\sqrt{3}}\oz+\frac{a_\theta}{2}\txo
-\bigg(\frac{1}{8}-a_\theta^2\bigg)\txo^{\,2}
+\mathcal{O}_3(\txo)
+\txt^{\,2}\mathcal{O}_0(\tbxi)\bigg\}\\
&=\bigg(\frac{\sqrt{3}}{2}-\oz\bigg)\txo
+\frac{\sqrt{3}a_\theta}{4}\txo^{\,2}
-\frac{\sqrt{3}}{2}\bigg(\frac{1}{8}-a_\theta^2\bigg)\txo^{\,3}
+\txo\mathcal{O}_3(\txo)
+\txo\txt^{\,2}\mathcal{O}_0(\tbxi),
\end{aligned}
\end{equation}
where \eqref{eq: a b relation} is used to eliminate $b_\theta$
from the coefficient of $\txo^{\,2}$ inside the braces. Note from \eqref{eq: translated modified phase function} that if $\txo>0$ is small enough, then
\begin{equation}\label{eq: double tilde varphi 2nd derivative}
|\partial_{\txo}^2\tilde{\psi}(\tbxi)|=\bigg|\frac{\sqrt{3}a_\theta}{2}-3\sqrt{3}\bigg(\frac{1}{8}-a_\theta^2\bigg)\txo+\mathcal{O}_2(\tbxi)\bigg|\sim 1,
\end{equation}
provided that $|a_\theta|\gtrsim 1$. Moreover, for any $\theta\in[0,2\pi)$, we have 
\begin{equation}\label{eq: double tilde varphi 3rd derivative}
|\partial_{\txo}^3\tilde{\psi}(\tbxi)|=\bigg|-3\sqrt{3}\bigg(\frac{1}{8}-a_\theta^2\bigg)+\mathcal{O}_1(\tbxi)\bigg|\sim 1,
\end{equation}
since $a_\theta^2\leq \frac{1}{12}$ (see \eqref{eq: a b relation}). Hence, applying the van der Corput lemma (Proposition~\ref{Lam:Van der Corput}) to the inner integral \eqref{eq: C Dirac 0}, we obtain
$$\bigg|\int_0^\infty e^{-it\tilde{\psi}(\tbxi)} \tilde{\tilde{\chi}}_\delta(\tbxi)\txo d\txo \bigg|\lesssim\left\{\begin{aligned}
&|t|^{-\frac{1}{2}}&&\textup{if }|a_\theta|\gtrsim 1\\
&|t|^{-\frac{1}{3}}&&\textup{otherwise}.
\end{aligned}\right.$$
Therefore, applying this bound to the inner integral in \eqref{eq: C Dirac 0} with $|\txt|\lesssim |t|^{-c}$, we prove the lemma.
\end{proof}

In the integral $\mathbb{C}^1_{\mathbf{K}_\star}(t; \oz, \theta)$, the frequencies such that $\partial_{\xt}\psi(\bxi; \oz, \theta)$ is small are truncated out. Thus, the additional decay is obtained by integration by parts for the $\xt$-variable as in the proof of the preliminary bound (Proposition \ref{prop: preliminary degenerate oscillatory integral near the Dirac point}), but we also capture the oscillation in the ${\xo}$-direction.

\begin{lemma}[Bound for $\mathbb{C}^1_{\mathbf{K}_\star}(t; z, \theta)$]\label{lemma: bound for C_1 Dirac}
$$\begin{aligned}
\sup_{\oz>0\textup{ and }\theta\in[0,2\pi)} |\mathbb{C}^1_{\mathbf{K}_\star}(t; \oz, \theta)|&\lesssim(1+|t|)^{-(\frac{4}{3}-c)},\\
\sup_{\oz >0\textup{ and } 
|\theta-\frac{n\pi}{3}|\ge\epsilon \text{ for all } n\in\Z}|\mathbb{C}^1_{\mathbf{K}_\star}(t; \oz, \theta)|&\lesssim_{\epsilon}(1+|t|)^{-(\frac{3}{2}-c)}.
\end{aligned}$$
\end{lemma}

\begin{proof}
Again, we may assume that $|t|\geq1$. By integration by parts with $\partial_{\xt}\psi(\bxi; \oz,\theta)=\frac{\sqrt{3}\xo}{2}G(\bxi)$ (see \eqref{eq: k2 derivative of varphi}) and distributing the derivative, $\mathbb{C}^1_{\mathbf{K}_\star}(t; \oz, \theta)$ can be written as  
$$\begin{aligned}
\mathbb{C}^1_{\mathbf{K}_\star}(t; \oz, \theta)&=-\frac{1}{it}\int_0^\infty\int_{-\infty}^\infty \partial_{\xt}(e^{-it\psi(\bxi; \oz, \theta)})\Bigg\{\frac{\tilde{\chi}_\delta(\bxi)\chi_1(\frac{\xt-\gamma(\xo)}{|t|^{-c}})}{\frac{\sqrt{3}\xo}{2}G(\bxi)}\Bigg\}\xo d\xt d\xo\\
&=\frac{2}{i\sqrt{3}t}\int_0^\infty\int_{-\infty}^\infty e^{-it\psi(\bxi; \oz, \theta)}\partial_{\xt}\Bigg\{\frac{\tilde{\chi}_\delta(\bxi)\chi_1(\frac{\xt-\gamma(\xo)}{|t|^{-c}})}{G(\bxi)}\Bigg\}d\xt d\xo\\
&=\frac{2}{i\sqrt{3}t}\int_{-\infty}^\infty\Bigg[\int_0^\infty e^{-it\tilde{\psi}(\tbxi)}\partial_{\txt}\Bigg\{\frac{\tilde{\tilde{\chi}}_\delta({\tbxi})\chi_1(\frac{\txt}{|t|^{-c}})}{G(\txo, \gamma(\txo)+\txt)}\Bigg\}d\txo\Bigg]d\txt,
\end{aligned}$$
where $\tilde{\psi}(\tbxi)$ and $\tilde{\tilde{\chi}}_\delta(\tbxi)$ are defined in \eqref{eq: double tilde varphi} and \eqref{eq: double tilde chi}. Hence, using the lower bounds for the second and the third derivatives of $\psi\big(\txo, \txt+\gamma(\txo);z,\theta\big)$ depending on $\theta$ (see \eqref{eq: double tilde varphi 2nd derivative} and \eqref{eq: double tilde varphi 3rd derivative}) and employing the van der Corput lemma (Proposition~\ref{Lam:Van der Corput}) for the inner $\txo$-integral as in the proof of Lemma \ref{lemma: bound for C_0 Dirac}, we obtain 
$$|\mathbb{C}^1_{\mathbf{K}_\star}(t; \oz, \theta)|\lesssim\left\{\begin{aligned}
&\mathcal{R}(t)|t|^{-\frac{3}{2}}&&\textup{if }|\theta-\tfrac{\pi n}{3}|\geq \epsilon\\
&\mathcal{R}(t)|t|^{-\frac{4}{3}}&&\textup{otherwise},
\end{aligned}\right.$$
where
$$\mathcal{R}(t):=\int_{-\infty}^\infty\int_0^\infty \bigg|\partial_{\txo}\partial_{\txt}\bigg\{\frac{\tilde{\tilde{\chi}}_\delta(\tbxi)\chi_1\big(\frac{\txt}{|t|^{-c}}\big)}{G(\txo, \gamma(\txo)+\txt)}\bigg\}\bigg|d\txo d\txt.$$
Therefore, it is enough to show that 
\begin{equation}\label{eq: R(t) bound}
\mathcal{R}(t)\lesssim|t|^{c}.
\end{equation}
Indeed, we observe that $\tilde{\tilde{\chi}}_\delta(\tbxi)$ is smooth and is supported in the area $[-3\delta, 3\delta]^2$. Thus, it follows that 
$$\begin{aligned}
|\mathcal{R}(t)|&\lesssim\int_{-3\delta}^{3\delta}\int_0^{3\delta} \bigg|\frac{\chi_1(\frac{\txt}{|t|^{-c}})}{G(\txo, \gamma(\txo)+\txt)}\bigg|+\bigg|\nabla_{\tbxi}\bigg(\frac{\chi_1(\frac{\txt}{|t|^{-c}})}{G(\txo, \gamma(\txo)+\txt)}\bigg)\bigg|\\
&\qquad\qquad\qquad\qquad\qquad\qquad\quad\ \ +\bigg|\partial_{\txo}\partial_{\txt}\bigg(\frac{\chi_1(\frac{\txt}{|t|^{-c}})}{G(\txo, \gamma(\txo)+\txt)}\bigg)\bigg|d\txo d\txt.
\end{aligned}$$
Note that by the fact that
\begin{equation}\label{eq: G k_2 derivative}
\partial_{\txt}G(\tbxi)\sim 1
\end{equation}
(see \eqref{eq: k2 derivative of G}), the fundamental theorem of calculus with $G(\xo, \gamma(\xo))=0$ yields 
\begin{equation}\label{eq: G comparability}
\big|G(\txo, \gamma(\txo)+\txt)\big|=\bigg|\int_0^{\txt}\partial_{\txt}G(\txo, \gamma(\txo)+s) ds\bigg|\sim |\txt|.
\end{equation}
Moreover, repeating the calculations in \eqref{eq: translated modified phase function} (see  \eqref{eq: k2 derivative of varphi} for the definition of $G(\tbxi)$), one can show that 
\begin{equation}\label{eq: G k_1 derivative}
\begin{aligned}
&\partial_{\txo}\big(G(\txo, \gamma(\txo)+\txt)\big)\\
&=\partial_{\txo}\partial_{\txt}\bigg\{1-\frac{2}{\sqrt{3}}z+\frac{a_\theta}{2}\txo-\bigg(\frac{1}{8}-a_\theta^2\bigg)\txo^{\,2}+\mathcal{O}_3(\txo)+\txt^{\,2} \mathcal{O}_0(\tbxi)\bigg\}=\txt\mathcal{O}_0(\tbxi)
\end{aligned}
\end{equation}
and
\begin{equation}\label{eq: G k_1 k_2 derivative}
\partial_{\txo}\partial_{\txt}\big(G(\txo, \gamma(\txo)+\txt)\big)=\mathcal{O}_0(\tbxi).
\end{equation}
Therefore, by \eqref{eq: G k_2 derivative}, \eqref{eq: G comparability} and \eqref{eq: G k_1 derivative}, we have
$$\begin{aligned}
&\bigg|\frac{\chi_1(\frac{\txt}{|t|^{-c}})}{G(\txo, \gamma(\txo)+\txt)}\bigg|+\bigg|\nabla_{\tbxi}\bigg(\frac{\chi_1(\frac{\txt}{|t|^{-c}})}{G(\txo, \gamma(\txo)+\txt)}\bigg)\bigg|\\
&\leq \chi_1\bigg(\frac{\txt}{|t|^{-c}}\bigg)\bigg\{\frac{1}{|G(\txo, \gamma(\txo)+\txt)|}+\frac{|\nabla_{\tbxi}(G(\txo, \gamma(\txo)+\txt))|}{G(\txo, \gamma(\txo)+\txt)^2}\bigg\}+\frac{|\frac{1}{|t|^{-c}}\chi_0'(\frac{\txt}{|t|^{-c}})|}{|G(\txo, \gamma(\txo)+\txt)|}\\
&\lesssim \frac{\chi_1(\frac{\txt}{|t|^{-c}})}{|\txt|}+ \frac{\chi_1(\frac{\txt}{|t|^{-c}})}{\txt^{\,2}}+\frac{|\frac{1}{|t|^{-c}}\chi_0'(\frac{\txt}{|t|^{-c}})|}{|\txt|}\lesssim |t|^c+\frac{\mathbbm{1}_{|\txt|\gtrsim |t|^{-c}}}{\txt^{\,2}}+|t|^{2c}\mathbbm{1}_{|\txt|\sim |t|^{-c}}.
\end{aligned}$$
On the other hand, applying \eqref{eq: G k_2 derivative}, \eqref{eq: G comparability}, \eqref{eq: G k_1 derivative} and \eqref{eq: G k_1 k_2 derivative} to 
$$\begin{aligned}
\partial_{\txo}\partial_{\txt}\bigg(\frac{\chi_1(\frac{\txt}{|t|^{-c}})}{G(\txo, \gamma(\txo)+\txt)}\bigg)&=\chi_1\bigg(\frac{\txt}{|t|^{-c}}\bigg)\bigg\{-\frac{\partial_{\txo}\partial_{\txt}(G(\txo, \gamma(\txo)+\txt))}{G(\txo, \gamma(\txo)+\txt)^2}\\
&\qquad\qquad\qquad +\frac{2\partial_{\txo}(G(\txo, \gamma(\txo)+\txt)) \partial_{\txt}(G(\txo, \gamma(\txo)+\txt))}{G(\txo, \gamma(\txo)+\txt)^3}\bigg\}\\
&\quad+\frac{\partial_{\txo}(G(\txo, \gamma(\txo)+\txt))\frac{1}{|t|^{-c}}\chi_0'(\frac{\txt}{|t|^{-c}})}{G(\txo, \gamma(\txo)+\txt)^2},
\end{aligned}$$
we obtain that 
$$\bigg|\partial_{\txo}\partial_{\txt}\bigg(\frac{\chi_1(\frac{\txt}{|t|^{-c}})}{G(\txo, \gamma(\txo)+\txt)}\bigg)\bigg|\lesssim\frac{\mathbbm{1}_{|\txt|\gtrsim |t|^{-c}}}{\txt^{\,2}}+|t|^{2c}\mathbbm{1}_{|\txt|\sim |t|^{-c}},$$
where for the second upper bound, we used that $\partial_{\txt}G(\tbxi)>0$ (see \eqref{eq: G k_2 derivative}). Therefore, collecting all, we prove that 
$$|\mathcal{R}(t)|\lesssim\int_{-3\delta}^{3\delta}\int_0^{3\delta} \bigg\{|t|^c+\frac{\mathbbm{1}_{|\txt|\gtrsim |t|^{-c}}}{\txt^{\,2}}+|t|^{2c}\mathbbm{1}_{|\txt|\sim |t|^{-c}}\bigg\}d\txt d\txo\sim |t|^c.$$
This completes the proof of \eqref{eq: R(t) bound}.
\end{proof}

Finally, we are ready to prove the main result of this subsection.
\begin{proof}[Proof of Proposition \ref{prop: refined degenerate oscillatory integral near the Dirac point}]
We apply Lemma \ref{lemma: bound for C_0 Dirac} and \ref{lemma: bound for C_1 Dirac} to the decomposition \eqref{eq: C Dirac decomposition}. Then, choosing $c>0$ optimizing the bound, we prove Proposition \ref{prop: refined degenerate oscillatory integral near the Dirac point}.
\end{proof}

\begin{proof}[Proof of Proposition \ref{prop: near a Dirac point}]
Applying Lemma \ref{lemma: non-stationary estimates near the Dirac point} and Proposition \ref{prop: refined degenerate oscillatory integral near the Dirac point} to the decomposition \eqref{eq: Dirac oscillatory integral decomposition}, we complete the proof of the proposition.
\end{proof}

\section{Oscillatory integral localized at an intersection of two degenerate frequency curves}\label{sec: oscillatory integral localized at a Dirac point at an intersection of two degenerate frequency curves}

Next, we show Theorem \ref{thm: oscillatory integral estimates} for $\mathbf{K}_2\in\mathcal{K}_2'$, that is, one of the intersections of two degenerate frequency curves (see Figure \ref{fig: first primitive cell}).

\begin{theorem}[Oscillatory integral estimate; $\mathbf{K}=\mathbf{K}_2$]\label{thm: near an intersection of two curves}
For $\mathbf{K}_2=(0,\pi)$, there exist $C_2>0$ and a smooth cut-off $\chi_{\leq\delta}$, whose support is contained in a sufficiently small disk of radius $\delta>0$, such that for any $\bz\in\mathbb{R}^2$, 
$$|\mathbb{I}_{\mathbf{k}\approx\mathbf{K}_2}(t;\bz)|\leq\frac{C_2}{(1+|t|)^{2/3}}.
$$
\end{theorem}

\begin{remark}
By the reduction in Section \ref{sec: reduction to the oscillatory integral estimates}, Theorem \ref{thm: near an intersection of two curves} implies Theorem \ref{thm: oscillatory integral estimates} for $\mathbf{K}\in\mathcal{K}_2$.
\end{remark}

\subsection{Reduction to the degenerate oscillatory integral}\label{sec: reduction, near an intersection of two curves}

For the oscillatory integral 
$$\mathbb{I}_{\mathbf{K}_2}(t; \bz):=e^{it\mathbf{K}_2\cdot\bz}\int_{-\infty}^\infty\int_{-\infty}^\infty e^{-it(\varphi(\mathbf{K}_2+\mathbf{k}')-\mathbf{k}'\cdot \bz)}\chi_{\leq\delta}(\mathbf{k}')d\mathbf{k}',$$
we change the variables by $\tilde{\mathbf{k}}=(\mathbf{k}'\cdot\mathbf{v}_1, \mathbf{k}'\cdot\mathbf{v}_2)$ with $\mathbf{v}_1=(\frac{\sqrt{3}}{2},\frac{1}{2})$ and $\mathbf{v}_2=(\frac{\sqrt{3}}{2},-\frac{1}{2})$. Then, it is important to observe from direct calculations for the phase formula \eqref{eq: phase function} with the expansion $\sqrt{1+x}=1+\frac{1}{2}x-\frac{1}{8}x^2+\frac{1}{16}x^3+\cdots$ that in the new coordinates, the phase function is expanded as 
$$\begin{aligned}
\varphi(\mathbf{K}_2+\mathbf{k}')&=\sqrt{3-2\sin(\mathbf{k}'\cdot\mathbf{v}_1) +2\sin(\mathbf{k}'\cdot\mathbf{v}_2)-2\cos(\mathbf{k}'\cdot(\mathbf{v}_1-\mathbf{v}_2))}\\
&=\sqrt{3-2\sin \tilde{k}_1 +2\sin \tilde{k}_2-2\cos(\tko-\tilde{k}_2)}\\
&=\bigg\{1-2(\tilde{k}_1-\tilde{k}_2)+(\tilde{k}_1-\tilde{k}_2)^2+\frac{(\tilde{k}_1)^3-(\tilde{k}_2)^3}{3}+\mathcal{O}_4(\tilde{\mathbf{k}})\bigg\}^{\frac{1}{2}}\\
&=1-\tilde{k}_1+\tilde{k}_2+\frac{1}{6}(\tilde{k}_1)^3-\frac{1}{6}(\tilde{k}_2)^3+\mathcal{O}_4(\tilde{\mathbf{k}}),
\end{aligned}$$
where $\mathcal{O}_4(\tilde{\mathbf{k}})$ is an analytic function such that $|\mathcal{O}_4(\tilde{\mathbf{k}})|\lesssim |\tilde{\mathbf{k}}|^4$ (see \eqref{eq: rm}). Subsequently, the integral becomes 
\begin{equation}\label{eq: integral for K2}
\mathbb{I}_{\mathbf{K}_2}(t; \bz)=\frac{2}{\sqrt{3}}e^{it(\mathbf{K}_2\cdot\bz-1)}\int_{-\infty}^\infty\int_{-\infty}^\infty e^{-it(\frac{1}{6}(\tilde{k}_1)^3-\frac{1}{6}(\tilde{k}_2)^3+\mathcal{O}_4(\tilde{\mathbf{k}})-\tilde{\oz}_1\tilde{k}_1-\tilde{\oz}_2\tilde{k}_2)}\chi_0\bigg(\frac{|\tilde{\mathbf{k}}|}{\delta}\bigg)d\tilde{\mathbf{k}},
\end{equation}
where $(\tilde{\oz}_1, \tilde{\oz}_2)=(1+\frac{1}{\sqrt{3}}\oz_1+\oz_2, -1+\frac{1}{\sqrt{3}}\oz_1-\oz_2)$, and the cut-off $\chi_{\leq\delta}$ is chosen so that $\chi_{\leq\delta}(\frac{\tilde{k}_1+\tilde{k}_2}{\sqrt{3}}, \tilde{k}_1-\tilde{k}_2)=\chi_0(\frac{|\tilde{\mathbf{k}}|}{\delta})$.

Now, generalizing the right hand side of \eqref{eq: integral for K2} up to trivial changes of variables  that normalize the coefficient of the lowest-order term in the phase function to $\pm1$ \footnote{For the proof, it is convenient for symmetric reductions to allow the coefficient for $k_1^3$  to be either 1 or $-1$, and not to specify the coefficients in $\mathcal{O}_4(\bxi)=\sum_{m_1+m_2\geq 4}c_{m_1, m_2} \xo^{m_1} \xt^{m_2}$ (see \eqref{eq: rm}).}, we define the oscillatory integral
\begin{equation}\label{eq: I2(t,z) definition}
\mathbb{I}_2(t;\bz):=\int_{-\infty}^\infty\int_{-\infty}^\infty e^{-it(\phi_2(\bxi)-\bxi\cdot \bz)}\chi_0\bigg(\frac{|\bxi|}{\delta_0\delta}\bigg)d\bxi,
\end{equation}
where $0<\delta\ll\delta_0\ll1$ and $\chi_0(\frac{\cdot}{\delta_0\delta})=\sum_{N\leq \delta}\eta(\frac{\cdot}{\delta_0N})$ and  
$$\phi_2(\bxi):=\pm \xo^3+\xt^3+\mathcal{O}_4(\bxi).$$
(See \eqref{eq: chi0} and \eqref{eq: eta} for the definition of $\chi_0$ and $\eta$.)
\begin{remark}
For numerical simplicity in the proof below, the smooth cut-off $\chi_0(\frac{|\cdot|}{\delta})$ in $\mathbb{I}_{\mathbf{K}_2}(t; \bz)$ is replaced by $\chi_0(\frac{|\cdot|}{\delta_0\delta})$ in $\mathbb{I}_2(t;\bz)$. Indeed, this change does not affect the result, because small $\delta>0$ is not specified in Theorem \ref{thm: near an intersection of two curves}. 
\end{remark}

Subsequently, the proof of Theorem \ref{thm: near an intersection of two curves} is reduced to the following proposition.

\begin{proposition}\label{prop: two curve intersection}
$$\sup_{\bz\in\R^2}|\mathbb{I}_2(t;\bz)| \lesssim(1+|t|)^{-\frac23}.$$
\end{proposition}

\subsection{Proof of Proposition \ref{prop: two curve intersection}}

For the readers who are familiar with the theory of degenerate oscillatory integrals such as Varchenko's and Karpushkin's theorems, a short proof is given in Section \ref{sec: alternative proof of Prop 6.4}. In this subsection, we provide a proof that does not rely on those theorems.

 For the dyadic decomposition, we choose a standard smooth cut-off
\begin{equation}\label{eq: eta}
\eta:\mathbb{R}\to [0,1]
\end{equation}
such that $\eta\equiv 1$ on $[\frac{6}{5}, \frac{9}{5}]$, $\eta\equiv 0$ outside $[\frac{4}{5}, \frac{11}{5}]$ and $\sum_{N\in 2^{\mathbb{Z}}}\eta(\frac{\cdot}{N})\equiv 1$. Replacing $\chi_0(\frac{\cdot}{\delta})$  in \eqref{eq: chi0} with $\sum_{N\leq \delta}\eta(\frac{\cdot}{N})$ if necessary, we may assume that $\chi_0(\frac{\cdot}{\delta})=\sum_{N\leq \delta}\eta(\frac{\cdot}{N})$.
Then, applying the dyadic decomposition and rescaling, we decompose 
$$\mathbb{I}_2(t; \bz)=\sum_{N\leq\delta}\int_{-\infty}^\infty\int_{-\infty}^\infty e^{-it(\phi_2(\bxi)-\bxi\cdot \bz)}\eta\bigg(\frac{|\bxi|}{\delta_0N}\bigg)d\bxi=\sum_{N\leq\delta} N^2\mathbb{I}_{2;N}\bigg(N^3t;\frac{\bz}{N^2}\bigg),$$
where $N=2^m\in 2^{\mathbb{Z}}$ denotes a dyadic number and
$$\mathbb{I}_{2;N}(t;\bz):=\int_{-\infty}^\infty\int_{-\infty}^\infty e^{-it(\phi_{2;N}(\bxi)-\bxi\cdot \bz)}\eta\bigg(\frac{|\bxi|}{\delta_0}\bigg)d\bxi$$
with the phase
\begin{equation}\label{eq: phase for two intersection}
\phi_{2;N}(\bxi):=\frac{1}{N^3}\phi_2(N\bxi)=\pm \xo^3+ \xt^3+\frac{1}{N^3}\mathcal{O}_4(N\bxi).
\end{equation}

In a sequel, we assume that $N\leq\delta$. Then, it is enough to show that
\begin{equation}\label{eq: prop two curve intersection dyadic piece reduction}
\sup_{\bz\in\mathbb{R}^2}|\mathbb{I}_{2;N}(t;\bz)|\lesssim|t|^{-a}
\end{equation}
for some $a>\frac{2}{3}$, because together with the trivial bound $|\mathbb{I}_{2;N}(t;\bz)|\lesssim1$, \eqref{eq: prop two curve intersection dyadic piece reduction} implies that $|\mathbb{I}_2(t; \bz)|\lesssim \sum_{N\leq\delta} N^2\min\{1,(N^3|t|)^{-a}\}\lesssim |t|^{-\frac{2}{3}}$. Indeed, for each $\mathbb{I}_{2;N}(t;\bz)$, we observe that for the contribution of the integral away from the $\xo$- and the $\xt$-axes, given by  
$$\mathbb{I}_{2;N}^{\textup{away}}(t;\bz):=\int_{-\infty}^\infty\int_{-\infty}^\infty e^{-it(\phi_{2;N}(\bxi)-\bxi\cdot \bz)}\eta\bigg(\frac{|\bxi|}{\delta_0}\bigg)\bigg\{1-\chi_0\bigg(\frac{\xt}{\delta |\xo|}\bigg)-\chi_0\bigg(\frac{\xo}{\delta |\xt|}\bigg)\bigg\}d\bxi,$$
the phase function satisfies $|\textup{det}(\nabla^2\phi_{2;N}(\bxi))|\sim |\xo||\xt|\sim_{\delta,\delta_0} 1$ (see \eqref{eq: phase for two intersection}). Thus, the standard non-degenerate phase oscillatory integral estimate immediately yields $|\mathbb{I}_{2;N}^{\textup{away}}(t;\bz)|\lesssim_{\delta, \delta_0} |t|^{-1}$. For \eqref{eq: prop two curve intersection dyadic piece reduction}, by symmetry\footnote{since the higher-order terms in $\mathcal{O}_4(\bxi)$ in $\phi_2(\bxi)=\pm \xo^3+\xt^3+\mathcal{O}_4(\bxi)$ are not specified.}, it remains to consider the contribution near the ${\xo}$-axis. Moreover, changing the variable ${\xo}$ by $-{\xo}$ for ${\xo}<0$ and by symmetry, it is enough to consider the integral near the positive part of the ${\xo}$-axis, that is, 
$$\mathbb{I}_{2;N}^{\textup{near}}(t;\bz):=\int_{-\infty}^\infty\int_0^\infty e^{-it(\phi_{2;N}(\bxi)-\bxi\cdot \bz)}\eta\bigg(\frac{|\bxi|}{\delta_0}\bigg)\chi_0\bigg(\frac{\xt}{\delta \xo}\bigg)d\xo d\xt.$$
(See \eqref{eq: chi0} and \eqref{eq: eta} for the definition of $\chi_0$ and $\eta$.)
Therefore, the proof of \eqref{eq: prop two curve intersection dyadic piece reduction} can be reduced to showing that there exists $a>\frac{2}{3}$ such that 
\begin{equation}\label{eq: prop two curve intersection dyadic piece reduction'}
\sup_{\bz\in\mathbb{R}^2}|\mathbb{I}_{2;N}^{\textup{near}}(t;\bz)|\lesssim|t|^{-a}.
\end{equation}

To analyze the integral $\mathbb{I}_{2;N}^{\textup{near}}(t;\bz)$, we note that $|\xo|\sim |\bxi|\sim\delta_0$ and $|\xt|\lesssim \delta|\xo|\sim\delta_0\delta $ in the support of $\eta(\frac{|\bxi|}{\delta_0}){\chi_0(\frac{\xt}{\delta \xo})}$. Then, replacing $\xo$ by $\txo\mp \txt$ to cancel the cubic term $\xt^{\,3}$ in \eqref{eq: phase for two intersection}, we write 
$$\mathbb{I}_{2;N}^{\textup{near}}(t;\bz):=\int_{-\infty}^\infty\int_0^\infty e^{-it\tilde{\phi}_{2;N}(\tbxi;\bz)}\tilde{\chi}(\tbxi)d\txo d\txt,$$
where 
\begin{equation}\label{eq: phase for two intersection tilde}
\begin{aligned}
\tilde{\phi}_{2;N}(\tbxi;\bz):&=\pm \txo^{\,3}- 3\txo^{\,2}\txt\pm 3\txo\txt^{\,2}+\frac{1}{N^3}\mathcal{O}_4(N\tbxi)-\oz_1\txo-(\oz_2\mp \oz_1)\txt
\end{aligned}
\end{equation}
and $\tilde{\chi}(\tbxi):=\eta(\frac{|(\txo\mp \txt,\txt)|}{\delta_0})\chi_0(\frac{\txt}{\delta(\txo\mp \txt)})$ is a smooth function supported in 
\begin{equation}\label{eq: K2 chi tilde range}
\txo\sim \delta_0\quad\textup{and}
\quad |\txt|\lesssim\delta_0\delta.
\end{equation}
We observe that 
\begin{equation}\label{eq: k2 derivative of modified phase for two intersection}
\partial_{\txt}\tilde{\phi}_{2;N}(\tbxi;\bz)=\delta_0G_N(\tbxi)-3\txo^{\,2}-(\oz_2\mp \oz_1),
\end{equation}
where
\begin{equation}\label{eq: G function for two intersection}
G_N(\tbxi):=\frac{1}{\delta_0}\bigg(\pm 6\txo\txt+\frac{1}{N^2}\mathcal{O}_3(N\tbxi)\bigg).
\end{equation}
Since $|G_N(\tbxi)|\lesssim\delta_0\delta$ when $N\leq\delta$, it is natural to decompose further as
$$\mathbb{I}_{2;N}^{\textup{near}}(t;\bz)=\mathbb{I}_{2;N; (0)}^{\textup{near}}(t;\bz)+\mathbb{I}_{2;N;(1)}^{\textup{near}}(t;\bz),$$
where 
$$\mathbb{I}_{2;N; (j)}^{\textup{near}}(t;\bz):=\int_{-\infty}^\infty\int_0^\infty e^{-it\tilde{\phi}_{2;N}(\tbxi;\bz)}\chi_j\bigg(\frac{3\txo^{\,2}+(\oz_2\mp \oz_1)}{\delta_0\delta}\bigg)\tilde{\chi}(\tbxi){d\txo}d\txt$$
(see \eqref{eq: chi0} and \eqref{eq: chi1} for the definition of $\chi_0$ and $\chi_1$). In the integral $\mathbb{I}_{2;N; (1)}^{\textup{near}}(t;\bz)$, we have $|\partial_{\txt}\tilde{\phi}_{2;N}(\cdot\,;\bz)|\geq |3\txo^{\,2}+(\oz_2\mp \oz_1)|-\delta_0|G_N(\tbxi)|\gtrsim\delta_0\delta$. Hence, one can show that $|\mathbb{I}_{2;N; (1)}^{\textup{near}}(t;\bz)|\lesssim_{\delta_0,\delta}|t|^{-1}$ by integration by parts for the $\txt$-variable.

It remains to consider $\mathbb{I}_{2;N; (0)}^{\textup{near}}(t;\bz)$. We will show that 
\begin{equation}\label{eq: I2N0 claim}
|\mathbb{I}_{2;N; (0)}^{\textup{near}}(t;\bz)|\lesssim|t|^{-\frac{5}{6}},
\end{equation}
which completes the proof of Proposition~\ref{prop: two curve intersection} (see \eqref{eq: prop two curve intersection dyadic piece reduction'}). Indeed, as we did in the proof of Lemma \ref{lemma: construction of curve gamma}, we will prove \eqref{eq: I2N0 claim} constructing an analytic curve $\txt=\tilde{\gamma}_N(\txo)$ such that $\tilde{\phi}_{2;N}(\txo,\cdot;\bz)$ is stationary in the $\txt$-direction. Suppose that $\tbxi$ is contained in the support of $\tilde{\chi}$ so that \eqref{eq: K2 chi tilde range} holds. Then, by the definition \eqref{eq: G function for two intersection}, we have
\begin{equation}\label{eq: G_N properties}
\left\{\begin{aligned}
G_N(\mathbf{0})&=0,\\
\partial_{\txo}G_N(\tbxi)&=\frac{1}{\delta_0}\bigg(\pm 6\txt+\frac{1}{N}\mathcal{O}_2(N\tbxi)\bigg),\\
\partial_{\txt}G_N(\tbxi)&=\frac{1}{\delta_0}\bigg(\pm 6\txo+\frac{1}{N}\mathcal{O}_2(N\tbxi)\bigg)\approx \pm \frac{6\txo}{\delta_0},\\
\partial_{\txo}^{j_1}\partial_{\txt}^{j_2}G_N(\tbxi)&=\frac{1}{\delta_0}N^{j_1+j_2-2}\mathcal{O}_0(N\tbxi)\quad\textup{for }j_1+j_2\geq3,
\end{aligned}\right.
\end{equation}
where $N\leq\delta$. We claim that there exists an analytic
function $\gamma_N$ near the origin such that
$G_N(\txo,\gamma_N(\txo))=0$ and
\begin{equation}\label{eq: gamma N expansion}
\gamma_N(\txo)=\frac{(N\txo)^2}{N}\mathcal{O}_0(N\txo).
\end{equation}
For the proof, we seek $\gamma_N$ in the form
$\gamma_N(\txo)=\frac{(N\txo)^2}{N}g(N\txo)$ for some unknown
function $g$. Substituting this into
\eqref{eq: G function for two intersection}, with
\[
\mathcal{O}_3(\tbxi)
=\sum_{m_1+m_2\geq3}
c_{m_1,m_2}\txo^{\,m_1}\txt^{\,m_2},
\]
we obtain
\[
G_N(\txo,\gamma_N(\txo))
=\frac{N\txo^{\,3}}{\delta_0}
H_N\big(N\txo,g(N\txo)\big),
\]
where
\[
H_N(\tbxi)
=\pm6\txt
+\sum_{m_1+m_2\geq3}
c_{m_1,m_2}\txo^{\,m_1+2m_2-3}\txt^{\,m_2}.
\]
Since $H_N$ is analytic,
$H_N(0,\mp c_{3,0}/6)=0$ and
$\partial_{\txt}H_N(0,\mp c_{3,0}/6)=\pm6\neq0$,
the analytic implicit function theorem gives an analytic
function $g$ near the origin satisfying
$H_N(\txo,g(\txo))=0$.
Thus, setting $\gamma_N(\txo)=N\txo^2g(N\txo)$, we obtain
an analytic function satisfying
$G_N(\txo,\gamma_N(\txo))=0$ and
\eqref{eq: gamma N expansion}.

On the region $\txo\sim\delta_0$, implicit differentiation gives
\[
\gamma_N'(\txo)
=-\frac{(\partial_{\txo}G_N)(\txo,\gamma_N(\txo))}
        {(\partial_{\txt}G_N)(\txo,\gamma_N(\txo))}.
\]
Moreover, \eqref{eq: gamma N expansion} yields
$|\gamma_N(\txo)|\lesssim\delta_0\delta$ and
$|\gamma_N'(\txo)|\lesssim\delta$.

Subsequently, for each $\txo\sim\delta_0$, expanding the Taylor series around $\txt=\gamma_N(\txo)$ with $G_N(\txo, \gamma_N(\txo))=0$, the equation $\partial_{\txt}\tilde{\phi}_{2;N}(\txo, \cdot\,;\bz)=0$ (see \eqref{eq: k2 derivative of modified phase for two intersection}) can be written as 
$$0=-\frac{3\txo^{\,2}+(\oz_2\mp \oz_1)}{\delta_0}+\sum_{m=1}^\infty \frac{\partial_{\txt}^mG_N(\txo, \gamma_N(\txo))}{m!}\big(\cdot-\gamma_N(\txo)\big)^m.$$
For this, we insert an ansatz $$\txt=\tilde{\gamma}_N(\txo)=\gamma_N(\txo)+\sum_{m=1}^\infty d_{N,m}(\txo)\bigg(\frac{3\txo^{\,2}+(\oz_2\mp \oz_1)}{\delta_0}\bigg)^m$$
and determine $d_{N,m}(\txo)$. In fact, this is possible because $|\frac{3\txo^{\,2}+(\oz_2\mp \oz_1)}{\delta_0}|\lesssim\delta\ll1$, $|\partial_{\txt}G_N(\tbxi)|\sim 1$,  $\partial_{\txt}^2G_N(\tbxi)= \frac{1}{\delta_0}\mathcal{O}_1(N\tbxi)$ and $\partial_{\txt}^mG_N(\tbxi)= \frac{N^{m-2}}{\delta_0}\mathcal{O}_0(N\tbxi)$ for $m\geq 3$ (see \eqref{eq: G_N properties}). Then, determining coefficients $d_{N,m}(\txo)$ together with \eqref{eq: gamma N expansion}, we obtain that 
\begin{equation}\label{eq: gamma tilde N expansion}
\tilde{\gamma}_N(\txo)=\frac{(N\txo)^2}{N}\mathcal{O}_0(N\txo)+\frac{\frac{3\txo^{\,2}+\oz_2\mp \oz_1}{\delta_0}}{\partial_{\txt}G_N(\txo, \gamma_N(\txo))}+N \mathcal{O}_2\bigg(\frac{3\txo^{\,2}+\oz_2\mp \oz_1}{\delta_0}\bigg).
\end{equation}

Now, for each $\txo$, expanding the series expansion of $\tilde{\phi}_{2;N}(\txo, \cdot;\bz)$ around $\txt=\tilde{\gamma}_N(\txo)$ with $\partial_{\txt}\tilde{\phi}_{2;N}\big(\txo, \tilde{\gamma}_N(\txo);\bz\big)=0$, we write 
$$\tilde{\phi}_{2;N}(\tbxi;\bz)=\tilde{\phi}_{2;N}\big(\txo, \tilde{\gamma}_N(\txo);\bz\big)+\sum_{m=2}^\infty\frac{\partial_{\txt}^m\tilde{\phi}_{2;N}(\txo, \tilde{\gamma}_N(\txo);\bz)}{m!}\big(\txt-\tilde{\gamma}_N(\txo)\big)^m.$$
Note that $\partial_{\txt}^2\tilde{\phi}_{2;N}(\tbxi;\bz)=\delta_0\partial_{\txt}G_N(\tbxi)\approx \pm \delta_0$  and $|\partial_{\txt}^m\tilde{\phi}_{2;N}(\tbxi;\bz)|=\delta_0|\partial_{\txt}^{m-1}G_N(\tbxi)|\lesssim_m N$ for $m\geq 3$ (see \eqref{eq: k2 derivative of modified phase for two intersection} and \eqref{eq: G_N properties}). Thus, we can substitute the sum on the right-hand side with
$$\sum_{m=2}^\infty\frac{\partial_{\txt}^m\tilde{\phi}_{2;N}(\txo, \tilde{\gamma}_N(\txo);\bz)}{m!}\big(\txt-\tilde{\gamma}_N(\txo)\big)^m=:\pm \delta_0\ttxt^{\,2}.$$
Then, relabeling $\ttxo:=\txo$, the integral becomes 
\begin{equation}\label{eq: I2N0 near transformed integral}
\mathbb{I}_{2;N; (0)}^{\textup{near}}(t;\bz):=\int_{-\infty}^\infty\int_0^\infty e^{-it\tilde{\tilde{\phi}}_{2;N}(\ttbxi;\bz)}\chi_0\bigg(\frac{3\ttxo^{\,2}+(\oz_2\mp \oz_1)}{\delta_0\delta}\bigg)\tilde{\tilde{\chi}}\big(\ttbxi\big)d\ttxo d\ttxt,
\end{equation}
where 
$$\tilde{\tilde{\phi}}_{2;N}\big(\ttbxi;\bz\big):=\tilde{\phi}_{2;N}\Big(\ttxo, \tilde{\gamma}_N\big(\ttxo\big);\bz\Big)\pm\delta_0\ttxt^{\,2},$$
and $\tilde{\tilde{\chi}}$ is the product of $\tilde{\chi}$ and the Jacobian determinant induced by the change of variables which satisfies that $|(\partial_{\ttxt})^m \tilde{\tilde{\chi}}(\ttbxi)|\lesssim_{m} 1 $ for $m\ge0$. 
We observe from \eqref{eq: phase for two intersection tilde}, \eqref{eq: G_N properties} (in particular, $\partial_{\ttxt}G_N(\ttbxi)\approx \pm \frac{6\ttxo}{\delta_0}$) and \eqref{eq: gamma tilde N expansion} that 
$$\begin{aligned}
&\tilde{\phi}_{2;N}\Big(\ttxo, \tilde{\gamma}_N\big(\ttxo\big);\bz\Big)\\
&=\pm \ttxo^{\,3}- 3\ttxo^{\,2}\tilde{\gamma}_N\big(\ttxo\big)\pm 3\ttxo\tilde{\gamma}_N\big(\ttxo\big)^2+\frac{1}{N^3}\mathcal{O}_4\big(N\ttxo, N\tilde{\gamma}_N\big(\ttxo\big)\big)-\oz_1\ttxo-(\oz_2\mp \oz_1)\tilde{\gamma}_N\big(\ttxo\big)\\
&=\pm \ttxo^{\,3}- 3\ttxo^{\,2}\bigg(\frac{3\ttxo^{\,2}+(\oz_2\mp \oz_1)}{\pm 6\ttxo}\bigg)\pm 3\ttxo\bigg(\frac{3\ttxo^{\,2}+(\oz_2\mp \oz_1)}{\pm 6\ttxo}\bigg)^2\\
&\quad-\oz_1\ttxo-(\oz_2\mp \oz_1)\bigg(\frac{3\ttxo^{\,2}+(\oz_2\mp \oz_1)}{\pm 6\ttxo}\bigg)+N\times (\textup{smooth function of }\ttxo)\\
&=\pm\frac{1}{4}\ttxo^{\,3}-\frac{\oz_1\pm \oz_2}{2}\ttxo\mp\frac{(\oz_2\mp \oz_1)^2}{12\ttxo}+N\times (\textup{smooth function of }\ttxo).
\end{aligned}$$

Now, we estimate $\mathbb{I}_{2;N; (0)}^{\textup{near}}(t;\bz)$ as we did for $\mathbb{C}_{\mathbf{K}_\star}(t; \oz, \theta)$ in the previous section (see the decomposition \eqref{eq: C Dirac decomposition}). For this, we further decompose the integral $\mathbb{I}_{2;N; (0)}^{\textup{near}}(t;\bz)$ in \eqref{eq: I2N0 near transformed integral} using $$\chi_0\bigg(\frac{\ttxt}{|t|^{-1/2}}\bigg)+\chi_1\bigg(\frac{\ttxt}{|t|^{-1/2}}\bigg)\equiv1$$
(see \eqref{eq: chi0} and \eqref{eq: chi1}). For the integral with the cut-off $\chi_0(\frac{\ttxt}{|t|^{-1/2}})$ near the $\ttxo$-axis, we apply the van der Corput lemma (Proposition~\ref{Lam:Van der Corput}) for the inner $\ttxo$--integral with $$\partial_{\ttxo}^3\tilde{\tilde{\phi}}_{2;N}\big(\ttbxi;\bz\big)\approx\pm\frac{3}{2}\pm \frac{(\oz_2\mp \oz_1)^2}{2\ttxo^{\,4}}\approx \pm 6$$
($|3\ttxo^{\,2}+(\oz_2\mp \oz_1)|\lesssim\delta_0\delta\ll 1$ is used) to obtain the $O(|t|^{-1/3})$ bound. Then, using $|\ttxt|\lesssim|t|^{-1/2}$ for the outer integral, we obtain $O(|t|^{-\frac{5}{6}})$-decay. For the other integral away from the $\ttxo$-axis, by integration by parts with $\partial_{\ttxt}\big(\tilde{\tilde{\phi}}_{2;N}\big(\ttbxi;\bz\big)\big)=\pm 2\delta_0\ttxt$, we write 
$$\begin{aligned}
\frac{1}{it}\int_{-\infty}^\infty\int_0^\infty e^{-it\tilde{\tilde{\phi}}_{2;N}\big(\ttbxi;\bz\big)}\partial_{\ttxt}\Bigg\{\frac{1}{\pm 2\delta_0\ttxt}\chi_0\bigg(\frac{3\ttxo^{\,2}+(\oz_2\mp \oz_1)}{\delta_0\delta}\bigg)\chi_1\bigg(\frac{\ttxt}{|t|^{-1/2}}\bigg)\tilde{\tilde{\chi}}\big(\ttbxi\big)\Bigg\}d\ttxo d\ttxt.
\end{aligned}$$
Then, applying the van der Corput lemma (Proposition~\ref{Lam:Van der Corput}) with $(\partial_{\ttxo})^3\tilde{\tilde{\phi}}_{2;N}\big(\ttbxi;\bz\big)\approx\pm 6$ for the $\ttxo$-integral, we prove that it satisfies an $O(|t|^{-5/6})$ bound. Therefore, collecting all, one can show \eqref{eq: prop two curve intersection dyadic piece reduction'} with $a=\frac{5}{6}$.

\section{Oscillatory integral localized at a non-intersection point}\label{sec: oscillatory integral localized at a non-intersection point}

For Theorem \ref{thm: oscillatory integral estimates} (as well as Theorem \ref{main thm: dispersion estimate}), it remains to show the following estimate.

\begin{theorem}[Oscillatory integral estimate; $\mathbf{K}\in\mathcal{K}_1'$]\label{thm: near a degenerate frequency, not an intersection}
Suppose that $\alpha_{1}(\mathbf{K})=0$ and $\mathbf{K}\in(\mathcal{B}_{\textup{rhom}}\cap\mathbb{R}_{\geq0}^2)\setminus (\mathcal{K}_2\cup\mathcal{K}_3)$. Then, for sufficiently small $\delta>0$, there exist $C_1>0$ and a smooth cut-off $\chi_{\leq\delta}$, whose support is contained in the disk of radius $\delta>0$, such that for any $\bz\in\mathbb{R}^2$, 
$$|\mathbb{I}_{\mathbf{k}\approx\mathbf{K}}(t;\bz)|\leq\frac{C_1}{(1+|t|)^{5/6}}.
$$
\end{theorem}

\begin{remark}\label{remark: choice of delta for K_1}
$(i)$ By the reduction in Section \ref{sec: reduction to the oscillatory integral estimates}, Theorem \ref{thm: oscillatory integral estimates} for $\mathbf{K}\in\mathcal{K}_1$ follows from Theorem \ref{thm: near a degenerate frequency, not an intersection}.\\
$(ii)$ In Theorem \ref{thm: near a degenerate frequency, not an intersection}, we choose $\delta>0$ so that the distance from $\mathbf{K}$ to $\mathbf{K}_2$ and that to $\mathbf{K}_\star$ are larger than $2\delta$. Hence, the choice of $\delta>0$ may depend on $\mathbf{K}$.
\end{remark}

\subsection{Reduction to the degenerate oscillatory integral}
We are concerned with the integral
$$\mathbb{I}_{\mathbf{K}}(t; \bz):=e^{it\mathbf{K}\cdot\bz}\int_{-\infty}^\infty\int_{-\infty}^\infty e^{-it(\varphi(\mathbf{K}+\mathbf{k}')-\bz\cdot \mathbf{k}')}\chi_{\leq\delta}(\mathbf{k}')d\mathbf{k}',$$
where $\chi_{\leq\delta}$ is a smooth cut-off to be chosen below. Here, taking sufficiently small $\delta>0$ so that the support of $\chi_{\leq\delta}$ does not contain any intersection of degenerate frequency curves, we may assume that 
$$|\alpha_{2}(\mathbf{K})|, |\alpha_{12}(\mathbf{K})|\gtrsim_\delta 1$$ 
(see Figure \ref{fig: first primitive cell} and Remark \ref{remark: choice of delta for K_1}). For this integral, we change the variables by $\tilde{\mathbf{k}}=(\mathbf{k}'\cdot\mathbf{v}_1, \mathbf{k}'\cdot\mathbf{v}_2)$ as we did in Section \ref{sec: reduction, near an intersection of two curves}. Then, we obtain
\begin{equation}\label{eq: integral for K1}
\mathbb{I}_{\mathbf{K}}(t; \bz)=\frac{2}{\sqrt{3}}e^{it\mathbf{K}\cdot\bz}\int_{-\infty}^\infty\int_{-\infty}^\infty e^{-it(\tilde{\varphi}(\tilde{\mathbf{K}}+\tilde{\mathbf{k}})-\tilde{\bz}\cdot\tilde{\mathbf{k}})}\tilde{\chi}_{\leq\delta}(\tilde{\mathbf{k}})d\tilde{\mathbf{k}},
\end{equation}
where $\tilde{\mathbf{K}}=(\mathbf{K}\cdot\mathbf{v}_1, \mathbf{K}\cdot\mathbf{v}_2)$, $\tilde{\bz}=(\frac{1}{\sqrt{3}}\oz_1+\oz_2, \frac{1}{\sqrt{3}}\oz_1-\oz_2)$ and 
$$\tilde{\varphi}(\tilde{\mathbf{k}})=\sqrt{3+2\cos \tilde{k}_1 +2\cos \tilde{k}_2 +2\cos(\tilde{k}_1-\tilde{k}_2)}.$$
Accordingly, $\tilde{\alpha}_{1}(\tilde{\mathbf{k}})=1+\cos \tilde{k}_2+\cos(\tilde{k}_1-\tilde{k}_2)$, $\tilde{\alpha}_{2}(\tilde{\mathbf{k}})=1+\cos \tilde{k}_1+\cos(\tilde{k}_1-\tilde{k}_2)$ and $\tilde{\alpha}_{12}(\tilde{\mathbf{k}})=1+\cos \tilde{k}_1+\cos \tilde{k}_2$ correspond to the three key functions $\alpha_{1}(\mathbf{k})$, $\alpha_{2}(\mathbf{k})$ and $\alpha_{12}(\mathbf{k})$ respectively. Thus, we have $\tilde{\alpha}_1(\tilde{\mathbf{K}})=0$ and $|\tilde{\alpha}_{2}(\tilde{\mathbf{K}})|, |\tilde{\alpha}_{12}(\tilde{\mathbf{K}})|\gtrsim_\delta 1$.

For the integral $\mathbb{I}_{\mathbf{K}}(t; \bz)$, it is important to know the asymptotic of the phase $\tilde{\varphi}(\tilde{\mathbf{K}}+\tilde{\mathbf{k}})$.
\begin{lemma}[Asymptotic of $\tilde{\varphi}(\tilde{\mathbf{K}}+\tilde{\mathbf{k}})$]\label{lemma: asymptotic of phase (K1 case)}
If $\tilde{\alpha}_1(\tilde{\mathbf{K}})=0$ and $|\tilde{\alpha}_{2}(\tilde{\mathbf{K}})|, |\tilde{\alpha}_{12}(\tilde{\mathbf{K}})|\gtrsim 1$, then near $\tilde{\mathbf{k}}=\mathbf{0}$,
$$\begin{aligned}
\tilde\varphi(\tilde{\mathbf{K}}+\tilde{\mathbf{k}})&=\tilde\varphi(\tilde{\mathbf{K}})+(\nabla\tilde{\varphi})(\tilde{\mathbf{K}})\cdot\tilde{\mathbf{k}}+\frac{(\partial_{\tilde{k}_1}^2\tilde{\varphi})(\tilde{\mathbf{K}})}{2}(\tilde{k}_1)^2\\
&\quad+\sum_{m_1+m_2\geq3}\frac{(\partial_{\tilde{k}_1}^{m_1}\partial_{\tilde{k}_2}^{m_2}\tilde{\varphi})(\tilde{\mathbf{K}})}{(m_1)!(m_2)!}(\tilde{k}_1)^{m_1}(\tilde{k}_2)^{m_2}
\end{aligned}$$
with $|(\partial_{\tilde{k}_1}^2\tilde{\varphi})(\tilde{\mathbf{K}})|\gtrsim 1$ and $ |(\partial_{\tilde{k}_2}^3\tilde{\varphi})(\tilde{\mathbf{K}})|\gtrsim 1$.
\end{lemma}

\begin{proof}
By Theorem \ref{thm: gradient and hessian} and the assumptions, we observe that
$$\nabla_{\tilde{\mathbf{k}}}^2\tilde{\varphi}=\begin{bmatrix} 
-\frac{\tilde{\alpha}_{2}(\tilde{\alpha}_1+\tilde{\alpha}_{12})}{\tilde{\varphi}^3}  & \frac{\tilde{\alpha}_{1}\tilde{\alpha}_{2}}{\tilde{\varphi}^3} \\
\frac{\tilde{\alpha}_{1}\tilde{\alpha}_{2}}{\tilde{\varphi}^3} & -\frac{\tilde{\alpha}_{1}(\tilde{\alpha}_2+\tilde{\alpha}_{12})}{\tilde{\varphi}^3}
\end{bmatrix}\quad\Rightarrow\quad (\nabla_{\tilde{\mathbf{k}}}^2\tilde{\varphi})(\tilde{\mathbf{K}})=\begin{bmatrix} 
-\frac{\tilde{\alpha}_{2}(\tilde{\mathbf{K}})\tilde{\alpha}_{12}(\tilde{\mathbf{K}})}{\tilde{\varphi}(\tilde{\mathbf{K}})^3}  & 0 \\
0 & 0
\end{bmatrix}$$
with $|(\partial_{\tilde{k}_1}^2\tilde{\varphi})(\tilde{\mathbf{K}})|=|\frac{\tilde{\alpha}_{2}(\tilde{\mathbf{K}})\tilde{\alpha}_{12}(\tilde{\mathbf{K}})}{\tilde{\varphi}(\tilde{\mathbf{K}})^3}|\gtrsim 1$. Hence, for the lemma, it suffices to show that $|(\partial_{\tilde{k}_2}^3\tilde{\varphi})(\tilde{\mathbf{K}})|\gtrsim1$. Indeed, differentiating $\partial_{\tilde{k}_2}^2\tilde{\varphi}=-\frac{\tilde{\alpha}_{1}(\tilde{\alpha}_2+\tilde{\alpha}_{12})}{\tilde{\varphi}^3}$ and inserting $\tilde{\mathbf{k}}=\tilde{\mathbf{K}}$ with $\tilde{\alpha}_1(\tilde{\mathbf{K}})=0$, we obtain 
$$(\partial_{\tilde{k}_2}^3\tilde{\varphi})(\tilde{\mathbf{K}})=-\frac{(\partial_{\tilde{k}_2}\tilde{\alpha}_{1})(\tilde{\mathbf{K}})(\tilde{\alpha}_2(\tilde{\mathbf{K}})+\tilde{\alpha}_{12}(\tilde{\mathbf{K}}))}{\tilde{\varphi}(\tilde{\mathbf{K}})^3}=-\frac{(\partial_{\tilde{k}_2}\tilde{\alpha}_{1})(\tilde{\mathbf{K}})}{\tilde{\varphi}(\tilde{\mathbf{K}})},$$
where in the last step, we used that $\tilde{\varphi}(\tilde{\mathbf{K}})^2=\tilde{\alpha}_1(\tilde{\mathbf{K}})+\tilde{\alpha}_2(\tilde{\mathbf{K}})+\tilde{\alpha}_{12}(\tilde{\mathbf{K}})=\tilde{\alpha}_2(\tilde{\mathbf{K}})+\tilde{\alpha}_{12}(\tilde{\mathbf{K}})$ by the definition of $\varphi(\mathbf{k})$. Note that $\tilde{\varphi}(\tilde{\mathbf{K}})\neq0$, because by the change of variables, $\tilde{\varphi}(\tilde{\mathbf{K}})=0$ is equivalent to $\varphi(\mathbf{K})=0$ if and only if $\mathbf{K}$ is the intersection of the three degenerate frequency curves. Hence, for $|(\partial_{\tilde{k}_2}^3\tilde{\varphi})(\tilde{\mathbf{K}})|\gtrsim1$, it remains to show that $(\partial_{\tilde{k}_2}\tilde{\alpha}_{1})(\tilde{\mathbf{K}})\neq0$.

For contradiction, we assume that 
\begin{equation}\label{eq: lemma: asymptotic of phase (K1 case) contradiction assumption}
(\partial_{\tilde{k}_2}\tilde{\alpha}_{1})(\tilde{\mathbf{K}})=\partial_{\tilde{k}_2}\big(1+\cos \tilde{k}_2+\cos(\tilde{k}_1-\tilde{k}_2)\big)\big|_{\tilde{\mathbf{k}}=\tilde{\mathbf{K}}}=-\sin (\tilde{K}_2)+\sin (\tilde{K}_1-\tilde{K}_2)=0,
\end{equation}
where $\tilde{\mathbf{K}}=(\tilde{K}_1, \tilde{K}_2)$. On the other hand, by the assumption, we have $\tilde{\alpha}_{1}(\tilde{\mathbf{K}})=1+\cos (\tilde{K}_2)+\cos(\tilde{K}_1-\tilde{K}_2)=0$. Hence, we obtain $1=\cos^2(\tilde{K}_1-\tilde{K}_2)+\sin^2(\tilde{K}_1-\tilde{K}_2)=(1+\cos (\tilde{K}_2))^2+\sin^2 (\tilde{K}_2)=2+2\cos(\tilde{K}_2)$. Thus, it follows that
$$\cos(\tilde{K}_2)=\cos(\tilde{K}_1-\tilde{K}_2)=-\frac{1}{2}\quad\textup{and}\quad \sin (\tilde{K}_2)=\sin (\tilde{K}_1-\tilde{K}_2)=\pm\frac{\sqrt{3}}{2}$$
(by \eqref{eq: lemma: asymptotic of phase (K1 case) contradiction assumption}, they have the same sign). By the trigonometric identities, we have
$$\begin{aligned}
\cos(\tilde{K}_1)&=\cos((\tilde{K}_1-\tilde{K}_2)+\tilde{K}_2)\\
&=\cos(\tilde{K}_1-\tilde{K}_2)\cos \tilde{K}_2-\sin(\tilde{K}_1-\tilde{K}_2)\sin \tilde{K}_2\\
&=\bigg(-\frac{1}{2}\bigg)^2-\bigg(\frac{\sqrt{3}}{2}\bigg)^2=-\frac{1}{2}.
\end{aligned}$$
Consequently, we obtain
$$\tilde{\alpha}_{2}(\tilde{\mathbf{K}})=1+\cos \tilde{K}_1+\cos(\tilde{K}_1-\tilde{K}_2)=1+\bigg(-\frac{1}{2}\bigg)+\bigg(-\frac{1}{2}\bigg)=0,$$ which leads to a contradiction.  
\end{proof}

Generalizing the integral on the right-hand side of \eqref{eq: integral for K1} up to simple changes of variables, for sufficiently small $\delta>0$, we define 
\begin{equation}\label{eq: I1 integral}
\mathbb{I}_1(t;\bz):=\int_{-\infty}^\infty\int_{-\infty}^\infty e^{-it(\phi_1(\bxi)-\bz\cdot\bxi)}  \chi_0\bigg(\frac{\xo}{\delta}\bigg)
  \chi_0\bigg(\frac{\xt}{\delta^{2/3}}\bigg)d\bxi,
\end{equation}
where
$$\phi_1(\bxi):=\xo^{\,2}+\sum_{m_1+m_2\geq 3}c_{m_1, m_2}\xo^{m_1}\xt^{m_2},$$
$c_{0,3}=\pm1$, assuming that there exists $c>0$ such that
\begin{equation}\label{eq: coefficient assumption for no intersection}
|c_{m_1, m_2}|\leq c^{|m_1|+|m_2|}, \qquad
\text{for } m_1+m_2\geq 3,\quad (m_1,m_2)\neq(0,3),
\end{equation}
and $\chi_0$ is defined in \eqref{eq: chi0}. Indeed, if necessary, replacing $\delta$ by $\delta^2$ in $\mathbb{I}_1(t;\bz)$ and then rescaling by \[
(t,z_1,z_2,\xi_1,\xi_2)
\mapsto
(\delta^2 t,\delta^{-1}z_1,\delta^{-\frac{4}{3}}z_2,
\delta\xi_1,\delta^{\frac{2}{3}}\xi_2),
\] we may assume that $c>0$ is sufficiently small in \eqref{eq: coefficient assumption for no intersection}, provided that $\delta$ is much smaller, that is, $0<\delta^{\frac{2}{3}}\ll c\ll1$. Then, Theorem \ref{thm: near a degenerate frequency, not an intersection} follows from the following proposition.

\begin{proposition}\label{prop: no intersection}
Under the assumption \eqref{eq: coefficient assumption for no intersection} with small $c>0$, we have
$$\sup_{\bz\in\R^2}|\mathbb{I}_1(t;\bz)| \lesssim(1+|t|)^{-\frac56}.$$
\end{proposition}

\subsection{Proof of Proposition \ref{prop: no intersection}}

Similarly to Proposition \ref{prop: two curve intersection}, we give a rather direct proof of Proposition \ref{prop: no intersection} in this subsection, while a shorter proof involving Varchenko's theorem and Karpushkin's theorem will be given in Section \ref{sec: alternative proof of Prop 7.4}.

We may assume that $|\bz|\leq 10\delta$, because otherwise it is easy to show that $|\mathbb{I}_1(t;\bz)|\lesssim \min \{1,|t|^{-1}\}$ by integration by parts once. Under this assumption, as a first step, we construct an analytic curve $\xo=\gamma(\xt)$ where the phase is stationary in the $\xo$-direction, that is,  
\begin{equation}\label{eq: gamma for no intersection}
(\partial_{\xo}\phi_1)(\gamma(\xt), \xt)=0,
\end{equation}
as follows. First, by direct calculations, we note that 
\begin{equation}\label{eq: phi 3 derivatives}
\left\{\begin{aligned}
(\partial_{\xo}\phi_1)(\bxi)&=2\xo+\sum_{m_1+m_2\geq 2} (m_1+1)c_{m_1+1, m_2}\xo^{m_1}\xt^{m_2},\\
(\partial_{\xo}^2\phi_1)(\bxi)&=2+\sum_{m_1+m_2\geq 1} (m_1+2)(m_1+1)c_{m_1+2, m_2}\xo^{m_1}\xt^{m_2}\approx2,\\
(\partial_{\xo}\partial_{\xt}\phi_1)(\bxi)&=\sum_{m_1+m_2\geq 1} (m_1+1)(m_2+1)c_{m_1+1, m_2+1}\xo^{m_1}\xt^{m_2}.
\end{aligned}\right.
\end{equation}
Hence, the analytic implicit function theorem gives a small
$\delta_0>0$ and a unique analytic function
$\gamma:(-\delta_0,\delta_0)\to\mathbb{R}$ such that
$\gamma(0)=0$,
$(\partial_{\xo}\phi_1)(\gamma(\xt),\xt)=0$, and
$\gamma'(\xt)
=-\frac{(\partial_{\xo}\partial_{\xt}\phi_1)(\gamma(\xt),\xt)}
{(\partial_{\xo}^2\phi_1)(\gamma(\xt),\xt)}$.
Note that $\delta_0$ can be chosen independently of sufficiently
small $c,\delta>0$, and so we may assume that
$0<\delta\ll c\ll\delta_0\ll1$.
By construction and \eqref{eq: phi 3 derivatives}, we have
$|\gamma(\xt)|\lesssim\delta_0$ and
$|\gamma'(\xt)|\lesssim\delta_0$.
Moreover, after decreasing $\delta_0$ if necessary,
repeated implicit differentiation and
\eqref{eq: coefficient assumption for no intersection} give
$|\gamma^{(j)}(\xt)|\lesssim_j\delta_0$
for every integer $j\geq2$.

Next, using the analytic curve constructed in the previous step, we construct another analytic curve $\xo=\tilde{\gamma}(\xt)$, where the phase including the linear term is stationary in the $\xo$-direction, i.e., 
\begin{equation}\label{eq: gamma tilde for no intersection}
\big(\partial_{\xo}(\phi_1-\bz\cdot\bxi)\big)(\tilde{\gamma}(\xt), \xt)=0.
\end{equation}
Indeed, for each $\xt$, expanding the Taylor series expansion around $\xo=\gamma(\xt)$  with \eqref{eq: gamma for no intersection}, we write the equation $(\partial_{\xo}(\phi_1-\bz\cdot\bxi))(\cdot, \xt)=0$ as 
\begin{equation}\label{eq: phase equation (K1 case)}
\begin{aligned}
0&=\sum_{j=1}^\infty \frac{(\partial_{\xo}^{j+1}\phi_1)(\gamma(\xt), \xt)}{j!}(\xo-\gamma(\xt))^j-\oz_1
\end{aligned}
\end{equation}
or equivalently, 
\begin{equation}\label{eq: phase equation (K1 case)'}
(\xo-\gamma(\xt))\Bigg\{1+\sum_{j=1}^\infty \frac{(\partial_{\xi_1}^{j+2}\phi_1)(\gamma(\xt), \xt)}{(j+1)!(\partial_{\xo}^2\phi_1)(\gamma(\xt), \xt)}(\xo-\gamma(\xt))^j\Bigg\}=\frac{\oz_1}{(\partial_{\xo}^2\phi_1)(\gamma(\xt), \xt)}.
\end{equation}
Since $(\partial_{\xo}^2\phi_1)(\mathbf{0})=2$,
the analytic implicit function theorem applied to
$(\partial_{\xo}\phi_1)(\xo,\xt)-\oz_1=0$
gives a solution $\xo=\tilde{\gamma}(\xt)$ that is analytic
in $(\xt,\oz_1)$ near $(0,0)$.
For $|\oz_1|\leq10\delta$ with $\delta$ sufficiently small,
this solution has a convergent expansion
$\tilde{\gamma}(\xt)
=\gamma(\xt)+\sum_{m=1}^{\infty}d_m(\xt)\oz_1^m$,
where each $d_m$ is analytic.
Thus, \eqref{eq: gamma tilde for no intersection} holds.
Moreover, \eqref{eq: coefficient assumption for no intersection},
$|\oz_1|\leq10\delta$, and repeated implicit differentiation yield,
after decreasing $\delta_0$ if necessary,
\begin{equation}\label{eq: gamma tilde derivatives for no intersection}
|\tilde{\gamma}^{(j)}(\xt)|\lesssim_j\delta_0
\quad\textup{for every integer }j\geq0.
\end{equation}

Now, fixing $\xt$, we expand the phase $\phi_1(\bxi)-\bz\cdot\bxi$ around $\xo=\tilde{\gamma}(\xt)$ as 
$$\phi_1(\bxi)-\bz\cdot\bxi=\phi_1(\tilde{\gamma}(\xt), \xt)-\oz_1 \tilde{\gamma}(\xt)-\oz_2\xt+\sum_{j=2}^\infty\frac{(\partial_{\xo}^j\phi_1)(\tilde{\gamma}(\xt), \xt)}{j!}(\xo-\tilde{\gamma}(\xt))^j.$$
Note that the right-hand side series does not include a linear term in $\xo-\tilde{\gamma}(\xt)$. Moreover, since $\partial_{\xo}^2\phi_1\approx 2$, we can change the variable in the integral $\mathbb{I}_1(t,\bz)$ by 
$$\begin{aligned}
\txo^{\,2}&:=\sum_{j=2}^\infty\frac{(\partial_{\xo}^j\phi_1)(\tilde{\gamma}(\xt), \xt)}{j!}(\xo-\tilde{\gamma}(\xt))^j\\
&=\frac{(\partial_{\xo}^2\phi_1)(\tilde{\gamma}(\xt), \xt)}{2}(\xi_1-\tilde{\gamma}(\xt))^2\Bigg\{1+\sum_{j=3}^\infty\frac{2(\partial_{\xo}^j\phi_1)(\tilde{\gamma}(\xt), \xt)}{j!(\partial_{\xo}^2\phi_1)(\tilde{\gamma}(\xt), \xt)}(\xo-\tilde{\gamma}(\xt))^{j-2}\Bigg\}.
\end{aligned}$$
Then, coming back to the integral \eqref{eq: I1 integral}, relabeling $\txt:=\xt$, it follows that   
$$\mathbb{I}_1(t;\bz)=\int_{-\infty}^\infty\int_{-\infty}^\infty e^{-it\tilde{\phi}_1(\tbxi;\bz)}\tilde{\chi}_\delta(\tbxi)d\tbxi,$$
where $\tilde{\chi}_\delta(\tbxi)$ is a smooth function supported in $\{\tbxi: |\txo|\lesssim\delta\textup{ and }|\txt|\lesssim\delta^{\frac{2}{3}}\}$ and
$$\tilde{\phi}_1(\tbxi;\bz):=\phi_1\big(\tilde{\gamma}(\txt), \txt\big)-\oz_1 \tilde{\gamma}(\txt)-\oz_2\txt+\txo^{\,2}.$$
Subsequently, we observe that the phase function $\tilde{\phi}_1(\tbxi;\bz)$ can be considered as 
$\tilde{\phi}_1(\tbxi;\bz)=\txo^{\,2}+c_{0,3}\txt^{\,3}+g(\txt)$ for some analytic function $g(\txt)$ and with $c_{0,3}=\pm1$ such that $|g^{(j)}(\txt)|\lesssim_j \delta_0$, because 
$$\begin{aligned}
\phi_1\big(\tilde{\gamma}(\txt), \txt\big)&=\tilde{\gamma}\big(\txt\big)^2+c_{0,3}\txt^{\,3}+\Bigg\{\sum_{m=0}^\infty c_{0, m+4}\txt^{\,m}\Bigg\}\txt^{\,4}\\
&\quad+\tilde{\gamma}(\txt)\sum_{m_1+m_2\geq 2}c_{m_1+1, m_2}\tilde{\gamma}\big(\txt\big)^{m_1}\txt^{\,m_2}.
\end{aligned}$$
and $\tilde{\gamma}$ is a small analytic function (see \eqref{eq: gamma tilde derivatives for no intersection}).

After the above manipulations, we are finally ready to estimate the oscillatory integral $\mathbb{I}_1(t;\bz)$. For this, we decompose 
$$\mathbb{I}_1(t;\bz)=\sum_{j=0,1}\int_{-\infty}^\infty\int_{-\infty}^\infty e^{-it\tilde{\phi}_1(\tbxi;\bz)}\tilde{\chi}_\delta(\tbxi)\chi_j\bigg(\frac{\txo}{|t|^{-1/2}}\bigg)d\txt d\txo=:\sum_{j=0,1}\mathbb{I}_{1;(j)}(t;\bz),$$
where $\chi_0$ and $\chi_1$ are the smooth cut-offs given in \eqref{eq: chi0} and \eqref{eq: chi1}. For $\mathbb{I}_{1;(0)}(t;\bz)$, we apply the van der Corput lemma (Proposition~\ref{Lam:Van der Corput}) with $|\partial_{\txt}^3\tilde{\phi}_1(\tbxi;\bz)|\sim1$ in the inner $\txt$-variable integral, and we obtain $|\mathbb{I}_{1;(0)}(t;\bz)|\lesssim |t|^{-\frac{5}{6}}$. On the other hand, for $\mathbb{I}_{1;(1)}(t;\bz)$, by integration by parts in the $\txo$-direction with $\partial_{\txo}\tilde{\phi}_1(\tbxi;\bz)=2\txo$, we write 
$$\begin{aligned}
\mathbb{I}_{1;(1)}(t;\bz)&=\frac{1}{it}\int_{-\infty}^\infty\bigg[\int_{-\infty}^\infty e^{-it\tilde{\phi}_1(\tbxi;\bz)}\partial_{\txo}\bigg\{\frac{1}{\partial_{\txo}\tilde{\phi}_1(\tbxi;\bz)}\tilde{\chi}_\delta(\tbxi)\chi_1\bigg(\frac{\txo}{|t|^{-1/2}}\bigg)\bigg\} d\txo\bigg]d\txt\\
&=\frac{1}{it}\int_{-\infty}^\infty\bigg\{\int_{-\infty}^\infty e^{-it\tilde{\phi}_1(\tbxi;\bz)}\tilde{\chi}_\delta(\tbxi) d\txt\bigg\}\partial_{\txo}\bigg\{\frac{1}{2\txo}\chi_1\bigg(\frac{\txo}{|t|^{-1/2}}\bigg)\bigg\} d\txo\\
&\quad+\frac{1}{it}\int_{-\infty}^\infty\bigg\{\int_{-\infty}^\infty e^{-it\tilde{\phi}_1(\tbxi;\bz)}\partial_{\txo}\tilde{\chi}_\delta(\tbxi) d\txt\bigg\}\frac{1}{2\txo}\chi_1\bigg(\frac{\txo}{|t|^{-1/2}}\bigg)d\txo.
\end{aligned}$$
Then, applying the van der Corput lemma \cite{Stein} in the $\txt$-variable integral, we prove that $|\mathbb{I}_{1;(1)}(t;\bz)|\lesssim |t|^{-\frac{5}{6}}$, since $|\txo|\gtrsim |t|^{-1/2}$ in the domain. Therefore, we complete the proof of Proposition \ref{prop: no intersection}.

\section{Nonlinear applications: Proof of Theorem \ref{Thm:nonlinear application}}\label{sec: nonlinear application}

\subsection{Global well-posedness of the nonlinear model}

As a preliminary, we establish the basic global well-posedness of the discrete NLS \eqref{NLS}. 

\begin{proposition}[Global well-posedness of the discrete NLS]\label{prop: GWP of DNLS}
Let $p>1$. For any $\mathbf{u}_0\in \ell_{\mathbf{x}}^2(\mathbf{\Lambda}; \mathbb{C}^2)$, there exists a unique global solution $\mathbf{u}(t)\in C_t(\mathbb{R}; \ell_{\mathbf{x}}^2(\mathbf{\Lambda}; \mathbb{C}^2))$ to the NLS \eqref{NLS} with initial data $\mathbf{u}_0$. Moreover, it obeys the mass conservation law, i.e., for all $t\in\mathbb{R}$, 
\begin{equation}\label{eq: conservation law}
  \|\mathbf{u}(t)\|_{\ell_{\mathbf{x}}^2(\mathbf{\Lambda}; \mathbb{C}^2)}^2 = \|\mathbf{u}_0\|_{\ell_{\mathbf{x}}^2(\mathbf{\Lambda}; \mathbb{C}^2)}^2.
\end{equation}  
\end{proposition}

In general, for discrete models, well-posedness can be proved easily by the standard contraction mapping argument, because the trivial inequality\footnote{Note from the definition \ref{eq: discrete Lebesgue norm} that the $L_{\mathbf{x}}^{r_2}(\mathbf{\Lambda}; \mathbb{C}^2)$-norm is in essence an $\ell^r$-summation norm.}
\begin{align}\label{ineq:embedding}
\|\mathbf{u}\|_{\ell_{\mathbf{x}}^{r_2}(\mathbf{\Lambda}; \mathbb{C}^2)}\lesssim_{r_1, r_2}\|\mathbf{u}\|_{\ell_{\mathbf{x}}^{r_1}(\mathbf{\Lambda}; \mathbb{C}^2)},
\end{align}
with $1\le r_1\le r_2\le\infty$, makes the nonlinear term easier to deal with. Therefore, we only give a sketch of the proof. For more details, we refer to \cite[Proposition~9]{HY2019}. In the following, for convenience, we denote $\ell_{\mathbf{x}}^r=\ell_{\mathbf{x}}^r(\mathbf{\Lambda}; \mathbb{C}^2)$. 

\begin{proof}[Sketch of Proof]
By the Duhamel formula of the equation \eqref{NLS}, it is natural to consider the nonlinear map
$$\Phi( \mathbf{u} )(t)=e^{it\mathbf{\Delta}}\mathbf{u}_0-i \int_0^te^{i(t-s)\mathbf{\Delta}}  \mathcal{N}(\mathbf{u}(s))ds.$$
Indeed, one can show that there exists a small $T>0$, depending only on the size of initial data $\|\mathbf{u}_0\|_{\ell_{\mathbf{x}}^2}$, such that $\Phi(\mathbf{u})$ is contractive in $\{u: \|u\|_{C([-T,T]; L_{\mathbf{x}}^2)}\leq 2\|\mathbf{u}_0\|_{\ell_{\mathbf{x}}^2}\}$, because by \eqref{ineq:embedding}, the nonlinear term can be estimated as 
$$\begin{aligned}
\bigg\|\int_0^te^{i(t-s)\mathbf{\Delta}}  \mathcal{N}(\mathbf{u}(s))ds\bigg\|_{C([-T,T]; \ell_{\mathbf{x}}^2)}&\leq \|\mathcal{N}(\mathbf{u})\|_{L_t^1([-T,T]; \ell_{\mathbf{x}}^2)}\le\|\mathbf{u}\|_{L_t^p([-T,T]; \ell_{\mathbf{x}}^{2p})}^p\\
&\lesssim T\|\mathbf{u}\|_{C([-T,T]; \ell_{\mathbf{x}}^{2})}^p.
\end{aligned}$$
Subsequently, the equation \eqref{NLS} is locally well-posed in $C([-T,T]; \ell_{\mathbf{x}}^2)$. Moreover, for the solution $\mathbf{u}(t)\in C([-T,T]; \ell_{\mathbf{x}}^2)$, differentiating $\|\mathbf{u}(t)\|_{\ell_{\mathbf{x}}^2}^2$ and inserting the equation \eqref{NLS}, one can show the conservation law \eqref{eq: conservation law}. Then, iterating the local well-posedness procedure arbitrarily many times with $\|\mathbf{u}(kT)\|_{\ell_{\mathbf{x}}^2}^2=\|\mathbf{u}_0\|_{\ell_{\mathbf{x}}^2}^2$ for all $k\in\mathbb{N}$, we conclude that $\mathbf{u}(t)$ exists globally in time. 
\end{proof}

\subsection{Proof of Theorem~\ref{Thm:nonlinear application}}

Next, we will show that the solution $\mathbf{u}(t)$, constructed in Proposition \ref{prop: GWP of DNLS}, scatters in time. Here, by time-reversal symmetry, we only consider the positive time direction $t>0$.

For the proof, we employ the dispersion estimate 
\begin{equation}\label{ineq:dispersive estimates}
  \|e^{it\mathbf{\Delta}}\mathbf{u}_0\|_{ \ell_{\mathbf{x}}^r}\leq\frac{c_0}{(1+t)^{\frac43(\frac12-\frac1r)}}
  \|\mathbf{u}_0\|_{ \ell_{\mathbf{x}}^{r'}}
\end{equation}
for $2\leq r<\infty$, which follows from Corollary \ref{cor: dispersion estimate} dropping $\mathbf{O}(-i\nabla_{\mathbf{x}})^*$ by Lemma \ref{lem:boundedness of O}. From now, we fix $p>3$ and assume that $2+\frac{6}{2p-3}<r$. For bootstrapping, we assume the a priori bound
\begin{equation}\label{eq: a priori bound for NLS}
  \sup_{t\in[0,T_{\max})}(1+t)^{\frac43(\frac12-\frac1r)} \|\mathbf{u}(t)\|_{ \ell_{\mathbf{x}}^{r}} \le 2c_0\|\mathbf{u}_0\|_{ \ell_{\mathbf{x}}^{r'}},
\end{equation}
where $T_{\max}\in(0,\infty]$ is the maximal time such that \eqref{eq: a priori bound for NLS} holds. Indeed, by \eqref{ineq:embedding} with $r_2=r\geq r_1=2$ and the conservation law \eqref{eq: conservation law}, the bound $(1+t)^{\frac43(\frac12-\frac1r)} \|\mathbf{u}(t)\|_{ \ell_{\mathbf{x}}^{r}} \le 2c_0\|\mathbf{u}_0\|_{ \ell_{\mathbf{x}}^{r'}}$ holds on $[0,T)$ at least for sufficiently small $T>0$.

For contradiction, we assume that $T_{\max}<\infty$. Then, applying the inequality \eqref{ineq:dispersive estimates} to the Duhamel representation
\begin{equation}\label{NLS duhamel}
\mathbf{u}(t)=e^{it\mathbf{\Delta}}\mathbf{u}_0-i\int_0^te^{i(t-s)\mathbf{\Delta}}  \mathcal{N}(\mathbf{u}(s))ds,
\end{equation}
we obtain that for all $0\leq t<T_{\max}$, 
\begin{equation}\label{eq: first estimate for nonlinear solution}
\begin{aligned}
  \|\mathbf{u}(t)\|_{ \ell_{\mathbf{x}}^{r}} 
  &\leq \frac{c_0\|\mathbf{u}_0\|_{ \ell_{\mathbf{x}}^{r'}}}{(1+t)^{\frac43(\frac12-\frac1r)}}
  +c_0\int_0^{t} \frac{\|\mathcal{N}(\mathbf{u}(s))\|_{ \ell_{\mathbf{x}}^{r'}}}{(1+t-s)^{\frac43(\frac12-\frac1r)}} ds \\ 
  &\leq \frac{c_0\|\mathbf{u}_0\|_{ \ell_{\mathbf{x}}^{r'}}}{(1+t)^{\frac43(\frac12-\frac1r)}}
  +c_0'\int_0^{t} \frac{\|\mathbf{u}(s)\|_{ \ell_{\mathbf{x}}^{r'p}}^p}{(1+t-s)^{\frac43(\frac12-\frac1r)}} ds.
\end{aligned}
\end{equation}
We claim that if the a priori bound \eqref{eq: a priori bound for NLS} holds on the interval $[0,T_{\max})$, then 
\begin{align}\label{ineq: nonlinear decay}
\|\mathbf{u}(s)\|_{\ell_{\mathbf{x}}^{r'p}}^p  
\lesssim_{r,p} \frac{1}{(1+s)^{\sigma_{r,p}}}\|\mathbf{u}_0\|_{ \ell_{\mathbf{x}}^{r'}}^p
\end{align}
for $s\in [0,T_{\max})$, where $\sigma_{r,p}$ is given by \eqref{eq: nonlinear decay rate}.\\
\textit{\underline{Case 1} ($p\ge r-1$)} In this case, we have $r'p\ge r$. Hence, by the embedding \eqref{ineq:embedding} and a priori bound \eqref{eq: a priori bound for NLS}, we have 
$$\|\mathbf{u}(s)\|_{ \ell_{\mathbf{x}}^{r'p}} \lesssim \|\mathbf{u}(s)\big\|_{ \ell_{\mathbf{x}}^{r}}
  \lesssim(1+s)^{-\frac43(\frac12-\frac1r)}\|\mathbf{u}_0\|_{ \ell_{\mathbf{x}}^{r'}}.$$  
Since $\sigma_{r,p}=\frac{4}{3}(\frac12-\frac1r)p$, \eqref{ineq: nonlinear decay} follows.\\
\textit{\underline{Case 2} ($p<r-1$)} Now, we have $r'p<r$ as well as $r'p>r'(\frac72-\frac2r)=r'(\frac{3}{2}+\frac{2}{r'})>2$ by the assumption on $p$ and $r$. Hence, using the interpolation inequality with $0<\theta<1$ such that $\frac{1}{r'p}=\frac{1-\theta}{2}+\frac{\theta}{r}$ ( $\Leftrightarrow\theta=\frac{\frac{1}{2}-\frac{1}{r'p}}{\frac{1}{2}-\frac{1}{r}}$), the mass conservation law and a priori bound \eqref{eq: a priori bound for NLS}, we obtain 
$$\|\mathbf{u}(s)\|_{ \ell_{\mathbf{x}}^{r'p}}^p\lesssim\|\mathbf{u}(s)\|_{ \ell_{\mathbf{x}}^{2}}^{(1-\theta)p}\|\mathbf{u}(s)\|_{ \ell_{\mathbf{x}}^{r}}^{\theta p}\lesssim\|\mathbf{u}_0\|_{ \ell_{\mathbf{x}}^{2}}^{(1-\theta)p}\Bigg(\frac{\|\mathbf{u}_0\|_{ \ell_{\mathbf{x}}^{r'}}}{(1+|s|)^{\frac43(\frac12-\frac1r)}}\Bigg)^{\theta p}.$$
Subsequently, \eqref{ineq: nonlinear decay} follows from the trivial inequality \eqref{ineq:embedding} for $\|\mathbf{u}_0\|_{ \ell_{\mathbf{x}}^{2}}$ with $r'<2$, because $\sigma_{r,p}=\frac{4}{3}(\frac{1}{2}-\frac{1}{r})\theta p=\frac{4}{3}(\frac{1}{2}-\frac{1}{r'p})p$. 

It is important to note that $\sigma_{r,p}>1$ in the inequality \eqref{ineq: nonlinear decay}. Indeed, for this, we assume $p>\max\{\frac{3}{2}+\frac{3}{r-2},\frac72-\frac2r\}$ in Theorem \ref{Thm:nonlinear application}. Hence, applying \eqref{ineq: nonlinear decay} to \eqref{eq: first estimate for nonlinear solution}, we obtain 
$$\begin{aligned}
\|\mathbf{u}(t)\|_{ \ell_{\mathbf{x}}^{r}}&\leq \frac{c_0\|\mathbf{u}_0\|_{ \ell_{\mathbf{x}}^{r'}}}{(1+t)^{\frac43(\frac12-\frac1r)}}+C\int_0^{t} \frac{\|\mathbf{u}_0\|_{ \ell_{\mathbf{x}}^{r'}}^p}{(1+t-s)^{\frac43(\frac12-\frac1r)}(1+s)^{\sigma_{r,p}}} ds\\
&\leq \frac{c_0\|\mathbf{u}_0\|_{ \ell_{\mathbf{x}}^{r'}}}{(1+t)^{\frac43(\frac12-\frac1r)}}+\frac{C\|\mathbf{u}_0\|_{ \ell_{\mathbf{x}}^{r'}}^{p}}{(1+t)^{\frac43(\frac12-\frac1r)}}\leq \frac{\frac{3}{2}c_0\|\mathbf{u}_0\|_{ \ell_{\mathbf{x}}^{r'}}}{(1+t)^{\frac43(\frac12-\frac1r)}}
\end{aligned}$$
for $t\in[0,T_{\max})$, provided that $\|\mathbf{u}_0\|_{ \ell_{\mathbf{x}}^{r'}}$ is small enough. Hence, the above inequality improves the a priori bound \eqref{eq: a priori bound for NLS}. It deduces a contradiction to the maximality of $T_{\max}$. Therefore, we conclude that $T_{\max}=\infty$ in \eqref{eq: a priori bound for NLS} as well as \eqref{nonlinear time decay} and \eqref{ineq: nonlinear decay} hold for all $t>0$.

It remains to show the nonlinear scattering \eqref{eq: nonlinear scattering}. Indeed, for $t_2\ge t_1$, by Duhamel's formula \eqref{NLS duhamel} and \eqref{ineq: nonlinear decay}, we obtain that
$$\begin{aligned}
&\|e^{-it_2\mathbf{\Delta}}\mathbf{u}(t_2)-e^{-it_1\mathbf{\Delta}}\mathbf{u}(t_1)\|_{\ell_{\mathbf{x}}^{2}}\\
&\leq \int_{t_1}^{t_2} \|\mathcal{N}(\mathbf{u}(s))\|_{\ell_{\mathbf{x}}^{2}} ds\lesssim \int_{t_1}^{t_2} \|\mathcal{N}(\mathbf{u}(s))\|_{\ell_{\mathbf{x}}^{r'}} ds
\le \int_{t_1}^{t_2} \|\mathbf{u}(s)\|_{\ell_{\mathbf{x}}^{r'p}}^p ds\lesssim\frac{\|\mathbf{u}_0\|_{ \ell_{\mathbf{x}}^{r'}}^p}{(1+t_1)^{\sigma_{r,p}-1}}.
\end{aligned}$$ 
Thus, we conclude that $e^{-it\mathbf{\Delta}}\mathbf{u}(t)$ has a limit $u_+$ in $L_{\mathbf{x}}^{2}$ as $t\to+\infty$, and 
$$\|e^{-it\mathbf{\Delta}}\mathbf{u}(t)-\mathbf{u}_+\|_{\ell_{\mathbf{x}}^{2}}\lesssim \frac{\|\mathbf{u}_0\|_{ \ell_{\mathbf{x}}^{r'}}^p}{(1+t)^{\sigma_{r,p}-1}}.$$

\appendix

\section{Factorization of the linear Schr\"odinger flow}\label{sec: kernel}

In this appendix, we derive the factorized representation \eqref{eq: flow factorization} of the linear Schr\"odinger flow $e^{it\mathbf{\Delta}}$. To this end, first, we observe that by the Fourier and the inverse Fourier transforms (see  Definition \ref{definition: Fourier transform}), the Laplacian $(-\mathbf{\Delta})$ is the Fourier multiplier with symbol 
$$\mathbf{P}(\mathbf{k})=4\begin{bmatrix}\ \ 3&-\overline{\ph(\mathbf{k})}\\
-\ph(\mathbf{k})&\ \ 3\end{bmatrix},$$
where $\ph=\ph(\mathbf{k})=1+e^{i\mathbf{k}\cdot\mathbf{v}_1}+e^{i\mathbf{k}\cdot\mathbf{v}_2}$ (see \eqref{eq: definition of Laplacian on honeycomb} for the definition of $\mathbf{\Delta}$). The multiplier is a hermitian matrix, and it can be diagonalized as
$$\mathbf{P}(\mathbf{k})=\mathbf{O}(\mathbf{k})\mathbf{D}(\mathbf{k})\mathbf{O}(\mathbf{k})^*,$$
where 
$$\mathbf{D}(\mathbf{k})=4\begin{bmatrix}3+|\ph(\mathbf{k})|&0\\
0&3-|\ph(\mathbf{k})|\end{bmatrix}\quad\textup{and}\quad\mathbf{O}(\mathbf{k})=\begin{bmatrix}\ \frac{1}{\sqrt{2}}&\frac{\overline{\ph(\mathbf{k})}}{\sqrt{2}|\ph(\mathbf{k})|}\\
-\frac{\ph(\mathbf{k})}{\sqrt{2}|\ph(\mathbf{k})|}&\frac{1}{\sqrt{2}}\end{bmatrix}.$$
Thus, by the Fourier transform, the equation \eqref{LS} with initial data $\mathbf{u}_0$ is equivalent to the equation 
$$i\partial_t\big(\mathbf{O}(\mathbf{k})^*\hat{\mathbf{u}}\big)=\mathbf{O}(\mathbf{k})^*i\partial_t\hat{\mathbf{u}}=\mathbf{O}(\mathbf{k})^*\mathbf{P}(\mathbf{k})\hat{\mathbf{u}}=\mathbf{D}(\mathbf{k})\big(\mathbf{O}(\mathbf{k})^*\hat{\mathbf{u}}\big)$$
with initial data $\mathbf{O}(\mathbf{k})^*\hat{\mathbf{u}}_0$, or
$$\big(\mathbf{O}(\mathbf{k})^*\hat{\mathbf{u}}\big)(t)=\begin{bmatrix}e^{-4it(3+|\ph(\mathbf{k})|)}&0\\
0&e^{-4it(3-|\ph(\mathbf{k})|)}\end{bmatrix}\mathbf{O}(\mathbf{k})^*\hat{\mathbf{u}}_0.$$
Then, taking the inverse Fourier transform, we obtain
$$e^{it\mathbf{\Delta}}\mathbf{u}_0=\mathbf{u}(t)=e^{-12it}\mathbf{O}(-i\nabla_{\mathbf{x}})\begin{bmatrix} e^{-4it\varphi(-i\nabla_{\mathbf{x}})}&0\\
0&e^{4it\varphi(-i\nabla_{\mathbf{x}})}\end{bmatrix}\mathbf{O}(-i\nabla_{\mathbf{x}})^*\mathbf{u}_0.$$

\section{Boundedness of the operator \texorpdfstring{$\mathbf{O}(-i\nabla_{\mathbf{x}})$}{}}\label{sec: boundedness of O operator}

The operator $\mathbf{O}(-i\nabla_{\mathbf{x}})$ and its adjoint $\mathbf{O}(-i\nabla_{\mathbf{x}})^*$ appear naturally when the Schr\"odinger flow $e^{it\mathbf{\Delta}}$ is diagonalized (see \eqref{eq: flow factorization}). The following lemma asserts that they are bounded on $\ell_{\mathbf{x}}^p(\mathbf{\Lambda}; \mathbb{C}^2)$.

\begin{lemma}[Boundedness of $\mathbf{O}(-i\nabla_{\mathbf{x}})$]\label{lem:boundedness of O}
For $1<p<\infty$, we have
\begin{equation}\label{boundedness of O}
\| \mathbf{O}(-i\nabla_{\mathbf{x}}) \mathbf{u}\|_{ \ell_{\mathbf{x}}^{p}(\mathbf{\Lambda}; \mathbb{C}^2)}, \| \mathbf{O}(-i\nabla_{\mathbf{x}})^* \mathbf{u}\|_{ \ell_{\mathbf{x}}^{p}(\mathbf{\Lambda}; \mathbb{C}^2)} \lesssim \|  \mathbf{u}\|_{ \ell_{\mathbf{x}}^{p}(\mathbf{\Lambda}; \mathbb{C}^2)}.
\end{equation}
\end{lemma}

\begin{remark}
Our proof does not include the endpoint cases $p=1$ and $p=\infty$.
\end{remark}

\begin{proof}
The strategy is to utilize the Hörmander-Mikhlin theorem  (Lemma~\ref{lem:Hormander-Mikhlin})  for functions in the square lattice domain $\mathbb{Z}^2$. To achieve this, we convert the operator on the lattice $\mathbf{\Lambda}$ into that on the square lattice $\mathbb{Z}^2$ by a simple change of variables.

By duality, it suffices to prove the lemma for
$\mathbf{O}(-i\nabla_{\mathbf{x}})$.
Let $\chi_{\approx\mathbf{K}_\star}$ and
$\chi_{\approx-\mathbf{K}_\star}$ be smooth cut-offs on
$\mathbb{R}^2/\mathbf{\Lambda}^*$ supported in sufficiently small,
disjoint neighborhoods of $\mathbf{K}_\star$ and
$-\mathbf{K}_\star$, respectively, and equal to one near the
corresponding Dirac points. As in
Theorem~\ref{main thm: dispersion estimate}, let
$P_{\approx\mathbf{K}_\star}$ and $P_{\approx-\mathbf{K}_\star}$
be the Fourier multipliers with these symbols.
We decompose
\[
\begin{aligned}
\mathbf{O}(-i\nabla_{\mathbf{x}})\mathbf{u}
&=\mathbf{O}(-i\nabla_{\mathbf{x}})
  P_{\approx\mathbf{K}_\star}\mathbf{u}
 +\mathbf{O}(-i\nabla_{\mathbf{x}})
  P_{\approx-\mathbf{K}_\star}\mathbf{u}\\
&\quad+\mathbf{O}(-i\nabla_{\mathbf{x}})
  \big(1-P_{\approx\mathbf{K}_\star}
        -P_{\approx-\mathbf{K}_\star}\big)\mathbf{u}.
\end{aligned}
\]
The multiplier symbol of the third term is smooth and periodic,
so its convolution kernel is absolutely summable. Hence, this
term is bounded on
$\ell_{\mathbf{x}}^p(\mathbf{\Lambda};\mathbb{C}^2)$.
We only estimate the first term, since the second term can be
treated similarly by translating $-\mathbf{K}_\star$ to the origin.

By the Fourier transform
(see Definition~\ref{definition: Fourier transform}) and
changing the variable by translation, we write
\begin{align*}
\big(\mathbf{O}(-i\nabla_{\mathbf{x}})
P_{\approx\mathbf{K}_\star}\mathbf{u}\big)(\mathbf{x})=
\sum_{\mathbf{y}\in\mathbf{\Lambda}}
\bigg\{\frac{\sqrt{3}}{8\pi^2}
\int_{\mathbb{R}^2/\mathbf{\Lambda}^*}
e^{i(\mathbf{K}_\star+\mathbf{k})\cdot(\mathbf{x}-\mathbf{y})}
\mathbf{O}(\mathbf{K}_\star+\mathbf{k})
\chi_{\approx\mathbf{K}_\star}(\mathbf{K}_\star+\mathbf{k})
\,d\mathbf{k}\bigg\}\mathbf{u}(\mathbf{y}).
\end{align*}
Then, changing the variables by $\mathbf{k}=\mathbf{A}\mathbf{\tilde{k}}$ for $\mathbf{\tilde{k}}=\tilde{k}_1\mathbf{e}_1+\tilde{k}_2\mathbf{e}_2\in \mathbb{R}^2/(2\pi\mathbb{Z}^2)$, where 
$\mathbf{A}=\frac{1}{2\pi} [\mathbf{k}_1 \ \mathbf{k}_2]$ is a $2\times2$ matrix with $\mathbf{k}_1=[\frac{2\pi}{\sqrt{3}} \ 2\pi]^{\textup{T}}$ and $\mathbf{k}_2=[\frac{2\pi}{\sqrt{3}} \ -2\pi]^{\textup{T}}$, we obtain 
\begin{align*}
  &\big(\mathbf{O}(-i\nabla_{\mathbf{x}})P_{\approx\mathbf{K}_\star}\mathbf{u}\big)(\mathbf{x})\\
  &\quad=e^{i\mathbf{K}_\star\cdot\mathbf{x}}  \sum_{\mathbf{y}\in\mathbf{\Lambda}} \bigg\{ \frac{1}{(2\pi)^2} \int_{\mathbb{R}^2/(2\pi\mathbb{Z}^2)} e^{i(\mathbf{A}\mathbf{\tilde{k}})\cdot(\mathbf{x}-\mathbf{y})} \mathbf{O} (\mathbf{K}_{\star}+\mathbf{A}\mathbf{\tilde{k}})\tilde{\chi}(\mathbf{\tilde{k}})d\mathbf{\tilde{k}} \bigg\} e^{-i\mathbf{K}_\star\cdot\mathbf{y}} \mathbf{u}(\mathbf{y}),
\end{align*}
where
\[
\tilde{\chi}(\mathbf{\tilde{k}})
:=\chi_{\approx\mathbf{K}_\star}
  (\mathbf{K}_\star+\mathbf{A}\mathbf{\tilde{k}}).
\]
Next, we introduce the change-of-variable operator $T_{\mathbf{V}}$ defined by 
$$(T_{\mathbf{V}}\mathbf{u})(\mathbf{\tilde{x}}):=e^{-i\mathbf{K}_\star\cdot(\mathbf{V}\mathbf{\tilde{x}})}\mathbf{u}(\mathbf{V}\mathbf{\tilde{x}}),\quad \mathbf{\tilde{x}}=[\tilde{x}_1 \  \tilde{x}_2]^\textup{T}\in\mathbb{Z}^2,$$
where $\mathbf{V}= [\mathbf{v}_1 \ \mathbf{v}_2]$, $\mathbf{v}_1=[\frac{\sqrt{3}}{2} \ \frac{1}{2}]^{\textup{T}}$ and $\mathbf{v}_2=[\frac{\sqrt{3}}{2} \ -\frac{1}{2}]^{\textup{T}}$. Then, substituting $\mathbf{x}=\mathbf{V}\mathbf{\tilde{x}}$ and $\mathbf{y}=\mathbf{V}\mathbf{\tilde{y}}$, we write 
\begin{equation}\label{eq: reformulation of O}
  \big(T_{\mathbf{V}}\mathbf{O}(-i\nabla_{\mathbf{x}})P_{\approx\mathbf{K}_\star}\mathbf{u}\big)(\mathbf{\tilde{x}})=\sum_{\mathbf{\tilde{y}}\in\mathbb{Z}^2} \bigg\{\frac{1}{(2\pi)^2}\int_{\mathbb{R}^2/(2\pi\mathbb{Z}^2)} e^{i\mathbf{\tilde{k}}\cdot(\mathbf{\tilde{x}}-\mathbf{\tilde{y}})}\mathbf{O} (\mathbf{K}_{\star}+\mathbf{A}\mathbf{\tilde{k}}) \tilde{\chi}(\mathbf{\tilde{k}}) d\mathbf{\tilde{k}} \bigg\}(T_\mathbf{V}\mathbf{u})(\mathbf{\tilde{y}}),
\end{equation}
where $(\mathbf{A}^\textup{T})\mathbf{V}=\mathbf{I}$ is used to obtain $e^{i\mathbf{\tilde{k}}\cdot(\mathbf{\tilde{x}}-\mathbf{\tilde{y}})}$ from $e^{i(\mathbf{A}\mathbf{\tilde{k}})\cdot(\mathbf{x}-\mathbf{y})}$.

Now, we recall that on the square lattice $\mathbb{Z}^2$, the Fourier transform $\tilde{\mathcal{F}}$ (resp., its inversion $\tilde{\mathcal{F}}^{-1}$) is given by 
$$(\tilde{\mathcal{F}}\mathbf{u})(\mathbf{\tilde{k}}):=\sum_{\mathbf{\tilde{x}}\in\mathbb{Z}^2}\mathbf{u}(\mathbf{\tilde{x}})e^{-i\mathbf{\tilde{k}}\cdot\mathbf{\tilde{x}}}\quad\bigg(\textup{resp., }(\tilde{\mathcal{F}}^{-1}\mathbf{u})(\mathbf{\tilde{x}}):=\frac{1}{(2\pi)^2}\int_{\mathbb{R}^2/(2\pi\mathbb{Z}^2)} \mathbf{u}(\mathbf{\tilde{k}})e^{i\mathbf{\tilde{k}}\cdot\mathbf{\tilde{x}}}d\mathbf{\tilde{k}}\bigg),$$
and introduce the $\mathbb{Z}^2$-Fourier multiplier $\widetilde{\mathbf{O}}(-i\nabla_{\mathbf{\tilde{x}}})$ by 
$$\big(\tilde{\mathcal{F}}\widetilde{\mathbf{O}}(-i\nabla_{\mathbf{\tilde{x}}})\mathbf{u}\big)(\mathbf{\tilde{k}})=\mathbf{O} (\mathbf{K}_{\star}+\mathbf{A}\mathbf{\tilde{k}})\tilde{\chi}(\mathbf{\tilde{k}}) (\tilde{\mathcal{F}}\mathbf{u})(\mathbf{\tilde{k}}).$$
Then, \eqref{eq: reformulation of O} can be written as 
$$\mathbf{O}(-i\nabla_{\mathbf{x}})P_{\approx\mathbf{K}_\star}\mathbf{u}=(T_{\mathbf{V}})^{-1}\widetilde{\mathbf{O}}(-i\nabla_{\mathbf{\tilde{x}}})T_\mathbf{V}\mathbf{u}.$$
By the definition, $T_{\mathbf{V}}$ is a constant multiple of an isometric isomorphism from $L_{\mathbf{x}}^p(\mathbf{\Lambda};\mathbb{C}^2)$ to $L_{\mathbf{\tilde{x}}}^p(\mathbb{Z}^2;\mathbb{C}^2)$. Hence, it follows that
$$\|\mathbf{O}(-i\nabla_{\mathbf{x}})P_{\approx\mathbf{K}_\star}\|_{\ell_{\mathbf{x}}^p(\mathbf{\Lambda};\mathbb{C}^2)\to \ell_{\mathbf{x}}^p(\mathbf{\Lambda};\mathbb{C}^2)}=\|\widetilde{\mathbf{O}}(-i\nabla_{\mathbf{\tilde{x}}})\|_{\ell_{\mathbf{\tilde{x}}}^p(\mathbb{Z}^2;\mathbb{C}^2)\to \ell_{\mathbf{\tilde{x}}}^p(\mathbb{Z}^2;\mathbb{C}^2)},$$
On the other hand, by a direct computation, we observe that for any multi-index $\alpha$,
\begin{align*}
  \big| \nabla_{\mathbf{\tilde{k}}}^{\alpha}
  \big( \mathbf{O} (\mathbf{K}_{\star}+\mathbf{A}\mathbf{\tilde{k}}) \tilde{\chi}(\mathbf{\tilde{k}})\big) \big| \lesssim |\mathbf{\tilde{k}}|^{-|\alpha|} \;\text { for all } \; \mathbf{\tilde{k}}\in\mathbb{R}^2/(2\pi\mathbb{Z}^2)\setminus\{\mathbf{0}\}.
 \end{align*}
Therefore, the Hörmander-Mikhlin theorem (Lemma~\ref{lem:Hormander-Mikhlin}) implies that $\widetilde{\mathbf{O}}(-i\nabla_{\mathbf{\tilde{x}}})$ is bounded on $\ell_{\mathbf{\tilde{x}}}^p(\mathbb{Z}^2;\mathbb{C}^2)$, which completes the proof.
\end{proof}
\begin{lemma}[Hörmander-Mikhlin: Theorem~4.1 in \cite{HY2019}]\label{thm: Hormander Mikhlin}\label{lem:Hormander-Mikhlin}
Given a symbol function $m$ on $\mathbb{R}^d/(2\pi\mathbb{Z}^d)$, we consider the Fourier multiplier operator $T_m$ defined by
$$\tilde{\mathcal{F}}(T_mf)(\mathbf{\tilde{k}})=m(\mathbf{\tilde{k}})\tilde{\mathcal{F}}{f}(\mathbf{\tilde{k}})\quad\textup{on} \quad\mathbb{R}^d/(2\pi\mathbb{Z}^d),$$
where $f:\mathbb{Z}^d\rightarrow\mathbb{C}$. Suppose that $m$ satisfies 
\begin{align*}
| \nabla^\alpha m (\mathbf{\tilde{k}})| \le c_\alpha|\mathbf{\tilde{k}}|^{-|\alpha|}\quad \textup{on} \quad \mathbb{R}^d/(2\pi\mathbb{Z}^d)\setminus\{0\}
\end{align*} for all multi-indices $|\alpha|\le d+2$. Then, for $1<p<\infty$, we have 
$$\| T_m f\|_{\ell_{\mathbf{\tilde{x}}}^p(\mathbb{Z}^d)} \le C_p \|f\|_{\ell_{\mathbf{\tilde{x}}}^p(\mathbb{Z}^d)}.$$
\end{lemma}

\section{Review of oscillatory integral estimates}\label{Appen: Review of OIE}
In this appendix, we summarize the oscillatory integral theory in two dimensions. Indeed, the propositions and theorems stated below can be extended fully or partially to general dimensions. However, for simplicity of discussion, we shall restrict ourselves to the two-dimensional case, but we also assume that phase functions are analytic.

We consider an oscillatory integral
\begin{equation}\label{general oscillatory integral}
I_{\varphi,\chi}(t):=\iint_{\mathbb{R}^2}e^{it\varphi(\mathbf{k})}\chi(\mathbf{k})d\mathbf{k},    
\end{equation}
where the \textit{phase} $\varphi$ is a non-zero real-analytic function in a small neighborhood and the \textit{amplitude} $\chi$ is a smooth function with sufficiently small support, denoted by $\textup{supp}\chi$.

We say that the phase function $\varphi$ is ``non-stationary" if $\nabla\varphi$ never vanishes in the support of $\chi$. In this case, an arbitrarily fast decay bound can be obtained by integration by parts as long as $\varphi$ is regular enough.

\begin{proposition}[Non-stationary phase method]\label{Prop:Non-stationary}
If $\nabla\varphi\neq\mathbf{0}$ on $\textup{supp}\chi$, then for every $N\ge0$, 
\begin{align*}
|I_{\varphi,\chi}(t)| \le C_N |t|^{-N}.
\end{align*}
\end{proposition}

Next, we consider the case where $\varphi$ has an isolated stationary point $\mathbf{k}_*$, i.e., $(\nabla\varphi)(\mathbf{k}_*)=\mathbf{0}$, in $\textup{supp}\chi$. Here, decomposing the integral by a partition of unity if necessary, we assume that $\textup{supp}\chi$ contains exactly one stationary point. 

We say that $\varphi$ is ``non-degenerate" if $\textup{det}\nabla^2\varphi(\mathbf{k})$ does not vanish in the support of $\chi$.
In this case, by the Morse lemma, we can change the coordinate system to obtain $|t|^{-1}$ optimal decay (see \cite{Wolff} for instance).
\begin{proposition}[Non-degenerate phase method]\label{Lam:Stationary phase method}
If $\textup{det}\nabla^2\varphi  \neq 0$ on $\textup{supp}\chi$, then
\begin{align*}
|I_{\varphi,\chi}(t)| \le C  |t|^{-1}.
\end{align*}
Furthermore, for any $N$, as $t\to\infty$,
\begin{align*}
I_{\varphi,\chi}(t)=2\pi e^{it\varphi(\mathbf{k}_*)}e^{i\pi \frac{\sigma}{4}}|\textup{det}\nabla^2\varphi(\mathbf{k}_*)|^{-\frac12}t^{-1}\Big( \chi(\mathbf{k}_*)+ \sum_{j=1}^Nt^{-j} \mathcal D_j\chi(\mathbf{k}_*) + O(t^{-(N+1)}) \Big),
\end{align*}
where $\sigma$ denotes the signature of $\nabla^2\varphi(\mathbf{k}_*)$ and $\mathcal D_j$ are certain differential operators of order $2j$ with coefficients depending on $\varphi$.
\end{proposition}

Finally, we call $\mathbf{k}_*$ a ``degenerate frequency" if $\textup{det}\nabla^2\varphi(\mathbf{k}_*)=0$.
In such degenerate cases, a slower decay bound is expected. In one dimension, the well-known Van der Corput lemma provides the desired bound.
\begin{proposition}[Van der Corput lemma]\label{Lam:Van der Corput}
Suppose that $\phi$ and $\chi$ are real-valued and smooth on $(a,b)$. Let $m\ge2$, and assume that $|\phi^{(m)}(y)|\ge 1$ for all $y\in (a,b)$. Then,
\begin{align}\label{ineq:Van der Corput lemma}
\bigg| \int_a^b e^{i\lambda \phi(y)} \chi(y)\,dy\bigg|    \le C_m \lambda^{-\frac{1}{m}} \bigg(|\chi(b)|+\int_a^b|\chi'(y)|\,dy \bigg),
\end{align}
where $C_m$ is independent of $a$, $b$, and $\chi$.
\end{proposition}

In two dimensions, the analysis of oscillatory integrals with degenerate phase becomes significantly more intricate due to the complex geometry of the degeneracy. Nevertheless, for real-analytic phase functions, a precise asymptotic expansion of such integrals has been established by Varchenko \cite{Varchenko} . Before stating the result, we introduce the necessary terminology.

By Proposition \ref{Lam:Stationary phase method}, we may be focused on the degenerate case where $\textup{det}\nabla^2\varphi(\mathbf{k}_*)=0$. By factoring out the term $e^{it\varphi(\mathbf{k}_*)}$ and substituting $\mathbf{k}'=\mathbf{k}-\mathbf{k}_*$ in the integral \eqref{general oscillatory integral}, we may assume that $\varphi(\mathbf{0})=0$ and $(\nabla \varphi)(\mathbf{0})=\mathbf{0}$. By this normalization, the phase function admits the Maclaurin series expansion
\begin{equation}\label{S taylor series}
\varphi(\mathbf{k})=\sum_{ \substack{n_1+n_2\geq 2 \\ n_1,n_2\in \mathbb{N}\cup\{0\}}} c_{n_1,n_2} k_1^{n_1} k_2^{n_2}.
\end{equation}

Under the above assumptions, the integral $I_{\varphi,\chi}(t)$ has the asymptotic expansion
\begin{equation}\label{asymptotic expansion}
I_{\varphi,\chi}(t)=\sum_{\{s\}} \sum_{p=0,1} a_{s,p}(\chi)t^s(\ln t)^p,
\end{equation}
where $\{s\}$ denotes a decreasing arithmetic progression of negative rational numbers (see \cite[Theorem 6.3]{AGZV2}) and the coefficients $a_{s,p}(\chi)$'s are considered as generalized functions of amplitudes supported in a small neighborhood of the origin. We remark that $\{s\}$ depends only on the phase $\varphi$, not on the amplitude $\chi$. In lexicographic ordering, let $(s_0, p_0)=(s_0(\varphi), p_0(\varphi))$ be the maximum among all $(s,p)$'s such that $a_{s,p}(\chi)\neq0$ for some $\chi$. Then, we have
$$I_{\varphi,\chi}(t)=a_{s_0,p_0}(\chi)t^{s_0}(\ln t)^{p_0}+o\Big(t^{s_0}(\ln t)^{p_0}\Big).$$
Here, $s_0$ is called the \textit{oscillation index} of $\varphi$ at $\mathbf{0}$, and $p_0$ is called the \textit{multiplicity} of the oscillation of $\varphi$ at $\mathbf{0}$. In this case, it is said that the \textit{type} of the asymptotic of $\varphi$ at $\mathbf{0}$ is the pair $(s_0, p_0)$.

Varchenko \cite{Varchenko} proved that the type of the phase is determined by certain geometric features based on the coefficients of the series. For the statement, we introduce the following definitions. Given a real-analytic phase $\varphi(\mathbf{k})$ of the form \eqref{S taylor series}, the associated \textit{Newton polyhedron} $N(\varphi)$ is defined as the convex hull of the union of $(n_1,n_2)+\mathbb{R}_{\geq 0}^2=\{(y_1,y_2)\in\mathbb{R}^2: y_1\geq n_1\textup{ and } y_2\geq n_2\}$ such that $c_{n_1,n_2}\neq 0$ in the series \eqref{S taylor series}, and its \textit{Newton distance} is defined by
$$d(\varphi)=\inf\Big\{y: (y,y)\in N(\varphi)\Big\}.$$
Note that the integral $I_{\varphi,\chi}(t)$ is invariant under a change of variables, but the Newton distance $d(\varphi)$ depends on the choice of a coordinate system. Thus, we sometimes denote by $d_{\mathbf{k}}(\varphi)$ to emphasize the coordinate dependence. Then, the \textit{height} of $\varphi$ is defined as 
$$h(\varphi)=\sup_{\mathbf{k}}\Big\{d_{\mathbf{k}}(\varphi)\Big\}$$
taking the supremum over all local coordinates. We say that a coordinate system $\mathbf{k}$ is \textit{adapted} if the corresponding Newton distance coincides with the height, $h(\varphi)=d_{\mathbf{k}}(\varphi)$. Using these terminologies, we state the main result of Varchenko \cite{Varchenko}.

\begin{theorem}[Varchenko \textbf{\cite[Theorem 0.6]{Varchenko}}]\label{Varchenko's theorem}
Suppose that the phase $\varphi$ is analytic near $\mathbf{0}$, $\varphi(\mathbf{0})=0$, $(\nabla \varphi)(\mathbf{0})=\mathbf{0}$ and $\textup{det}\nabla^2\varphi(\mathbf{0})=0$. Then, the following hold.
\begin{enumerate}
\item $\varphi$ has an adapted local coordinate system.
\item The oscillation index $s_0(\varphi)$ equals $-\frac{1}{h(\varphi)}$.
\item The multiplicity $p_0(\varphi)$ equals one if the ray $(y,y)$ passes through a vertex of $N(\varphi)$ in some adapted local coordinate system. Otherwise, $p_0(\varphi)$ is zero. 
\item If a smooth amplitude $\chi$ has sufficiently small support near the origin, (2) and (3) together with \eqref{asymptotic expansion} imply that 
\begin{equation}\label{Varchenko inequality}
|I_{\varphi,\chi}(t)|\lesssim |t|^{-\frac{1}{h(\varphi)}}(\ln |t|)^{p_0(\varphi)}.
\end{equation}
The inequality \eqref{Varchenko inequality} is optimal in that if $\chi(\mathbf{0})\neq 0$, then $a_{s_0,p_0}(\chi)\neq 0$ in \eqref{asymptotic expansion}.
\end{enumerate}
\end{theorem}

In applications, in order to obtain an optimal decay estimate using the above theorem, we need to find an adapted coordinate system. The following proposition provides convenient criteria for adaptedness of local coordinates. 

\begin{proposition}[Varchenko \textbf{\cite[Proposition 0.7 and 0.8]{Varchenko}}]\label{prop : adapted coordinate system}
Suppose that the phase $\varphi$ is analytic near $\mathbf{0}$, $\varphi(\mathbf{0})=0$, $(\nabla \varphi)(\mathbf{0})=\mathbf{0}$ and $\textup{det}\nabla^2\varphi(\mathbf{0})=0$ so that $\varphi(\mathbf{k})$ has a power series expansion \eqref{S taylor series} in the $\mathbf{k}$-coordinates. We also assume that $(d(\varphi), d(\varphi))$ lies on a closed compact face $\gamma$ of the Newton polyhedron $N(\varphi)$, and that the face $\gamma$ is on the straight line $a_1s_1+a_2s_2=m$, where $a_1$, $a_2$ and $m$ are natural numbers and $a_1$ and $a_2$ are relatively prime. Then, the coordinate system $\mathbf{k}$ is adapted if either of the following holds:
\begin{enumerate}
\item Both $a_1$ and $a_2$ are strictly greater than one.
\item $a_1=1$, and the polynomial $P(y)=\sum_{{(n_1,n_2)\in\gamma}} c_{n_1,n_2} y^{n_2}$ does not have a real root of multiplicity greater than $\frac{m}{1+a_2}$.
\end{enumerate}

\end{proposition}

Next, we state the main result of Karpushkin \cite{Karpushkin 1}, asserting that the oscillatory integral estimate in Theorem \ref{Varchenko's theorem} is stable under analytic perturbations specified as follows. For a small neighborhood $D$ of the origin in $\mathbb{C}^2$ and a function $f$ on its closure $\bar{D}$, let $f|_{\bar{D}\cap\mathbb{R}^2}$ be the restriction of $f$ on $\bar{D}\cap\mathbb{R}^2$. Then, we define the function space $\textup{H}_D$ as the collection of continuous functions such that $f$ is analytic on $D$ and continuous on $\bar{D}$, and  $f|_{\bar{D}\cap\mathbb{R}^2}$ is real-analytic on $D\cap\mathbb{R}^2$ and continuous on $\bar{D}\cap\mathbb{R}^2$. Here, the function space $\textup{H}_D$ is equipped with the norm $\|f\|_{D}=\max_{(z_1,z_2)\in\bar{D}}|f(z_1,z_2)|$.

\begin{theorem}[Karpushkin \textbf{\cite[Theorem 2.1]{Karpushkin 1}}]\label{Thm:Karpushkin}
Suppose that $\varphi$ is real-analytic near the origin, $\varphi(\mathbf{0})=0$, $(\nabla \varphi)(\mathbf{0})=\mathbf{0}$, $\textup{det}\nabla^2\varphi(\mathbf{0})=0$ and its oscillatory type is $(s_0, p_0)$. 
Then, for any sufficiently small neighborhood $D$ of the origin in
$\mathbb{C}^2$, there exist $\epsilon_D>0$ and a neighborhood $U_D$
of the origin in $\mathbb{R}^2$ such that, for each smooth $\chi$
supported in $U_D$, there is a constant $C_{D,\chi}>0$ satisfying
\[
|I_{\varphi+f,\chi}(t)|
\le C_{D,\chi}|t|^{s_0}(\ln|t|)^{p_0}
\]
for all $|t|\ge2$ and all $f\in\textup{H}_D$ with
$\|f\|_D\le\epsilon_D$.
\end{theorem}

\section{ Alternative proofs of Proposition~\ref{prop: two curve intersection} and Proposition~\ref{prop: no intersection}}\label{sec: alternative proof of K1 and K2}
We provide alternate proofs of Proposition~\ref{prop: two curve intersection} and Proposition~\ref{prop: no intersection} using the theorems in Appendix~\ref{Appen: Review of OIE}.

\subsection{  Alternative proof of Proposition~\ref{prop: two curve intersection}}\label{sec: alternative proof of Prop 6.4}
We recall from \eqref{eq: I2(t,z) definition} that 
$$\mathbb{I}_2(t;\bz):=\iint_{\R^2}e^{-it(\phi_2(\bxi)-\bxi\cdot \bz)}\chi_0\bigg(\frac{|\bxi|}{\delta}\bigg)d\bxi,$$
where $\phi_2$ is analytic at $\mathbf{0}$ with the expansion
$$\phi_2(\bxi)=\pm \xi_1^3+\xi_2^3+\mathcal{O}_4(\bxi).$$
If $|\bz|$ is not small enough, then by choosing small $\delta>0$, one can use the non-stationary phase method (Proposition~\ref{Prop:Non-stationary}) to obtain a faster decay bound. Therefore, we may assume that $|\bz|$ is sufficiently small. However, in this case, the upper bound for $\mathbb{I}_2(t;\bz)$ is stable under small linear perturbations of the phase function (Theorem~\ref{Thm:Karpushkin}). Therefore, assuming $\bz=\mathbf{0}$, it suffices to prove that
\begin{equation}\label{eq: I2 bound at zero}
|\mathbb{I}_2(t;\mathbf{0})| \lesssim |t|^{-\frac{2}{3}}.
\end{equation}

The Newton polyhedron associated with $\phi_2$ is defined as the convex hull of  
$$ \Big((3,0)+\mathbb{R}_{\geq 0}^2\Big) \cup \Big((0,3)+\mathbb{R}_{\geq 0}^2\Big) \cup
\Big(\bigcup_{n_1+n_2\ge 4} (n_1,n_2)+\mathbb{R}_{\geq 0}^2\Big). $$
See Figure~\ref{fig: Newton’s polyhedron}. We observe that it has only one compact face, namely the line segment joining $(3,0)$ and $(0,3)$, which lies on the line $s_1+s_2=3$ and also contains the point $(d(\phi_2), d(\phi_2))$.
We also note that the coordinate system $\{\xi_1,\xi_2\}$ is adapted by part (2) of Proposition~\ref{prop : adapted coordinate system}, because the associated polynomial is
$$P(y) = \left[ \pm \xi_1^3 + \xi_2^3 \right]_{\xi_1 = 1, \xi_2 = y} = \pm 1 + y^3,$$
which has a unique real root of multiplicity one, while $\frac{m}{1 + a_2} = \frac{3}{2}$. Therefore, by Theorem~\ref{Varchenko's theorem}, we obtain the desired decay bound \eqref{eq: I2 bound at zero}.

\subsection{  Alternative proof of Proposition~\ref{prop: no intersection}}\label{sec: alternative proof of Prop 7.4}

We recall from \eqref{eq: I1 integral} that 
$$\mathbb{I}_1(t;\bz):=\iint_{\R^2} e^{-it(\phi_1(\bxi)-\bz\cdot\bxi)}\chi_0\bigg(\frac{|\bxi|}{\delta}\bigg)d\bxi,$$
where $\phi_1$ is analytic at $\mathbf{0}$ with the expansion 
$$\phi_1(\bxi)=\xo^{\,2}+\sum_{m_1+m_2\geq 3}c_{m_1, m_2}\xo^{m_1}\xt^{m_2},\quad  c_{0,3}=\pm1.$$
Again, by stability (Theorem~\ref{Thm:Karpushkin}), we may assume that $\bz=\mathbf{0}$. Now, we observe that the Newton polyhedron associated with $\phi_1$ has only one compact face, which also contains the point
$(d(\phi_1), d(\phi_1))$ and lies on the line $3s_1 + 2s_2 = 6$ (see figure~\ref{fig: Newton’s polyhedron}). Since 2 and 3 are relatively prime, the coordinate system $\{\xi_1, \xi_2\}$ is adapted by part (1) of Proposition~\ref{prop : adapted coordinate system}. The Newton distance is $d(\phi_1) = \frac{6}{5}$, and since the coordinates are adapted, the height is $h(\phi_1) = \frac{6}{5}$. Therefore, we conclude that 
$$ |\mathbb{I}_1(t;\mathbf{0})| \lesssim |t|^{-\frac56}.$$

\color{black}

\begin{figure}
\begin{center}
\adjustbox{center}{
\begin{tikzpicture}[scale=1.5]
\begin{scope}

    \draw[->] (-0.3,0) -- (3.5,0) node[right] {$s_1$};
    \draw[->] (0,-0.3) -- (0,3.5) node[left] {$s_2$};

    \path[fill=blue!10, even odd rule, opacity=0.5]
       (0,0) rectangle (3.5,3.5)
       (0,3) -- (0,0) -- (2,0) -- cycle;

    \node at (1.5,2.3) {$3s_1+ 2s_2=6$};
    \node[text=blue] at (2.5,2.8) {$N(\phi_1)$};

    \draw[blue, thick] (0,3.5) -- (0,3) -- (2,0) -- (3.5,0);

    \filldraw[blue] (0,3) circle(0.03) node[left] {$(0,3)$};
    \filldraw[blue] (2,0) circle(0.03) node[below] {$(2,0)$};

    \foreach \y in {0,1,2}
        \draw (0,\y) circle(0.03);

    \foreach \x in {0,1}
        \draw (\x,0) circle(0.03);

    \draw (1,1) circle(0.03);

    \filldraw[red] (6/5,6/5) circle(0.03) node[right] {$\left(\frac{6}{5}, \frac{6}{5}\right) = (d(\phi_1), d(\phi_1)) $};

    \draw[red, dashed] (6/5,6/5) -- (6/5,0);
    \draw[red, dashed] (6/5,6/5) -- (0,0);
    \draw[red, dashed] (0,0) .. controls (0,-0.5) and (6/5,-0.5) .. (6/5,0)
     node[midway, red,fill=white, inner sep=1pt,] {$h(\phi_1)=\frac{6}{5}$};

\end{scope}

\begin{scope}[xshift=5cm]

    \draw[->] (-0.3,0) -- (3.5,0) node[right] {$s_1$};
    \draw[->] (0,-0.3) -- (0,3.5) node[left] {$s_2$};

    \path[fill=blue!10, even odd rule, opacity=0.5]
       (0,0) rectangle (3.5,3.5)
       (0,3) -- (0,0) -- (3,0) -- cycle;

   \draw[blue, thick] (0,3.5) -- (0,3) -- (3,0) -- (3.5,0);

    \filldraw[blue] (0,3) circle(0.03) node[left] {$(0,3)$};
    \filldraw[blue] (3,0) circle(0.03) node[below] {$(3,0)$};

    \foreach \y in {0,1,2}
        \draw (0,\y) circle(0.03);

    \foreach \x in {0,1,2}
        \draw (\x,0) circle(0.03);

    \draw (1,1) circle(0.03);

    \filldraw[red] (3/2,3/2) circle(0.03) node[right] {$\left(\frac{3}{2}, \frac{3}{2}\right) = (d(\phi_2), d(\phi_2)) $};

    \node at (1.5,2.3) {$s_1+s_2=3$};
    \node[text=blue] at (2.5,2.8) {$N(\phi_2)$};

    \draw[red, dashed] (3/2,3/2) -- (3/2,0);
    \draw[red, dashed] (3/2,3/2) -- (0,0);
    \draw[red, dashed] (0,0) .. controls (0,-0.5) and (3/2,-0.5) .. (3/2,0)
     node[midway, red,fill=white, inner sep=1pt,] {$h(\phi_2)=\frac{3}{2}$};

\end{scope}

\end{tikzpicture}
}
\end{center}  
\caption{Newton’s polyhedron and Newton’s distance of the Taylor series corresponding to $\phi_1$ and $\phi_2$  in Section~\ref{sec: alternative proof of Prop 6.4} and \ref{sec: alternative proof of Prop 7.4} }
    \label{fig: Newton’s polyhedron}
\end{figure}

\end{document}